\pdfoutput=1
\documentclass[a4paper,final,fleqn]{article}
	
\usepackage[lm,grey]{janm}
\usepackage{math}

\usepackage[numbers]{natbib}
\def\@biblabel#1{#1.}

\newcommand{\loc}{\mathrm{loc}}

\newcommand{\uno}{Id}

\newcommand{\ld}{\text{\tiny\ensuremath{\bullet}}}
\newcommand{\dDl}{\Dl^{\ld}}
\newcommand{\dir}{\mathrm{dir}}
\newcommand{\per}{\mathrm{per}}
\newcommand{\spi}[1]{\lin{#1}_{\Ic}}
\newcommand{\spii}[1]{\lin{#1}}

\newcommand{\omk}{\om^{\kdv}}	
\newcommand{\omr}{\om^{*}}		

\newcommand{\Hkc}{\Hc^{\kdv}}

\newcommand{\Hk}{H^{\kdv}}
\newcommand{\Hn}{H^{*}}

\newcommand{\Wp}{\Ws}
\newcommand{\w}[1]{\lin{#1}}
\renewcommand{\th}{\theta}
\DeclareMathOperator{\arccosh}{arccosh}
\DeclareMathOperator{\Iso}{Iso}
\DeclareMathOperator{\spec}{spec}

\let\Id\undefined
\let\uno\undefined
\DeclareMathOperator{\Id}{Id}
\DeclareMathOperator{\uno}{Id}
\DeclareMathOperator{\Null}{Null}

\newcommand{\spanx}{\mathrm{span}}

\renewcommand{\kdv}{\mathrm{kdv}}
\newcommand{\KdV}{\text{KdV}\xspace}

\numberwithin{equation}{section}

\title{On the convexity of the KdV Hamiltonian}
\author{T. Kappeler \footnote{Supported in part by the Swiss National Science Foundation},
        A. Maspero  \footnote{Supported in part by the Swiss National Science Foundation},
        J.C. Molnar \footnote{Supported in part by the Swiss National Science Foundation},
        P. Topalov  \footnote{Supported in part by NSF DMS-0901443.}}
\date{\today}

\begin{document}

\maketitle


\begin{abstract}
We prove that the nonlinear part $H^{*}$ of the KdV Hamiltonian $H^{kdv}$, when expressed in action variables $I = (I_{n})_{n\ge 1}$, extends to a real analytic function on the positive quadrant $\ell^2_+(\N)$ of $\ell^{2}(\N)$ and is strictly concave near $0$. As a consequence, the differential of $H^{*}$ defines a local diffeomorphism near $0$ of $\ell_{\C}^{2}(\N)$.


Furthermore we prove that the Fourier-Lebesgue spaces $\Ls^{s,p}$ with $-1/2 \le s \le 0$ and $2 \le p < \infty$, admit global KdV-Birkhoff coordinates.
In particular, it means that $\ell^2_+(\N)$ is the space of action variables of the underlying phase space $\Ls^{-1/2,4}$ and that the KdV equation is globally in time $C^0$-well-posed on $\Ls^{-1/2,4}$.


%

\end{abstract}

\section{Introduction}
\label{section1}

The Korteweg-de Vries (KdV) equation
\begin{align}
\label{kdv.1}
   & \partial_t u  = -\partial_x^3 u + 6 u \partial_x u 
\end{align}
on the circle $\T = \R/\Z$ is well known to be a Hamiltonian PDE. When considered on the Sobolev spaces $H^{s} \equiv H^{s}(\T,\R)$, $s\ge 1$, it can be written in the form
\[
  \partial_{t}u = \partial_{x}\partial_{u}\Hkc,
\]
where $\partial_{u}\Hkc$ denotes the $L^{2}$-gradient of the \KdV-Hamiltonian
\[
  \Hkc(u) \defl \int_0^1 \left(\frac{1}{2} \left(\partial_x u \right)^2 + u^3 \right) \, \dx
\]
and $\partial_{x}\partial_{u}\Hkc$ is the Hamiltonian vector field defined by the Poisson bracket
\begin{equation}
  \label{pbr}
  \pbr{F, G} \defl \int_0^1 \partial_u F \, \partial_x \partial_u G \, \dx.
\end{equation}
The main purpose of this paper is to study convexity properties of the KdV Hamiltonian. Note that the mean $u_{0} \defl \int_{0}^{1} u(x)\,\dx$ is a Casimir of the bracket~\eqref{pbr} and that the level sets
\[
  H_{c}^{s} \defl \setdef{u\in H^{s}}{u_{0} = c},\qquad c\in \R,
\]
are symplectic leaves. We concentrate on the leaf $H_{0}^{s}$ only as our results can be easily extended to any other leaf. To state them we first need to describe in some detail the property of the \KdV equation of being an integrable PDE. For any $s\in\R$ and $1\le p < \infty$, let
\[
  \ell^{s,p}_{0,\C} \defl \setdef{z=(z_{n})_{n\in\Z}\subset\C}{z_{0} = 0,\quad \n{z}_{s,p} < \infty},
\]
where
\begin{equation}
  \n{z}_{s,p} \defl \p*{ \sum_{n\in\Z} \lin{n}^{ps}\abs{z_{n}}^{p} }^{1/p},\qquad 
  \lin{n} \defl 1 + \abs{n},
\end{equation}
and denote by $\ell^{s,p}_{0}$ the real subspace $\setdef{(z_{n})_{n\in\Z}\in \ell_{0,\C}^{s,p}}{z_{-n} = \ob{z_{n}} \forall n\ge 1}$ of $\ell_{0,\C}^{s,p}$. The spaces $\ell^{s,p}(\N) = \ell^{s,p}(\N,\R)$ and $\ell^{s,p}_{\C}(\N) = \ell^{s,p}(\N,\C)$ are defined in an analogous way. Note that the Sobolev spaces $H_{0}^{s}$ can then be described by
\[
  H_{0}^{s} = \setdef{u\in \Sc_{\C}'}{(u_{n})_{n\in\Z}\in \ell_{0}^{s,2}},
\]
where $u_{n} = \spi{u,\e^{\ii 2n\pi x}}$, $n\in\Z$, and $\spi{\cdot,\cdot}$ denotes the $L^{2}$-inner product on $L^{2}(\T,\C)$, $\spi{f,g} = \int_{0}^{1} f(x)\ob{g(x)}\,\dx$ extended by duality to a pairing of the Schwartz space $\Sc_{\C} = \setdef{f\in C^{\infty}(\R,\C)}{f(x+1) = f(x)\forall x\in\R}$ and its dual $\Sc_{\C}'$. It is shown in \cite{Kappeler:2003up,Kappeler:2005fb} that on $H_{0}^{-1}$ and by restriction, on $H_{0}^{s}$, $s > -1$, the \KdV equation admits canonical real analytic coordinates $x_{n}$, $y_{n}$, $n\ge 1$, so that the initial value problem of the \KdV equation, when expressed in these coordinates, can be solved by quadrature (cf. \cite{Kappeler:2006fr}). The coordinates $x_{n}$, $y_{n}$ are canonical in the sense that $\pbr{x_{n},y_{n}} = \int_{0}^{1} \partial_{u}x_{n}\partial_{x}\partial_{u}y_{n}\,\dx = 1$ for any $n\ge 1$ whereas all other brackets between coordinate functions vanish. For our purposes it is convenient to introduce the complex coordinates
\[
  z_{n}  = \frac{x_{n}-\ii y_{n}}{\sqrt{2}},\quad
  z_{-n} = \frac{x_{n}+\ii y_{n}}{\sqrt{2}},\quad 
  \forall n\ge 1,\qquad
  z_{0} = 0,
\]
referred to as (complex) Birkhoff coordinates.

In more detail, the above result says the following: there exists a complex neighborhood $\Wp$ of $H_{0}^{-1}$ in $H_{0}^{-1}(\T,\C)$ so that $\Phi(u) \defl (z_{n}(u))_{n\in\Z}$ is an analytic map from $\Ws$ to $\ell_{0,\C}^{-1/2,2}$ with $\Phi(0) = 0$ satisfying the property that for any $s\ge -1$, the restriction of $\Phi$ to $H_{0}^{s}$ is a canonical real analytic diffeomorphism onto $\ell_{0}^{s+1/2,2}$. The actions $I = (I_{n})_{n\ge 1}$ defined by $I_{n} = z_{n}z_{-n} \ge 0$, $n\ge 1$, then take values on the positive quadrant $\ell_{+}^{2s+1,1}(\N)$ of $\ell^{2s+1,1}(\N,\R)$,
\[
  \ell_{+}^{2s+1,1}(\N) = \setdef{(r_{n})_{n\ge 1}\in \ell^{2s+1,1}(\N,\R)}{r_{n}\ge 0 \forall n\ge 1}.
\]
In addition, $\Phi$ has the property that the \KdV Hamiltonian $\Hkc\colon H_{0}^{1}\to \R$, when expressed in the coordinates $(z_{n})_{n\in\Z}$, is a function of the actions $(I_{n})_{n\ge 1}$ alone. By \cite{Kappeler:2003up} this function, denoted by $\Hk$, is a real analytic map, $\Hk\colon \ell^{3,1}_{+}(\N)\to \R$ and according to \cite[Corollary 15.2]{Kappeler:2003up} admits an expansion at $I=0$ of the form
\begin{equation}
  \label{exp-H-kdv}
  \sum_{n\ge 1} 8n^{3}\pi^{3}I_{n} - 3\sum_{n\ge 1} I_{n}^{2} + \dotsb
\end{equation}
where the dots stand for higher order terms in $I$. In particular, at $I=0$
\begin{equation}
  \label{diff.H.0}
  \left.\partial_{I_n} \Hk\right\vert_{I=0} = 8 \pi^3 n^3, \qquad
  \left.\partial_{I_n}\partial_{I_m} \Hk\right\vert_{I=0} = - 6 \dl_{nm}, \quad \forall n,m \ge 1.
\end{equation}
See Section~\ref{section2} for further details on these results.
The main aim of this paper is to prove a conjecture of Korotyaev and Kuksin~\cite{Korotyaev:2011tw}, motivated by the perturbation theory of the KdV equation. It says that the nonlinear part of the KdV Hamiltonian
\begin{equation}
  \Hn(I) \defl \Hk(I) - \sum_{n \ge 1} 8 \pi^3 n^3 I_n
\end{equation}
extends analytically to a function on $\ell^2_+(\N)\equiv \ell^{0,2}_+(\N)$ and that it is strictly concave near $0$. Note that the KdV frequencies, defined as $\omk_n = \partial_{I_n} \Hk$, differ from $\omr_n \defl \partial_{I_n} \Hn$ by a constant,
\[
  \omk_n(I) = 8 \pi^3 n^3 + \omr_n(I) \qquad \forall n \ge 1, \ \forall I \in \ell^{3,1}_+(\N).
\]
Furthermore, we recall that in \cite{Kappeler:2006fr} it was shown that the KdV frequencies $(\omk_{n})_{n\ge 1}$ analytically extend to a neighborhood $V_{0}$ of $\ell_{+}^{-1,1}(\N)$ in $\ell_{\C}^{-1,1}(\N)$ so that 
\begin{equation}
  \label{omn-asymptotics}
  \omk_{n}-8n^{3}\pi^{3} = O(1) \text{ as } n\to\infty \text{ locally uniformly on } V_{0}.
\end{equation}

\begin{thm}
\label{main} 
The nonlinear part $\Hn$ of the KdV Hamiltonian admits a real analytic extension to a complex neighborhood $V \subset \ell^2_\C(\N)$ of the positive quadrant $\ell^2_+(\N)$ of $\ell^2(\N)$. The map of the renormalized frequencies $\omr = \partial_{I} \Hn$, restricted to $V_{0}\cap V$ equals
\[
  V\cap V_{0} \to \ell^2_\C(\N) \ , \quad 
  I \mapsto \left( \omk_{n}(I)- 8n^{3}\pi^{3} \right)_{n \ge 1}.
\]
It vanishes at $I=0$ and its differential $\ddd_0 \omr$ at $I=0$ is given by $-6 \uno_{\ell_{\C}^2(\N)}$. 
In particular near $I=0$ in $\ell_{\C}^{2}(\N)$, $\omr$ defines a local diffeomorphism and $\Hn$ is strictly concave in the sense that for $I$ in a (sufficiently small) neighborhood of $\, 0$ in $\ell^2_+(\N)$
\begin{equation}
  \label{convexity} 
  \ddd^2_I \Hn(J, J) \le - \sum_{n \ge 1} J_n^2 \qquad \forall J
   = \left( J_n \right)_{n \ge 1} \in \ell^2(\N).
\end{equation}
\end{thm}

\begin{rem}
Due to a mismatch of norms, it is difficult to put identity \eqref{diff.H.0} in the setup with the space $\ell^{3,1}_+(\N)$ to use for studying perturbations of the KdV equation. 
The space $\ell_{+}^{2}(\N)$ is the appropriate one to express the strict concavity of $\Hn$ -- see \cite{Guan:2014jc}. In combination with results from~\cite{Korotyaev:2011tw} Theorem~\ref{main} implies that $\ddd_{I}^{2}\Hn \le 0$ on all of $\ell^{2}_{+}(\N)$.
\end{rem}

The second main result of this paper says that the Birkhoff coordinates can be analytically extended to a class of Fourier-Lebesgue spaces. To state our result precisely, denote for $s \in \R$ and $1 \le p < \infty$ by $\Ls_{0,\C}^{s,p} \subset \Sc_\C'$ the Fourier-Lebesgue space of distributions on $\T$ given by
\[
  \Ls_{0,\C}^{s,p} \defl \setdef{ u \in \Sc_\C' }{ (u_n)_{n \in \Z} \in \ell^{s,p}_{0,\C} } 
\]
and by $\Ls_{0}^{s,p} \subset \Ls_{0,\C}^{s,p}$ the real subspace $\Ls_{0}^{s,p} = \setdef{ u \in \Ls_{0,\C}^{s,p} }{ u_{-n} = \ob{u_n} \ \forall n \ge 1 }$. Note that $\Ls_0^{s,2} = H_0^s$ and by the Sobolev embedding theorem, $\Ls_{0,\C}^{s,p} \hookrightarrow \Ls_{0,\C}^{-1,2}$ for any $-1/2 \le s \le 0$ and $1 \le p < \infty$. 

\begin{thm}
\label{main2} 
For any $-1/2 \le s \le 0$ and $2 \le p < \infty$, the restriction of $\Phi$ to $\Ls_{0}^{s,p}$ takes values in $\ell_{0}^{s + 1/2,p}$ and $\Phi\colon \Ls_{0}^{s,p} \to \ell_{0}^{s+1/2,p}$ is a real analytic diffeomorphism, and therefore provides global Birkhoff coordinates on $\Ls_{0}^{s,p}$.
\end{thm}

\begin{rem}
Note that for any $u \in \Ls_0^{-1/2, 4}$, $\Phi(u) = (z_n)_{n \in \Z} \in \ell_0^{0,4}$ and hence the corresponding sequence of action variables $(I_n = z_n z_{-n})_{n \ge 1}$ is in $\ell_+^2(\N)$. Thus the space of action variables $\ell_+^2(\N)$ considered in Theorem \ref{main} stems from the underlying phase space $\Ls_0^{-1/2, 4}$ which admits global Birkhoff coordinates.

Although our principal interest lies in the space $\Ls_{0}^{-1/2,4}$, Theorem~\ref{main2} is stated for the family of spaces $\Ls_{0}^{s,p}$ with $-1/2\le s \le 0$, $2\le p < \infty$ since the proof of the more general result does not require any additional effort. With some more work, Theorem~\ref{main2} can be proved to hold for a larger family of spaces $\Ls_{0}^{s,p}$.
For any complex Borel measure $\mu$ on $\T$, the distribution $\mu- \mu(\T)$ is an example of an element of $\Ls_{0,\C}^{s,p}$ for any pair $-1/2\le s < 0$, $2\le p < \infty$ satisfying $\abs{s}p >1$. In particular, $\mu-\mu(\T)\in \Ls_{0,\C}^{-1/2,4}$.
\end{rem}

Theorem \ref{main2} has an immediate application to the initial value problem of the KdV equation. As already mentioned above it was shown in \cite{Kappeler:2006fr} that the KdV frequencies $(\omk_n(I))_{n\ge 1}$ analytically extend to a complex neighborhood of the positive quadrant $\ell_{+}^{-1,1}(\N)$, satisfying \eqref{omn-asymptotics}. In Birkhoff coordinates on $H_0^{-1}$, the KdV equation \eqref{kdv.1} takes the form
\[
  \dot z_n = \ii \omk_n z_n \ , \quad \dot z_{-n} = -\ii \omk_n z_{-n} \qquad \forall n \ge 1.
\]
We refer to \cite{Kappeler:2006fr} for a discussion of the notion of solutions of the initial value problem of KdV in spaces of distributions such as $H_0^{-1}$ and the one of $C^0$-well-posedness. Theorem \ref{main2} leads to the following extension of the results of \cite{Kappeler:2006fr}:

\begin{thm}
\label{main3} 
The initial value problem of the Korteweg-de Vries equation \eqref{kdv.1} is globally in time $C^0$-well-posed on the Fourier-Lebesgue spaces $\Ls_0^{s,p}$ with $-1/2 \le s \le 0$ and $2 \le p < \infty$.
\end{thm}

Other applications concern the characterization of potentials in $H_{0}^{-1}$ which are in $\Ls_{0}^{s,p}$ for $-1/2\le s\le 0$, $2\le p < \infty$, in terms of the decay of the sequence of gap lengths of the spectrum as well as a spectral description of isospectral sets in $\Ls_{0}^{s,p}$ -- see Corollary~\ref{inv-gap} and Corollary~\ref{iso-q}.

\vspace{1em}

\noindent {\em Method of proof:} Inspired from the work in \cite{Korotyaev:2011tw} we derive a formula for the nonlinear part of the KdV Hamiltonian $\Hkc(u)$ for finite gap potentials,      using the asymptotics of the discriminant $\Dl(\lm,u)$ of the operator $-\partial_x^2 + u$ at $\lm = \infty$ and the residue calculus. 
It leads to a formula for $\Hkc(u) - \sum_{n=1}^{\infty} 8 \pi^3 n^3 I_n$ for smooth potentials in terms of the weighted sum $\sum_{n \ge 1} (8n\pi)\Rc_n$ of functionals $\Rc_n$, each of which is an expression involving only the discriminant $\Dl(\lm,u)$ -- or equivalently the periodic spectrum of $-\partial_x^2 + u$. Using that by \cite{Kappeler:2005fb}, $\Dl(\lm,u)$ can be analytically extended to $H_0^{-1}$ one shows that the same holds true for each functional $\Rc_n$. 
The asymptotics of the periodic eigenvalues of $-\partial_x^2 + u$, related to Theorem \ref{main2}, then imply that $\sum_{n \ge 1} (8n\pi)\Rc_n$ converges absolutely for $u$ in a complex neighborhood of $\Ls_0^{-1/2,4}$, yielding the claimed statement for $\Hn$. 
The starting point of the proof of Theorem \ref{main2} are the Birkhoff coordinates constructed in \cite{Kappeler:2005fb} on $H_0^{-1}$. The main ingredients are  asymptotic estimates of the periodic and Dirichlet eigenvalues in the case where $u$ is in $\Ls^{s,p}_{0,\C}$, $-1/2 \le s \le 0$, $2 \le p < \infty$.

\vspace{1em}

\noindent{\em Related works:} Using the conformal mapping theory together with the heights associated with the periodic eigenvalues of $-\partial_x^2 + u$, Korotyaev and Kuksin proved in \cite{Korotyaev:2011tw} that $\Hn$ can be continuously extended to $\ell_+^2(\N)$, leaving open the question if the KdV frequencies, which are the dynamically relevant quantities, can be extended to $\ell_+^2(\N)$ as well. Using different techniques, we succeed in answering this question. Actually we prove much more. Theorem \ref{main} and Theorem \ref{main2} together show that $\Hn$ and hence the KdV frequencies can be indeed extended analytically to $\ell_+^2(\N)$ and that $\ell_+^2(\N)$ is the space of action variables with underlying phase space $\Ls^{-1/2,4}_0$.

Using results on the \KdV frequencies established in \cite{Bikbaev:1993jl}, it was also shown in \cite{Korotyaev:2011tw} that $\Hn$ is concave on all of $\ell_{+}^{2}(\N)$ and hence Theorem~\ref{main} implies that $\ddd_{I}^{2}\Hn\le 0$ on $\ell_{+}^{2}(\N)$ as well. The problem whether $\Hn$ is \emph{strictly} concave on \emph{all} of $\ell_{+}^{2}(\N)$ is still open -- see \cite{Bikbaev:1993jl,Bobenko:1991fv,Kappeler:2003up,Krichever:1992uq} as well as the paper \cite{Henrici:2009ek}, concerned with the study of the frequencies of the Toda lattice.

Finally, we note that the initial value problem of the \KdV equation has been intensively studied -- see \cite{DispWiki} for a comprehensive account.

\vspace{1em}

\noindent{\em Organization of the paper:} In Section~\ref{section2} we introduce additional notation and review known results used throughout the paper. In Section~\ref{section3} we prove asymptotic estimates for the periodic and Dirichlet eigenvalues of $-\partial_x^2 + u$ for $u \in \Ls_{0,\C}^{s,p}$ with $-1/2\le s\le 0$ and $2 \le p < \infty$. They are used in the proofs of Theorem \ref{main} and Theorem \ref{main2} provided in Section~\ref{section4} and Section~\ref{section5} respectively.


\section{Preliminaries}
\label{section2}

In \cite{Kappeler:2005fb} it is shown that the Birkhoff map of the \KdV equation on $\T=\R/\Z$, defined and studied in detail in \cite{Kappeler:2003up}, can be analytically extended to a real analytic diffeomorphism
\[
  \Phi\colon H_{0}^{-1}\to \ell^{-1/2,2}.
\]
A key ingredient of the construction of the Birkhoff map and its extension to $H_{0}^{-1}$ are spectral quantities of the differential operator
\begin{equation}
  \label{Lq}
  L(q) = -\partial_{x}^{2} + q,
\end{equation}
which appears in the Lax pair formulation of the \KdV equation. In this section we review properties of $\Phi$ as well as  spectral properties of $L(q)$ used throughout the paper.

Let $q$ be a \emph{complex potential} in  $H^{-1}_{0,\C} \defl H_{0}^{-1}(\R/\Z,\C)$. In order to treat periodic and antiperiodic boundary conditions at the same time, we consider the differential  operator $L(q) = -\partial_{x}^{2} + q$, on $H^{-1}(\R/2\Z,\C)$ with domain of definition $H^{1}(\R/2\Z,\C)$. See Appendix~\ref{a:hill-op} for a more detailed discussion.
The spectral theory of $L(q)$, while classical for $q \in L_{0,\C}^{2}$, has been only fairly recently extended to the case  $q \in H^{-1}_{0,\C}$ -- see e.g. \cite{Djakov:2009fx,Kappeler:2001bi,Kappeler:2003vh,Korotyaev:2003gp,Savchuk:2003vl}
 and the references therein.
The spectrum of $L(q)$,  called the \emph{periodic spectrum of $q$} and denoted by $\spec L(q)$,  is discrete and the eigenvalues, when counted with their multiplicities and ordered lexicographically -- first by their real part and second by their imaginary part -- satisfy
\begin{align}
\label{asympt.eig.0}
  \lm_{0}^{+}(q) \lex \lm_{1}^{-}(q) \lex \lm_{1}^{+}(q) \lex \dotsb
          \ ,
  \qquad
  \lm_{n}^{\pm}(q) = n^{2}\pi^{2} + n\ell^2_n.
\end{align}
Furthermore, we define the \emph{gap lengths} $\gm_{n}(q)$ and the \emph{mid points} $\tau_{n}(q)$ by
\[
  \gm_{n}(q) \defl \lm_{n}^{+}(q)-\lm_{n}^{-}(q),
  \quad
  \tau_{n}(q) \defl \frac{\lm_n^+(q) + \lm_n^-(q)}{2},
  \qquad n\ge 1.
\]
For $q\in H_{0,\C}^{-1}$ we also consider the operator $L_{\dir}(q)$ defined as the operator $-\partial_{x}^{2} + q$ on $H^{-1}_{\dir}([0,1],\C)$ with domain of definition $H^{1}_{\dir}([0,1],\C)$. See Appendix~\ref{a:hill-op} as well as \cite{Djakov:2009fx,Kappeler:2001bi,Kappeler:2003vh,Korotyaev:2003gp,Savchuk:2003vl} for a more detailed discussion. The spectrum of $L_{\dir}(q)$ is called the \emph{Dirichelt spectrum of $q$}. It is also discrete and given by a sequence of eigenvalues $(\mu_{n})_{n\ge 1}$, counted with multiplicities, which when ordered lexicographically satisfies
\begin{align}
\label{asymt.mu.0}
  \mu_{1} \lex \mu_{2} \lex \mu_{2} \lex \dotsb,\qquad \mu_{n}  = n^{2}\pi^{2} + n\ell^2_n.
  \end{align}
Finally, define $\ell^{s,p}_\C(\N) \equiv \ell^{s,p}(\N, \C)$ to be the space of sequences $z=(z_n)_{n \ge 1} \subset \C$ with $\n{z}_{s,p}\defl \left( \sum_{n \ge 1} \langle n \rangle^{sp} \abs{z_n}^p\right)^{1/p}$ and denote by $\ell^{s,p}(\N) \equiv \ell^{s,p}(\N, \R)$ the real subspace of real sequences $z=(z_n)_{n \ge 1} \subset \R$ in $\ell^{s,p}_\C(\N)$.

In \cite{Kappeler:2005fb} the following asymptotic estimates are proved:

\begin{thm}[\cite{Kappeler:2005fb}]
\label{spec-W}
There exists a complex neighborhood $\Wp $ of $H_0^{-1}$ in $H_{0,\C}^{-1}$ so that the following holds:
\begin{renum}
\item
For any $q \in \Wp$, the sequence $(\gm_n(q))_{n \ge 1}$ is in $\ell_{\C}^{-1,2}(\N)$ and the map $\Wp \to \ell_{\C}^{-1,2}(\N)$, $q \mapsto (\gm_n(q))_{n \ge 1}$ is locally bounded.
\item
The mappings
\[
  q \mapsto (n^2\pi^2 - \tau_n(q))_{n \ge 1} \quad \mbox{ and } \quad
  q \mapsto (n^2\pi^2 - \mu_n(q))_{n \ge 1}
\]
are analytic from $\Wp$ into $\ell_{\C}^{-1,2}(\N)$.
\end{renum}
\end{thm}

By \cite{Kappeler:2005fb} the discriminant $\Dl(\lm,q)$ of $L(q)$ admits an analytic extension to  $\C\times \Wp$ and we have the product representation
\begin{equation}
  \label{Dl2}
  \Dl^2(\lm,q) - 4
   = 
  -4(\lm-\lm_{0}^{+})\prod_{m\ge 1} \frac{(\lm_m^+-\lm)(\lm_m^--\lm)}{m^4\pi^4}.
\end{equation}
The $\lm$-derivative $\dDl$ of the discriminant $\Dl$ is analytic on $\C\times\Ws$, too, and admits  the product representation
\begin{equation}
  \label{dDl2}
  \dDl(\lm)  = -\prod_{m\ge 1} \frac{\lm_{m}^{\ld}-\lm}{m^{2}\pi^{2}}
\end{equation}
where $(\lm_m^\ld)_{m \ge 1} \subset \C$ are ordered lexicographically
\[
\lm_1^\ld \lex \lm_2^\ld \lex \cdots \ , \quad \lm_n^\ld = n^2\pi^2 + n\ell_n^2 \ . 
\]
For each potential $q\in H_{0}^{-1}$ there exists a complex neighborhood $\Wp_{q}$ of $q$ in $H_{0,\C}^{-1}$ such that the closed intervals
\[
  G_{0} = \setdef{\lm_{0}^{+} + t}{-\infty < t \le 0},\qquad 
  G_{n} = [\lm_{n}^{-},\lm_{n}^{+}],\quad n\ge 1,
\]
are disjoint from each other for every $p\in \Wp_{q}$. By \cite{Kappeler:2005fb} there exist mutually disjoint neighborhoods $U_{n}\subsetneq \C$, $n\ge 0$, called \emph{isolating neighborhoods}, which satisfy:

\begin{renum}
\item $G_{n}$, $\lm_{n}^{\ld}$, and $\mu_{n}$ are contained in the interior of $U_{n}$ for every $p\in \Wp_{q}$,

\item there exists a constant $c \ge 1$ such that for $m\neq n$,
\[
  c^{-1}\abs{m^{2}-n^{2}} \le \dist(U_{n},U_{m}) \le c\abs{m^{2}-n^{2}},
\]

\item $U_{n} = \setdef{\lm\in\C}{\abs{\lm - n^{2}\pi^{2}} \le n/2}$ for $n$ sufficiently large.
\end{renum}

\noindent
In the sequel, for any $q \in \Wp$, $\Wp_{q}$ denotes a neighborhood of $q$ such that a common set of isolating neighborhoods for all $p\in \Wp_{q}$ exists. We  shrink $\Wp$, if necessary,  such that $\Wp$ is contained in the union of all $\Wp_{q}$ with $q\in H_{0}^{-1}$.

Following \cite{Kappeler:2005fb}, for $q \in \Wp$, one can define action variables for the $\KdV$ equation by
\begin{equation}
  \label{action}
  I_n
   = 
  \frac{1}{\pi}\int_{\Gm_n}
  \frac{\lm \dDl(\lm)}{\sqrt[c]{\Dl^2(\lm)-4}}
  \,\dlm,\qquad n\ge 1.
\end{equation}
Here  $\Gm_{n}$ denotes any sufficiently close counter clockwise oriented circuit around $G_{n}$, and the \emph{canonical root} $\sqrt[c]{\Dl^{2}(\lm)-4}$ is defined on $\C\setminus \bigcup_{\atop{n\ge 0}{\gm_{n}\neq 0}} G_{n}$ with $\gm_{0}= \infty$, where, for $q$ real, the sign of the root is determined by
\[
  \ii\sqrt[c]{\Dl^{2}(\lm)-4} > 0,\qquad \lm_{0}^{+} < \lm < \lm_{1}^{-} \ ,
\]
and for $q\in\Wp$ it is defined by continuous extension.

The  Dirichlet eigenvalues and the discriminant can be used to construct the angles $\th_k(q)$, $k\ge 1$, which are conjugated to the actions $I_n(q)$, $n\ge 1$, (cf. \cite[\S\;3.4]{Kappeler:2005fb}). 
For any given $k\ge 1$ the action $I_{k}$ is a real analytic function on $\Wp$ (cf. \cite[Proposition~3.3]{Kappeler:2005fb}), whereas the angle $\th_{k}$ is defined modulo $2\pi$ on $\Wp\setminus Z_{k}$ and is a real analytic function on $\Wp\setminus Z_{k}$ when considered modulo $\pi$.
Here
\begin{equation}
  \label{Zk}
  Z_k \defl \setdef{q\in \Wp}{\gm_k(q)=0} \ , \quad k \ge 1.
\end{equation}
It was shown in \cite[Proposition~4.3]{Kappeler:2005fb} that $Z_{k}\cap H_{0}^{-1}$ is a real analytic submanifold of $H_{0}^{-1}$ of codimension two.

By \cite[\S\;6]{Kappeler:2005fb} the following commutator relations hold for any $m,n\ge 1$
\begin{align}
  \label{e:commutators0}
  & \pbr{I_m,I_n} = 0, &&\mbox{on}\;\; \Wp,
  \qquad\qquad\qquad\qquad\qquad\qquad\qquad\qquad\qquad\qquad\\
  \label{e:commutators1}
  &\pbr{I_m,\th_n} = \dl_{nm}, &&\mbox{on}\;\; \Wp\setminus Z_n,\\
  \label{e:commutators2}
  & \pbr{\th_m,\th_n} = 0, && \mbox{on}\;\;\Wp\setminus(Z_m\cup Z_n).
\end{align}

For any $q\in H^{-1}_0\setminus Z_k$ with $k \ge 1$ define
\begin{equation}
  \label{e:z_k}
  z_k(q) \defl \sqrt[+]{I_k(q)}\, \e^{-\ii \th_k(q)} \ ,  \qquad 
  z_{-k}(q) \defl \sqrt[+]{I_k(q)}\, \e^{\ii \th_k(q)}.
\end{equation}
It is shown in \cite[\S\,5]{Kappeler:2005fb} that the mappings $H^{-1}_0\setminus Z_k\to{\C}$, $q\mapsto z_{\pm k}(q)$, analytically extend to a common neighborhood $\Wp$ of $H_0^{-1}$ in $H^{-1}_{0,\C}$.
The {\em Birkhoff map} is then defined as follows
\begin{equation}
  \label{Phi.def} 
  \Phi\colon \Wp \to \ell^{-1/2,2}_{0,\C}, \quad 
  q \mapsto \Phi(q) \defl \left( z_k(q)\right)_{k \in  \Z}
\end{equation}
with $z_0(q) = 0$.

\begin{thm}[\cite{Kappeler:2005fb}] 
\label{Th:main*}
The mapping $\Phi \colon \Wp \to\ell^{-1/2,2}_{0,\C}$ is analytic and satisfies the following properties:
\begin{renum}
  \item $\Phi$ is canonical in the sense that $\pbr{z_{n},z_{-n}} = \int_{0}^{1}\partial_{u}z_{n}\partial_{x}\partial_{u}z_{-n}\,\dx = \ii$ for all $n\ge 1$, whereas all other brackets between coordinate functions vanish.

  \item For any $s\ge-1$, the restriction $\Phi_s\equiv\Phi|_{H^s_0}$ is a map 
  $\Phi|_{H^s_0} \colon H^s_0\to\ell^{s+1/2,2}_0$ which is a bianalytic diffeomorphism.

  \item The \KdV Hamiltonian, when expressed in the new variables $\Hkc\circ\Phi^{-1}$, is defined on $\ell_{0}^{3/2,2}$ and depends on the action variables alone. In fact, it is a real analytic function of the actions on the positive quadrant $\ell_{+}^{3,1}(\N)$, introduced in the introduction.
\end{renum}
\end{thm}

We will also need the following result (cf. \cite[\S\,3]{Kappeler:2005fb}).

\begin{thm}[\cite{Kappeler:2005fb}]
\label{bm.H-1}
After  shrinking, if necessary,   the  complex neighborhood $\Wp$ of $H_0^{-1}$ in $H_{0,\C}^{-1}$ constructed above satisfies:
\begin{renum}
\item
The quotient $I_{n}/\gm_n^2$, defined on $H_0^{-1} \setminus Z_n$, extends analytically to $\Wp$ for any $n$. Moreover, for any $\ep >0$ and any $q \in \Wp$ there exists $n_0 \ge 1$ and an open neighborhood $\Wp_{q}$ of $q$ in $\Wp$ so that
\begin{equation}
  \label{In-gmn}
  \abs*{ 8 n\pi  \frac{I_n}{\gm_n^2} - 1 } \le \ep,\qquad 
  \forall n\ge n_{0}\; \forall p\in \Wp_{q}.
\end{equation}

\item
The Birkhoff coordinates $(z_n)_{n \in \Z}$ are analytic as maps from $\Wp$ into $\C$ and fulfill
locally uniformly in $\Wp$ and uniformly for $n \ge 1$, the estimate
\[
  |z_{\pm n}| = O \p*{ \frac{|\gm_n| + |\mu_n - \tau_n|}{\sqrt{n}} }.
\]

\item
For any $q\in \Wp$ and $n \ge 1$ one has $I_{n}(q) = 0$ if and only if $\gm_{n}(q) = 0$. In particular, $\Phi(0) = 0$. 
\end{renum} 
\end{thm}


\section{Asymptotic estimates for the periodic and Dirichlet eigenvalues}
\label{section3}

In this section we prove asymptotic estimates of the quantities $|\gamma_n|$ and $|\mu_n - \tau_n|$ appearing in the estimates of the Birkhoff coordinates of Theorem \ref{bm.H-1} (i) - (ii) in the case where  $q\in\Ls^{s,p}_{0,\C}$ with  $-1/2\le s\le 0$ and $2\le p < \infty$. They are  needed in the proofs of Theorem~\ref{main} and Theorem~\ref{main2}.

Recall from \eqref{asympt.eig.0} that for $q \in H^{-1}_{0,\C}$,  the sequence of gap lengths $(\gm_{n}(q))_{n\ge 1}$ is in $\ell^{-1,2}_{\C}(\N)$.
For $q \in \Ls^{s,p}_{0,\C}$, this sequence has the following stronger decay:

\begin{thm}
\label{fw:per}
For any $q\in\Ls_{0,\C}^{s,p}$ with $-1/2\le s \le 0$ and $2\le p < \infty$, one has $(\gm_{n}(q))_{n\ge 1}\in\ell_{\C}^{s,p}(\N)$ and the map
\[
  \Ls_{0,\C}^{s,p}\to \ell_{\C}^{s,p}(\N),\qquad q \mapsto (\gm_{n}(q))_{n\ge 1},
\]
is locally bounded.\fish
\end{thm}
Recall from \eqref{asympt.eig.0}--\eqref{asymt.mu.0} that for any $q \in H^{-1}_{0,\C},$ $(\tau_n(q) - \mu_n(q))_{n \ge 1}$ is in $\ell^{-1,2}_\C(\N)$. For $q$ in $\Ls^{s,p}_{0,\C}$, the decay of this sequence is again stronger:
\begin{thm}
\label{fw:dir}
For any $q\in\Ls_{0,\C}^{s,p}$ with $-1/2\le s \le 0$ and $2\le p < \infty$, one has $(\tau_{n}-\mu_{n}(q))_{n\ge 1}\in\ell_{\C}^{s,p}(\N)$ and the map
\[
  \Ls_{0,\C}^{s,p}\to \ell_{\C}^{s,p}(\N),\qquad q \mapsto (\tau_{n}(q)-\mu_{n}(q))_{n\ge 1},
\]
is locally bounded.\fish
\end{thm}

In the remainder of this section we prove Theorem~\ref{fw:per} and Theorem~\ref{fw:dir}. For this purpose we adapt the methods developed in \cite{Kappeler:2001hsa,Poschel:2011iua,Djakov:2006ba} for potentials $q \in L^2$ to singular potentials, using results from \cite{Kappeler:2001bi} and \cite{Kappeler:2005fb}.

\subsection{Setup}

As in the introduction we extend the $L^{2}$-inner product of $L^{2}(\R/2\Z)$ by duality to $\Sc_{\C}'(\R/2\Z)\times \Sc_{\C}(\R/2\Z)$. For $s\in\R$ and $1\le p < \infty$ denote by $\Ls_{\star,\C}^{s,p}$ the space of periodic, complex valued distributions $f\in \Sc_{\C}'(\R/2\Z)$ so that the sequence of their  Fourier coefficients $f_{n} = \spii{f,e_{n}}$ is in the space $\ell^{s,p}_\C = \setdef{ z=(z_n)_{n \in \Z} \subset \C }{ \n{z}_{s,p} < \infty }$. We write 
\[
  \n{f}_{s,p}\defl  \n{(f_n)_{n \in \Z}}_{s,p} =  \p*{ \sum_{n\in\Z} \w{n}^{sp} \abs{f_{n}}^{p} }^{1/p}.
\]

In the sequel we will identify a potential $q\in \Ls_{0,\C}^{s,p}$ with the corresponding element $\sum_{n \in \Z} q_{n}e_{n}$ in $\Ls_{\star,\C}^{s,p}$ where $q_{n}$ is the $n$th Fourier coefficient of the potential obtained from $q$ by viewing it as a distribution on $\R/2\Z$ instead of $\R/\Z$, i.e., $q_{2n} = \spii{q, e_{2n}}$, whereas $q_{2n+1} = \spii{q, e_{2n+1}} = 0$ and $q_0 = \spii{q,1} = 0$. We denote by $V$ the operator of multiplication by $q$ with domain $H_{\C}^{1}(\R/2\Z)$. See Appendix~\ref{a:hill-op} for a detailed discussion of this operator as well as the operator $L(q)$ introduced in~\eqref{Lq}. When expressed in its Fourier series, the image $Vf$ of $f = \sum_{n\in\Z} f_{n}e_{n}\in H_{\C}^{1}(\R/2\Z)$ is the distribution $Vf = \sum_{n\in \Z}\p*{ \sum_{m\in \Z} q_{n-m}f_{m} }e_{n}\in H_{\C}^{-1}(\R/2\Z)$. To prove the asymptotic estimates of the gap lengths stated in Theorem~\ref{fw:per} we need to study the eigenvalue equation $L(q)f = \lm f$ for sufficiently large periodic eigenvalues $\lm$. For $q\in H_{0,\C}^{-1}$, the domain of $L(q)$ is $H_{\C}^{1}(\R/2\Z)$ and hence the eigenfunction $f$ is an element of this space. It is shown in Appendix~\ref{a:hill-op} that for $q\in \Ls^{s,p}_{0,\C}$ one has $f\in \Ls_{\star,\C}^{s+2,p}$ and $\partial_{x}^{2}f$, $Vf\in \Ls_{\star,\C}^{s,p}$. Note that for $q = 0$ and any $n\ge 1$, $\lm_{n}^{+}(0) = \lm_{n}^{-}(0) = n^{2}\pi^{2}$, and the eigenspace corresponding to the double eigenvalue $\lm_{n}^{+}(0) = \lm_{n}^{-}(0)$ is spanned by $e_{n}$ and $e_{-n}$. Viewing for $n$ large $L(q)-\lm_{n}^{\pm}(q)$ as a perturbation of $L(0)-\lm_{n}^{\pm}(0)$, we are led to decompose $\Ls_{\star,\C}^{s,p}$ into the direct sum $\Ls_{\star,\C}^{s,p} = \Pc_{n}\oplus\Qc_{n}$,
\[
  \Pc_{n} = \spanx\setd{e_{n},\, e_{-n}},\qquad
  \Qc_{n} = \overline{\spanx}\setdef{e_{k}}{k\neq \pm n}.
\]
The projections onto $\Pc_{n}$ and $\Qc_{n}$ are denoted by $P_{n}$ and $Q_{n}$, respectively. It is convenient to write the eigenvalue equation $Lf = \lm f$ in the form $A_{\lm}f = Vf$, where $A_{\lm}f = \partial_{x}^{2}f + \lm f$ and $V$ denotes the operator of multiplication with $q$. Since $A_{\lm}$ is a Fourier multiplier, we write $f = u+v$, where $u = P_{n}f$ and $v = Q_{n}f$, and decompose the equation $A_{\lm}f = V f$ into the two equations
\begin{equation}
  \label{P-Q-eqn}
  A_{\lm}u = P_{n}V(u+v),\qquad
  A_{\lm}v = Q_{n}V(u+v),
\end{equation}
referred to as $P$- and $Q$-equation. As $\lm_{n}^{\pm}(q) = n^{2}\pi^{2} + n\ell_{n}^{2}$, we conclude that for $n$ sufficiently large, $\lm_{n}^{\pm}(q)\in S_{n}$ where $S_{n}$ denotes the closed vertical strip
\[
  S_{n} \defl \setdef{\lm\in \C}{\abs{\Re \lm - n^{2}\pi^{2}} \le 12 n},\qquad n\ge 1.
\]
Note that $\setdef{\lm\in \C}{\Re \lm \ge 0}\subset \bigcup_{n\ge 1} S_{n}$. Given any $n\ge 1$, $u\in \Pc_{n}$, and $\lm\in S_{n}$, we derive in a first step from the $Q$-equation an equation for $w\defl Vv$ which for $n$ sufficiently large can be solved as a function of $u$ and $\lm$.
In a second step, for $\lm$ a periodic eigenvalue in $S_{n}$, we solve the $P$ equation for $u$ after having substituted in it the expression of $Vv$.
The solution of the $Q$-equation is then easily determined.
Towards the first step note that for any $\lm\in S_{n}$, $A_{\lm}\colon \Qc_{n}\cap \Ls_{\star,\C}^{s+2,p}\to \Qc_{n}$ is boundedly invertible as for any $k\neq n$,
\[
  \min_{\lm\in S_{n}}\abs{\lm-k^{2}\pi^{2}}
   \ge \min_{\lm\in S_{n}} \abs{\Re \lm - k^{2}\pi^{2}}
   \ge \abs{n^{2}-k^{2}} \ge 1.
\]
In order to derive from the $Q$-equation an equation for $Vv$, we apply to it the operator $VA_{\lm}^{-1}$ to get
\[
  Vv = VA_{\lm}^{-1}Q_{n}V(u+v) = T_{n}V(u+v),
\]
where
\[
  T_{n}\equiv T_{n}(\lm) \defl VA_{\lm}^{-1}Q_{n}
  \colon \Ls_{\star,\C}^{s,p}\to \Ls_{\star,\C}^{s,p}.
\]
It leads to the following equation for $w \defl Vv$
\begin{equation}
  \label{w-eqn}
  (\Id-T_{n}(\lm))w = T_{n}(\lm)Vu.
\end{equation}
To show that $\Id-T_{n}(\lm)$ is invertible, we introduce for any $l\in\Z$ the shifted norm of $f\in\Ls_{\star,\C}^{s,p}$,
\begin{equation*}
  \n{f}_{s,p;l} \defl \n{fe_{l}}_{s,p}
   = \p*{ \sum_{n\in\Z} \w{n+l}^{p} \abs{f_{n}}^{p} }^{1/p},
\end{equation*}
and denote by $\n{T_{n}}_{s,p;l}$ the corresponding operator norm.
Furthermore, we denote by $R_{N}f$, $N\ge 1$ the tail of the Fourier series of $f\in \Ls_{\star,\C}^{s,p}$,
\[
  R_{N}f = \sum_{\abs{n}\ge N} f_{n}e_{n}.
\]

\begin{lem}
\label{Tn-est}
Let $-1/2\le s \le 0$, $2\le p < \infty$, and $n\ge 1$ be given. For any $q\in\Ls_{0,\C}^{s,p}$ and $\lm\in S_{n}$,
\[
  T_{n}(\lm)\colon \Ls_{\star,\C}^{s,p}\to \Ls_{\star,\C}^{s,p}
\]
is a bounded linear operator satisfying the estimate
\[
  \n{T_{n}(\lm)}_{s,p;\pm n} \le c_{s,p}\p*{ \ep_{s,p}(n)\n{q}_{s,p} + \n{R_{n/2}q}_{s,p} },
\]
where $c_{s,p}$ is a constant depending only on $s$ and $p$ and $\ep_{s,p}$ is defined here as
\[
  \ep_{s,p}(n) \defl \begin{cases}
  \frac{1}{n^{1-\abs{s}}}, & 2 \le p < \frac{1}{\abs{s}},\\
  \frac{(\log \w{n})^{1-1/p}}{n^{1-\abs{s}}}, & p = \frac{1}{\abs{s}},\\
  \frac{1}{n^{1-2\abs{s}+1/p}}, & \frac{1}{\abs{s}} < p < \infty.
  \end{cases}\fish
\]
\end{lem}

\begin{proof}
Note that $A_{\lm}^{-1}\colon \Ls_{\star,\C}^{s,p}\to \Ls_{\star,\C}^{s+2,p}$ is bounded for any $\lm\in S_{n}$ and hence for any $f\in \Ls_{\star,\C}^{s,p}$, $q\in \Ls_{0,\C}^{s,p}$, and $\lm\in S_{n}$, the multiplication of $A_{\lm}^{-1}Q_{n}f$ with $q$, defined by
\[
  T_{n}(\lm)f = VA_{\lm}^{-1}Q_{n}f = \sum_{m\in\Z}\p*{ \sum_{\abs{k}\neq n} \frac{q_{m-k} f_{k}}{\lm-k^{2}\pi^{2}} }e_{m}
\]
is a distribution in $\Sc_{\C}'(\R/2\Z)$. The norm $\n{T_{n}f}_{s,p;n}$ satisfies
\[
  \n{T_{n}f}_{s,p;n}^{p} \le \sum_{m\in\Z} \p*{ \w{m+n}^{s}\sum_{\abs{k}\neq n} 
  \frac{\abs{q_{m-k}} \abs{f_{k}}}{\abs{n^{2}-k^{2}}}
   }^{p}.
\]
We write the latter expression in the form
\[
  \sum_{m\in\Z} \p*{ \sum_{\abs{k}\neq n} 
    \frac{\w{k+n}^{\abs{s}}  \w{m-k}^{\abs{s}}}
         {\abs{n+k}\abs{n-k} \w{m+n}^{\abs{s}}}
    \frac{\abs{q_{m-k}}} {\w{m-k}^{\abs{s}}}
    \frac{\abs{f_{k}}}   {\w{k+n}^{\abs{s}}}
   }^{p} \le 2^{p}(I+II),
\]
where $I$ and $II$ are defined by splitting the sum $\sum_{\abs{k}\neq n}$ into a sum over $\setdef{\abs{k}\neq n}{\abs{m-k}\ge n/2}$ and one over $\setdef{\abs{k}\neq n}{\abs{m-k} < n/2}$.
Using that $\w{m-k}\le \w{m+n}\w{n+k}$ we conclude
\[
  I \le \sum_{m\in\Z} \p*{ 
             \sum_{\abs{k}\neq n} \mathbf{1}_{\setd{\abs{m-k} \ge n/2}} 
             \frac{\w{k+n}^{2\abs{s}}}{\abs{n+k}}
             \frac{1}{\abs{n-k}}
             \frac{\abs{q_{m-k}}}{\w{m-k}^{\abs{s}}}
             \frac{\abs{f_{k}}}{\w{k+n}^{\abs{s}}}         
   }^{p}.
\]
As $-1/2\le s\le 0$ and $\w{k+n}^{2\abs{s}}/\abs{n+k} \le (1+\abs{n+k})/\abs{n+k} \le 2$, one then gets by Hölder's inequality applied to the sum $\sum_{\abs{k}\neq n}$
\[
  I \le \p*{\sum_{\abs{k}\neq n} 
             \frac{2^{p'}}{\abs{n-k}^{p'}}}^{p/p'}
             \n{R_{n/2}q}_{s,p}^{p}\n{f}_{s,p;n}^{p}.
\]
To estimate $II$, we show that
\begin{equation}
  \label{eq-x}
  \frac{\w{m-k}}{\w{n-k}\w{m+n}}
  \mathbf{1}_{\setd{\abs{m-k}< n/2}} \le 1.
\end{equation}
Indeed, if $\abs{n-k} > n/2$, this is obvious as $\abs{m-k} < n/2$. In case $\abs{n-k} \le n/2$ one has that $\abs{n+k} \ge 2n-\abs{n-k}\ge 3n/2$ and hence
\[
  \abs{m+n} \ge \abs*{\abs{n+k}-\abs{m-k}} \ge n,
\]
which implies \eqref{eq-x}.
With \eqref{eq-x} and Hölder's inequality it then follows that
\begin{align*}
  II &\le \sum_{m\in\Z} \p*{ 
             \sum_{\abs{k}\neq n} \mathbf{1}_{\setd{\abs{m-k} < n/2}}
             \frac{2^{\abs{s}}}{\abs{n+k}^{1-\abs{s}}}
             \frac{2^{\abs{s}}}{\abs{n-k}^{1-\abs{s}}}
             \frac{\abs{q_{m-k}}}{\abs{\w{m-k}^{\abs{s}}}}
             \frac{\abs{f_{k}}}{\w{k+n}^{\abs{s}}}}^{p}\\
     &\le \p*{ 
             \sum_{\abs{k}\neq n} 
             \frac{2^{p'}}{\abs{n^{2}-k^{2}}^{(1-\abs{s})p'}}
             }^{p/p'}
             \n{q}_{s,p}^{p}\n{f}_{s,p;n}^{p}.
\end{align*}
Using Lemma~\ref{hilbert-sum}, one verifies in a straightforward way that
\[
  \p*{\sum_{\abs{k}\neq n} 
             \frac{1}{\abs{n^{2}-k^{2}}^{(1-\abs{s})p'}}
             }^{1/p'} \le C_{s,p}\ep_{s,p}(n),
\]
for some constant $C_{s,p} > 0$ which only depends on $s$ and $p$. Altogether we have shown that for some constant $c_{s,p} \ge C_{s,p}$,
\[
  \n{T_{n}f}_{s,p;n} \le c_{s,p}\p*{ \ep_{s,p}(n)\n{q}_{s,p} + \n{R_{n/2}q}_{s,p} }
  \n{f}_{s,p;n}.
\]
Going through the arguments of the proof one sees that the same kind of estimates also lead to the claimed bound for $\n{T_{n}f}_{s,p;-n}$.\qed
\end{proof}

Lemma~\ref{Tn-est} can be used to solve, for $n$ sufficiently large, the equation \eqref{w-eqn} as well as the $Q$-equation \eqref{P-Q-eqn} in terms of any given $u\in\Pc_{n}$ and $\lm\in S_{n}$.

\begin{cor}
\label{Q-soln}
For any $q\in\Ls_{0,\C}^{s,p}$ with $-1/2\le s \le 0$ and $2\le p < \infty$, there exists $n_{s,p} = n_{s,p}(q) \ge 1$ so that for any $n\ge n_{s,p}$ and $\lm\in S_{n}$, $T_{n}(\lm)$ is a $1/2$ contraction on $\Ls_{\star,\C}^{s,p}$ with respect to the norms shifted by $\pm n$, $\n{T_{n}(\lm)}_{s,p;\pm n}\le 1/2$. The threshold $n_{s,p}(q)$ can be chosen locally uniformly in $q$. As a consequence, equation~\eqref{w-eqn} and \eqref{P-Q-eqn} can be uniquely solved for any given $u\in\Pc_{n}$, $\lm\in S_{n}$,
\begin{align}
  \label{sol-w-eqn}
  w_{u,\lm} &= K_{n}(\lm)T_{n}(\lm)Vu \in \Ls_{\star,\C}^{s,p},
  \qquad K_{n} \equiv K_{n}(\lm) \defl (\Id - T_{n}(\lm))^{-1},\\
  \label{sol-Q-eqn}
    v_{u,\lm} &= A_{\lm}^{-1}Q_{n}Vu + A_{\lm}^{-1}Q_{n}w_{u,\lm}
    = 
  A_{\lm}^{-1}Q_{n}K_{n}Vu \in \Ls_{\star,\C}^{s+2,p}\cap \Qc_{n}.
\end{align}
In particular, one has $ w_{u,\lm} = Vv_{u,\lm} $. We increase $n_{s,p}$, if necessary, so that in addition
\begin{equation}
  \label{n-s-p}
  \begin{split}
  &\abs{\lm_{n}^{\pm}} \le (n_{s,p}-1)^{2}\pi^{2} + n_{s,p}/2,
  && \forall n < n_{s,p},\\
  &\abs{\lm_{n}^{\pm}-n^{2}\pi^{2}} \le n/2,
  && \forall n\ge n_{s,p},\\
  &\lm_{n}^{\pm}\text{ are 1-periodic [1-antiperiodic] if }n \text{ even [odd] }
  && \forall n\ge n_{s,p}.\fish
  \end{split}
\end{equation}
\end{cor}

\begin{rem}
\label{inhom-1}
By the same approach, one can study the inhomogeneous equation
\[
  (L-\lm)f = g,\qquad g\in \Ls_{\star,\C}^{s,p},
\]
for $\lm\in S_{n}$ and $n\ge n_{s,p}$. Writing $f=u+v$ and $g = P_{n}g + Q_{n}g$, the $Q$-equation becomes
\[
  A_{\lm}v = Q_{n}V(u+v)- Q_{n}g = Q_{n}Vv + Q_{n}(Vu-g)
\]
leading for any given $u\in\Pc_{n}$ to the unique solution $w$ of the equation corresponding to~\eqref{w-eqn}
\[
  w = Vv = K_{n}T_{n}(Vu-g) \in \Ls_{\star,\C}^{s,p},
\]
and, in turn, to the unique solution $v \in \Ls_{\star,\C}^{s+2,p}\cap \Qc_{n}$ of the $Q$-equation
\[
  v = A_{\lm}^{-1}Q_{n}(Vu - g) + A_{\lm}^{-1}Q_{n}K_{n}T_{n}(Vu-g)  = A_{\lm}^{-1}Q_{n}K_{n}(Vu-g).
\]
\end{rem}

\subsection{Reduction}
\label{ss:spec-reduction}

In a next step we study the $P$-equation $A_{\lm}u = P_{n}V(u+v)$ of \eqref{P-Q-eqn}. For $n\ge n_{s,p}(q)$, $u\in \Pc_{n}$, and $\lm\in S_{n}$, substitute in it the solution $w_{u,\lm}$ of \eqref{w-eqn}, obtained in~\eqref{sol-w-eqn},
\[
  A_{\lm}u = P_{n}Vu + P_{n}w_{u,\lm} = P_{n}(\Id + K_{n}T_{n})Vu.
\]
Using that $\Id + K_{n}T_{n} = K_{n}$ one then obtains $A_{\lm}u = P_{n}K_{n}Vu$ or $B_{n}u = 0$, where
\begin{equation}
  B_{n} \equiv B_{n}(\lm)\colon \Pc_{n}\to \Pc_{n},\quad u\mapsto (A_{\lm}-P_{n}K_{n}(\lm)V)u.
\end{equation}

\begin{lem}
\label{eval-char}
Assume that $q\in \Ls_{0,\C}^{s,p}$ with $-1/2\le s \le 0$ and $2\le p <\infty$. Then for any $n\ge n_{s,p}$ with $n_{s,p}$ given by Corollary~\ref{Q-soln}, $\lm\in S_{n}$ is an eigenvalue of $L(q)$ if and only if $\det(B_{n}(\lm)) = 0$.\fish
\end{lem}

\begin{proof}
Assume that $\lm\in S_{n}$ is an eigenvalue of $L=L(q)$. By Lemma~\ref{regularity-En} there exists $0\neq f\in \Ls_{\star,\C}^{s+2,p}$ so that $Lf = \lm f$. Decomposing $f = u+v\in \Pc_{n}\oplus \Qc_{n}$ it follows by the considerations above and the assumption $n\ge n_{s,p}$ that $u\neq 0$ and $B_{n}(\lm)u = 0$. Conversely, assume that $\det(B_{n}(\lm)) = 0$ for some $\lm\in S_{n}$. Then there exists $0\neq u\in \Pc_{n}$ so that $B_{n}(\lm)u = 0$. Since $n\ge n_{s,p}$, there exists $w_{u,\lm}$ and $v_{u,\lm}$ as in \eqref{sol-w-eqn} and \eqref{sol-Q-eqn}, respectively. Then $v \equiv v_{u,\lm}\in \Ls_{\star,\C}^{s+2,w,p}\cap \Qc_{n}$ solves the $Q$-equation by Corollary~\ref{Q-soln}. To see that $u$ solves the $P$-equation note that $B_{n}(\lm)u = 0$ implies that
\[
  A_{\lm}u = P_{n}K_{n}Vu = P_{n}(\Id + K_{n}T_{n})Vu.
\]
As by \eqref{sol-w-eqn}, $w_{u,\lm} = K_{n}T_{n}Vu$, and as $w_{u,\lm} = Vv$, one sees that indeed
\[
  A_{\lm}u = P_{n}Vu + P_{n}Vv.\qed
\]
\end{proof}

\begin{rem}
\label{inhom-2}
Solutions of the inhomogeneous equation $(L-\lm)f = g$ for $g\in\Ls_{\star,\C}^{s,p}$, $\lm\in S_{n}$, and $n\ge n_{s,p}$ can be obtained by substituting into the $P$-equation
\[
  A_{\lm}u = P_{n}Vu + P_{n}Vv - P_{n}g
\]
the expression for $Vv$ obtained in Remark~\ref{inhom-1}, $Vv = K_{n}T_{n}(Vu-g)$, to get
\begin{align*}
  A_{\lm}u &= P_{n}Vu + P_{n}K_{n}T_{n}Vu - P_{n}g - P_{n}K_{n}T_{n}g\\
  &= P_{n}(\Id + K_{n}T_{n})Vu - P_{n}(\Id + K_{n}T_{n})g.
\end{align*}
Using that $\Id + K_{n}T_{n} = K_{n}$ one concludes that
\begin{equation}
  \label{B-inhom}
  B_{n}(\lm)u = -P_{n}K_{n}(\lm)g.
\end{equation}
Conversely, for any solution $u$ of \eqref{B-inhom}, $f = u+v$, with $v$ being the element in $\Ls_{\star,\C}^{s+2,p}$ given in Remark~\ref{inhom-1}, satisfies $(L-\lm)f = g$ and $f\in \Ls_{\star,\C}^{s+2,p}$.
\end{rem}

We denote the matrix representation of a linear operator $F\colon \Pc_{n}\to \Pc_{n}$ with respect to the orthonormal basis $e_{n}$, $e_{-n}$ of $\Pc_{n}$ also by $F$,
\[
  F = \begin{pmatrix}
  \spii{Fe_{n},e_{n}}  & \spii{Fe_{-n},e_{n}}\\
  \spii{Fe_{n},e_{-n}}  & \spii{Fe_{-n},e_{-n}}
  \end{pmatrix}.
\]
In particular,
\[
  A_{\lm} = 
  \begin{pmatrix}
  \lm-n^{2}\pi^{2} & 0\\
  0 & \lm-n^{2}\pi^{2}
  \end{pmatrix},
  \qquad
  P_{n}K_{n}V = 
  \begin{pmatrix}
  a_{n} & b_{n}\\
  b_{-n} & a_{-n}
  \end{pmatrix},
\]
where for any $\lm\in S_{n}$ and $n\ge n_{s,p}$ the coefficients of $P_{n}K_{n}V$ are given  by
\begin{align*}
  a_{n}  &\equiv a_{n}(\lm)  \defl \spii{K_{n}Ve_{n},e_{n}},&
  a_{-n} &\equiv a_{-n}(\lm) \defl \spii{K_{n}Ve_{-n},e_{-n}},\\
  b_{n}  &\equiv b_{n}(\lm)  \defl \spii{K_{n}Ve_{-n},e_{n}},&
  b_{-n} &\equiv b_{-n}(\lm) \defl \spii{K_{n}Ve_{n},e_{-n}}.
\end{align*}
By a straightforward verification it follows from the expression of $a_{n}$ in terms of the representation of $K_{n}$ and $V$ in Fourier space that for any $n\ge n_{s,p}$
\[
  a_{n} = \spii{K_{n}Ve_{-n},e_{-n}} = a_{-n}.
\]
Hence the diagonal of $B_{n}$ is a complex multiple of the identity and
\begin{equation}
  \label{Bn-mat}
   B_{n} = \begin{pmatrix}
            \lm-n^{2}\pi^{2} - a_{n} & -b_{n},\\
           -b_{-n}                   &  \lm-n^{2}\pi^{2} - a_{n}
           \end{pmatrix}.
\end{equation}

\begin{lem}
\label{coeff-est}
Suppose $q\in \Ls_{0,\C}^{s,p}$ with $-1/2\le s\le 0$ and $2\le p < \infty$. Then for any $n\ge n_{s,p}$, with $n_{s,p}$ as in Corollary~\ref{Q-soln}, the coefficients $a_{n}(\lm)$ and $b_{\pm n}(\lm)$ are analytic functions on the strip $S_{n}$ and for any $\lm\in S_{n}$
\begin{align*}
  &\abs{a_{n}(\lm)}            \le 2\n{T_{n}(\lm)}_{s,p;\pm n}\n{q}_{s,p},\\
  &\w{n}^{s}\abs{b_{n}(\lm)  - q_{2n}}  \le 2\n{T_{n}(\lm)}_{s,p;n}    \n{q}_{s,p},\\
  &\w{n}^{s}\abs{b_{-n}(\lm) - q_{-2n}} \le 2\n{T_{n}(\lm)}_{s,p;-n}   \n{q}_{s,p}.\fish
\end{align*}
\end{lem}

\begin{proof}
Since $\n{T_{n}(\lm)}_{s,p;\pm n}\le 1/2$ for $n\ge n_{s,p}$ and $\lm\in S_{n}$, the series expansions of $a_{n}(\lm)$ and $b_{\pm n}(\lm)$ converge uniformly on $S_{n}$ to functions analytic in $\lm$. Moreover, we obtain
from the identity $K_{n} = \Id + T_{n}K_{n}$
\[
  b_{n} = \spii{Ve_{-n},e_{n}} + \spii{T_{n}K_{n}Ve_{-n},e_{n}}
  = q_{2n} + \spii{T_{n}K_{n}Ve_{-n},e_{n}}.
\]
Furthermore, for any $f\in \Ls_{\star,\C}^{s,p}$ we compute
\[
  \w{n}^{s}\abs{\spii{f,e_{n}}}
   = \w{n}^{s}\abs{\spii{fe_{n},e_{2n}}}
  \le \n{f}_{s,p;n}.
\]
Consequently, using that $\n{T_{n}}_{s,p;n} \le 1/2$ and hence $\n{K_{n}}_{s,p;n} \le 2$, one gets
\begin{align*}
  \w{n}^{s}\abs{b_{n}-q_{2n}} 
   \le \n{T_{n}K_{n}Ve_{-n}}_{s,p;n}
  &\le 2\n{T_{n}}_{s,p;n}\n{Ve_{-n}}_{s,p;n}\\
  &= 2\n{T_{n}}_{s,p;n}\n{q}_{s,p}.
\end{align*}
The estimates for $\abs{b_{-n}-q_{-2n}}$ and $\abs{a_{n}}$ are obtained in a similar fashion.\qed
\end{proof}

The preceding lemma implies that for any $n\ge n_{s,p}$, $\det B_{n}(\lm) = (\lm-n^{2}\pi^{2}-a_{n})^{2} - b_{n}b_{-n}$ is an analytic function of $\lm\in S_{n}$ which is close to $(\lm-n^{2}\pi^{2})^{2}$ for $n\ge n_{s,p}$ sufficiently large.

\begin{lem}
\label{Sn-roots}
Suppose $q\in \Ls_{0,\C}^{s,p}$ with $-1/2\le s \le 0$ and $2\le p < \infty$. Choose $n_{s,p} = n_{s,p}(q) \ge 1$ as in Corollary~\ref{Q-soln} and increase it if necessary so that
\begin{equation}
  \n{q}_{s,p} \le \sqrt{n_{s,p}} / 8.
\end{equation}
Then for any $n\ge n_{s,p}$, $\det(B_{n}(\lm))$ has exactly two roots $\xi_{n,1}$ and $\xi_{n,2}$ in $S_{n}$ counted with multiplicity. They are contained in
\[
  D_{n} \defl \setdef{\lm}{\abs{\lm-n^{2}\pi^{2}} < 4 n^{1/2}\n{q}_{s,p}} \subset S_{n}
\]
and satisfy
\begin{equation}
  \label{xi-est}
  \abs{\xi_{n,1}-\xi_{n,2}} \le \sqrt{6}\sup_{\lm\in S_{n}} \abs{b_{n}(\lm)b_{-n}(\lm)}^{1/2}.\fish
\end{equation}
\end{lem}

\begin{proof}
Since for any $n\ge n_{s,p}$ and $\lm\in S_{n}$, $\n{T_{n}(\lm)}_{s,p;\pm n} \le 1/2$, one concludes from the preceding lemma that $\abs{a_{n}(\lm)} \le \n{q}_{s,p}$ and, with $\abs{b_{\pm n}(\lm)} \le \abs{q_{\pm 2n}} + \abs{b_{\pm n}(\lm)-q_{\pm 2n}}$, that
\[
  \abs{b_{\pm n}(\lm)} \le 2\w{n}^{\abs{s}}\n{q}_{s,p}.
\]
Therefore, for any $\lm,\mu\in S_{n}$,
\begin{equation}
  \label{an-cn-est}
  \begin{split}
    \abs{a_{n}(\mu)} + \abs{b_{n}(\lm)b_{-n}(\lm)}^{1/2} &\le 
  \p*{ 1 + 2\w{n}^{\abs{s}} }\n{q}_{s,p}\\
  &< 4 n^{1/2}\n{q}_{s,p} = \inf_{\lm\in \partial D_{n}} \abs{\lm - n^{2}\pi^{2}}.
  \end{split}
\end{equation}
It then follows that $\det B_{n}(\lm)$ has no root in $S_{n}\setminus D_{n}$. Indeed, assume that $\xi\in S_{n}$ is a root, then $\abs{\xi - n^{2}\pi^{2}-a_{n}(\xi)} = \abs{b_{n}(\xi)b_{-n}(\xi)}^{1/2}$ and hence
\[
  \abs{\xi-n^{2}\pi^{2}}
   \le \abs{a_{n}(\xi)} + \abs{b_{n}(\xi)b_{-n}(\xi)}^{1/2}
    <  4n^{1/2}\n{q}_{s,p}
\]
implying that $\xi\in D_{n}$. In addition, \eqref{an-cn-est} implies that by Rouché's theorem the two analytic functions $\lm-n^{2}\pi^{2}$ and $\lm-n^{2}\pi^{2} -a_{n}(\lm)$, defined on the strip $S_{n}$ have the same number of roots in $D_{n}$ counted with multiplicities. As a consequence $(\lm-n^{2}\pi^{2}-a_{n}(\lm))^{2}$ has a double root in $D_{n}$. Finally, \eqref{an-cn-est} also implies that
\begin{align*}
  \sup_{\lm\in S_{n}} \abs{b_{n}(\lm)b_{-n}(\lm)}^{1/2} 
  &< \inf_{\lm\in \partial D_{n}} \abs{\lm-n^{2}\pi^{2 }} - \sup_{\lm\in S_{n}} \abs{a_{n}(\lm)}\\
  &\le \inf_{\lm\in \partial D_{n}} \abs{\lm-n^{2}\pi^{2} - a_{n}(\lm)}
\end{align*}
and hence again by Rouché's theorem, the analytic functions $(\lm-n^{2}\pi^{2}-a_{n}(\lm))^{2}$ and $(\lm-n^{2}\pi^{2}-a_{n}(\lm))^{2}-b_{n}(\lm)b_{-n}(\lm)$ have the same number of roots in $D_{n}$.
Altogether we thus have established that $\det(B_{n}(\lm)) = (\lm-n^{2}\pi^{2}-a_{n}(\lm))^{2} - b_{n}(\lm)b_{-n}(\lm)$ has precisely two roots $\xi_{n,1}$, $\xi_{n,2}$ in $D_{n}$.

To estimate the distance of the roots, write $\det B_{n}(\lm)$ as a product $g_{+}(\lm)g_{-}(\lm)$ where $g_{\pm}(\lm) = \lm-n^{2}\pi^{2} - a_{n}(\lm) \mp \ph_{n}(\lm)$
and $\ph_{n}(\lm) = \sqrt{b_{n}(\lm)b_{-n}(\lm)}$ with an arbitrary choice of the sign of the root for any $\lm$. Each root $\xi$ of $\det(B_{n})$ is either a root of $g_{+}$ or $g_{-}$ and thus satisfies
\[
  \xi \in \setd{n^{2}\pi^{2} + a_{n}(\xi) \pm \ph_{n}(\xi)}.
\]
As a consequence,
\begin{equation}
  \label{xi-1-2}
  \begin{split}
  \abs{\xi_{n,1}-\xi_{n,2}} &\le \abs{a_{n}(\xi_{n,1})-a_{n}(\xi_{n,2})} + 
  \max_{\pm}\abs{\ph_{n}(\xi_{n,1}) \pm \ph_{n}(\xi_{n,2})}\\
  &\le \sup_{\lm\in D_{n}} \abs{\partial_{\lm} a_{n}(\lm)}\abs{\xi_{n,1}-\xi_{n,2}}
  + 2\sup_{\lm\in D_{n}} \abs{\ph_{n}(\lm)}.
  \end{split}
\end{equation}
Since $\n{q}_{s,p}\le n_{s,p}^{1/2}/8 \le n^{1/2}/8$ and hence
\[
  \dist(D_{n},\partial S_{n}) \ge 12n - n/2 \ge 6n
\]
one concludes from Cauchy's estimate
\[
  \sup_{\lm\in D_{n}} \abs{\partial_{\lm}a_{n}(\lm)}
   \le \frac{\sup_{\lm\in S_{n}} \abs{a_{n}(\lm)}}{\dist(D_{n},\partial S_{n})}
   \le \frac{\n{q}_{s,p}}{6n} \le \frac{1}{6}.
\]
Therefore, by \eqref{xi-1-2},
\[
  \abs{\xi_{n,1}-\xi_{n,2}}^{2} \le 6\sup_{\lm\in D_{n}}\abs{b_{n}(\lm)b_{-n}(\lm)}
\]
as claimed.\qed
\end{proof}

\subsection{Proof of Theorem~\ref{fw:per}}

Let $q\in \Ls_{0,\C}^{s,p}$ with $-1/2\le s \le 0$ and $2\le p < \infty$. The eigenvalues of $L(q)$, when listed with lexicographic ordering, satisfy
\[
  \lm_{0}^{+}\lex \lm_{1}^{-}\lex \lm_{1}^{+}\lex \dotsb,\quad
  \text{and}\quad
  \lm_{n}^{\pm} = n^{2}\pi^{2} + n\ell_{n}^{2}.
\]
Hence for $n\ge n_{s,p}$ with $n_{s,p}$ as in \eqref{n-s-p}, $\lm_{n}^{\pm}\in S_{n}$ and $\lm_{n}^{\pm} \notin S_{k}$ for any $k\neq n$. It then follows from Lemma~\ref{eval-char} and Lemma~\ref{Sn-roots} that $\setd{\xi_{n,1,},\xi_{n,2}} = \setd{\lm_{n}^{-},\lm_{n}^{+}}$ and hence
\[
  \abs{\gm_{n}} = \abs{\lm_{n}^{+}-\lm_{n}^{-}} = \abs{\xi_{n,1}-\xi_{n,2}},\qquad \forall n\ge n_{s,p}.
\]
In view of estimate~\eqref{xi-est} of Lemma~\ref{Sn-roots}, we will need the following estimates on $b_{n}(\lm)-q_{2n}$ and $b_{-n}(\lm)-q_{-2n}$ for the proof of Theorem~\ref{fw:per}.

\begin{lem}
\label{cn-est}
Let $q\in \Ls_{0,\C}^{s,p}$ with $-1/2\le s\le 0$ and $2\le p < \infty$.
Then the following holds:
\begin{renum}
\item For any $f\in\Ls_{\star,\C}^{s,p}$ and $\lm\in S_{n}$ with $n\ge n_{s,p}$ and $n_{s,p} = n_{s,p}(q)$ as in Lemma~\ref{Sn-roots}
\[
  \w{n}^{s}\abs{\spii{T_{n}f,e_{\pm n}}}
   \le C_{s,p}\ep_{s,p}(n)\n{q}_{s,p} \n{f}_{s,p;\pm n}
\]
where $C_{s,p}$ is independent of $q$, $n$, and $\lm$, and here
\[
  \ep_{s,p}(n) = 
  \begin{cases}
  \frac{1}{n}, & 2 \le p < \frac{2}{\abs{s}},\\
  \frac{(\log \w{n})^{1-\abs{s}}}{n}, & p=\frac{2}{\abs{s}},\\
  \frac{1}{n^{1-\abs{s}+2/p}}, & \frac{2}{\abs{s}} < p < \infty.
  \end{cases}
\]
\item For any $N\ge n_{s,p}$ with $n_{s,p} = n_{s,p}(q)$ as in Lemma~\ref{Sn-roots}
\begin{align*}
  &\sum_{n\ge N} \w{n}^{sp}\p*{ \sup_{\lm\in S_{n}} \abs{b_{n}(\lm)-q_{2n}}^{p}
   + \sup_{\lm\in S_{n}} \abs{b_{-n}(\lm)-q_{-2n}}^{p} }
  \le C_{s,p}\frac{\n{q}_{s,p}^{2p}}{N^{p^{*}}},
\end{align*}
where $p^{*}\ge 1$ is given by $p^{*} = p-1$ if $2\le p < 2/\abs{s}$, $p^{*} = (p-1)-0^{+}$ if $p=2/\abs{s}$ and $p^{*} = p/2+1$ if $2/\abs{s} < p < \infty$, and the constant $C_{s,p}$ is independent of $q$ and $N$.\fish
\end{renum}

\end{lem}

\begin{proof}
(i) As the estimates of $\spii{T_{n}f,e_{n}}$ and $\spii{T_{n}f,e_{-n}}$ can be proved in a similar way we concentrate on $\spii{T_{n}f,e_{n}}$. Since by definition $T_{n} = VA_{\lm}^{-1}Q_{n}$,
\[
  \spii{T_{n}f,e_{n}} = \sum_{\abs{m}\neq n} \frac{q_{n-m}f_{m}}{\lm-m^{2}\pi^{2}}.
\]
Using that $\w{n+m}/\abs{n+m}$, $\w{n-m}/\abs{n-m}\le 2$ for $\abs{m}\neq n$ one gets $\w{n}^{s}\abs{\spii{T_{n}f,e_{n}}} \le I_{n}(\lm)$, where
\[
  I_{n}(\lm) = \sum_{\abs{m}\neq n} 
  \frac{1}{\abs{n^{2}-m^{2}}^{1-\abs{s}}}
  \frac{1}{\w{n}^{\abs{s}}}
  \frac{\abs{q_{n-m}}}{\w{n-m}^{\abs{s}}}
  \frac{\abs{f_{m}}}{\w{n+m}^{\abs{s}}}
\]
In the case $p=2$, use that $1/\abs{m^{2}-n^{2}} \le 1/n$ for all $\abs{m}\neq n$ to conclude by Cauchy--Schwarz that
\[
  I_{n}(\lm) \le \frac{1}{n}\n{q}_{s,2}\n{f}_{s,2;n}.
\]
In the case $2 < p < \infty$, apply Hölder's inequality for $p/2 > 1$ to get with $p'' = p/(p-2)>1$,
\begin{align*}
  I_{n}(\lm) \le &\frac{1}{n^{\abs{s}}}
  \p*{ \sum_{\abs{m}\neq n} \frac{1}{\abs{m^{2}-n^{2}}^{(1-\abs{s})p''} }}^{1/p''} \cdot\\
  &\qquad \cdot
     \p*{ \sum_{\abs{m}\neq n} \p*{\frac{\abs{q_{n-m}}}{\w{n-m}^{\abs{s}}}}^{p/2}
   \p*{\frac{\abs{f_{m}}}{\w{n+m}^{\abs{s}}}}^{p/2}
     }^{2/p}.
\end{align*}
By applying Cauchy--Schwarz to the latter sum one gets
\[
  I_{n}(\lm) \le
  \frac{1}{n^{\abs{s}}}
  \p*{ \sum_{\abs{m}\neq n} \frac{1}{\abs{m^{2}-n^{2}}^{(1-\abs{s})p''} }}^{1/p''}
  \n{q}_{s,p}\n{f}_{s,p;n}.
\]
Using Lemma~\ref{hilbert-sum} one sees that for some constant $C_{s,p}' > 0$,
\[
  \frac{1}{n^{\abs{s}}}
  \p*{ \sum_{\abs{m}\neq n} \frac{1}{\abs{m^{2}-n^{2}}^{(1-\abs{s})p''} }}^{1/p''}
  \le C_{s,p}'\ep_{s,p}(n),
\]
yielding the claimed statement.

(ii) Recall that for any $n\ge n_{s,p}$ and $\lm\in S_{n}$
\[
  b_{n}-q_{2n} = \spii{T_{n}u^{(n)},e_{n}},\qquad u^{(n)}\equiv u^{(n)}(\lm) = K_{n}(\lm)Ve_{-n}.
\]
Using that
\[
  \n{u^{(n)}}_{s,p;n}\le \n{K_{n}}_{s,p;n}\n{Ve_{-n}}_{s,;n} \le 2\n{q}_{s,p},
\]
it follows from item (i) that for any $\lm\in S_{n}$
\[
  \w{n}^{s}\abs{b_{n}-q_{2n}} \le 2C_{s,p}\ep_{s,p}(n)\n{q}_{s,p}^{2}.
\]
Hence for any $N\ge n_{s,p}$ and after increasing $C_{s,p}$, if necessary,
\[
  \sum_{n\ge N} \w{n}^{sp}\sup_{\lm\in S_{n}} \abs{b_{n}(\lm)-q_{2n}}^{p}
  \le C_{s,p}\n{q}_{s,p}^{2p} \sum_{n \ge N} (\ep_{s,p}(n))^{p}.
\]
Taking into account that $2\le p < \infty$ and $-1/2\le s \le 0$ one verifies that $\sum_{n\ge N} \ep_{s,p}(n)^{p} = O(1/N^{p^{*}})$.
In the same way, one shows that
\[
  \sum_{n\ge N} \w{n}^{sp}\sup_{\lm\in S_{n}} \abs{b_{-n}(\lm)-q_{-2n}}^{p}
  \le C_{s,p}\n{q}_{s,p}^{2p} \sum_{n \ge N} (\ep_{s,p}(n))^{p}.\qed
\]
\end{proof}

\begin{proof}[Proof of Theorem~\ref{fw:per}.]
By Lemma~\ref{Sn-roots},
\begin{align*}
  \abs{\gm_{n}}^{p} = \abs{\xi_{n,1}-\xi_{n,2}}^{p} 
  &\le \sup_{\lm\in S_{n}} \abs{6b_{n}(\lm)b_{-n}(\lm)}^{p/2}\\
  &\le \frac{3^{p}}{2}\p*{\sup_{\lm\in S_{n}} \abs{b_{n}(\lm)}^{p} + \sup_{\lm\in S_{n}}\abs{b_{-n}(\lm)}^{p}}.
\end{align*}
It then follows from Lemma~\ref{cn-est} that with $N\defl n_{s,p}$
\begin{align*}
  \sum_{n\ge N} \w{n}^{sp}\abs{\gm_{n}}^{p} 
  &\le \frac{6^{p}}{2} \sum_{n\ge N} \w{n}^{sp}
  \biggl( \abs{q_{2n}}^{p} + \abs{q_{-2n}}^{p} \\
  &\qquad + \sup_{\lm\in S_{n}} \abs{b_{n}(\lm) - q_{2n}}^{p}
          + \sup_{\lm\in S_{n}} \abs{b_{-n}(\lm) - q_{-2n}}^{p} \biggr)\\
  &\le \frac{6^{p}}{2}\n{R_{2N}}_{s,p}^{p} + \frac{6^{p}}{2} C_{s,p}\frac{1}{N}\n{q}_{s,p}^{2p}.
\end{align*}
In particular, $(\gm_{n}(q))_{n\ge 1}\in\ell_{\C}^{s,p}(\N)$. As $n_{s,p}$ can be chosen locally uniformly in $q\in \Ls_{0,\C}^{s,p}$, the map $\Ls_{0,\C}^{s,p} \to \ell_{\C}^{s,p}(\N)$, $q\mapsto (\gm_{n}(q))_{n\ge 1}$ is locally bounded.~\qed
\end{proof}

\subsection{Jordan blocks of $L(q)$}

\label{ss:jordan}

To treat the Dirichlet problem we need to analyze $L(q)$ further. We again consider potentials in the space $\Ls_{0,\C}^{s,p}$ with $-1/2\le s \le 0$ and $2\le p < \infty$.  In case $q\in\Ls_{0,\C}^{s,p}$ is not real valued, the operator $L(q)$ might have complex eigenvalues and it might happen that the geometric multiplicity of an eigenvalue is less than its algebraic multiplicity. 
For $n\ge n_{s,p}$ with $n_{s,p}$ as in Lemma~\ref{Sn-roots}, let
\[
  E_{n} = \begin{cases}
  \Null(L-\lm_{n}^{+})\oplus\mathrm{Null}(L-\lm_{n}^{-}), &
  \lm_{n}^{+}\neq \lm_{n}^{-},\\
  \Null(L-\lm_{n}^{+})^{2}, & \lm_{n}^{+} = \lm_{n}^{-}.
  \end{cases}
\]
We need to estimate the coefficients of $L(q)\big|_{E_{n}}$ when represented with respect to an appropriate orthonormal basis of $E_{n}$. In the case where $\lm_{n}^{+} = \lm_{n}^{-}$ the matrix representation will be in Jordan normal form. By Lemma~\ref{regularity-En}, $E_{n}\subset\Ls_{\star,\C}^{s+2,p}\opento L^{2} \defl L^{2}([0,2],\C)$. Denote by $f_{n}^{+}\in E_{n}$ an $L^{2}$-normalized eigenfunction corresponding to $\lm_{n}^{+}$ and by $\ph_{n}$ an $L^{2}$-normalized element in $E_{n}$ so that $\setd{f_{n}^{+},\ph_{n}}$ forms an $L^{2}$-orthonormal basis of $E_{n}$. 
Then the following lemma holds.

\begin{lem}
\label{En-basis}
Let $q\in\Ls_{0,\C}^{s,p}$ with $-1/2\le s\le 0$ and $2\le p < \infty$. Then there exists $n_{s,p}'\ge n_{s,p}$ so that for any $n\ge n_{s,p}'$,
\[
  (L-\lm_{n}^{+})\ph_{n} = -\gm_{n}\ph_{n} + \eta_{n}f_{n}^{+},
\]
where $\eta_{n}\in \C$ satisfies the estimate
\[
  \abs{\eta_{n}} \le 16(\abs{\gm_{n}} + \abs{b_{n}(\lm_{n}^{+})} + \abs{b_{-n}(\lm_{n}^{+})}).
\]
The threshold $n_{s,p}'$ can be chosen locally uniformly in $q\in\Ls_{0,\C}^{s,p}$.\fish
\end{lem}

\begin{proof}
We begin by verifying the claimed formula for $(L-\lm_{n}^{+})\ph_{n}$ in the case where $\lm_{n}^{+}\neq \lm_{n}^{-}$. Let $f_{n}^{-}$ be an $L^{2}$-normalized eigenfunction corresponding to $\lm_{n}^{-}$. As $f_{n}^{-}\in E_{n}$ there exist $a,b\in\C$ with $\abs{a}^{2}+\abs{b}^{2} = 1$ and $b\neq0$ so that
\[
  f_{n}^{-} = a f_{n}^{+} + b\ph_{n}
  \quad\text{or}\quad
  \ph_{n} = \frac{1}{b}f_{n}^{-} - \frac{a}{b}f_{n}^{+}.
\]
Hence
\[
  L\ph_{n} = \frac{1}{b}\lm_{n}^{-}f_{n}^{-} - \frac{a}{b}\lm_{n}^{+}f_{n}^{+}.
\]
Substituting the expression for $f_{n}^{-}$ into the latter identity then leads to
\[
  (L-\lm_{n}^{+})\ph_{n} = (\lm_{n}^{-}-\lm_{n}^{+})\ph_{n} + \frac{a}{b}(\lm_{n}^{-}-\lm_{n}^{+})f_{n}^{+} = -\gm_{n}\ph_{n} + \eta_{n}f_{n}^{+}
\]
where $\eta_{n} = -\gm_{n}a/b$. In the case $\lm_{n}^{+}$ is a double eigenvalue of geometric multiplicity two, $\ph_{n}$ is an eigenfunction of $L$ and one has $\eta_{n} = 0$. Finally, in the case $\lm_{n}^{+}$  is a double eigenvalue of geometric multiplicity one, $(L-\lm_{n}^{+})\ph_{n}$ is in the eigenspace $E_{n}^{+}\subset E_{n}$ as claimed.

To prove the claimed estimate for $\eta_{n}$, we view $(L-\lm_{n}^{+})\ph_{n} = -\gm_{n}\ph_{n} + \eta_{n}f_{n}^{+}$ as a linear equation with inhomogeneous term $g = -\gm_{n}\ph_{n} + \eta_{n}f_{n}^{+}$. By identity~\eqref{B-inhom} one has
\[
  B_{n}P_{n}\ph_{n} = \gm_{n} P_{n}K_{n}\ph_{n} - \eta_{n}P_{n}K_{n}f_{n}^{+},
\]
where $K_{n} \equiv K_{n}(\lm_{n}^{+})$ and $B_{n} \equiv B_{n}(\lm_{n}^{+})$. To estimate $\eta_{n}$, take the $L^{2}$-inner product of the latter identity with $P_{n}f_{n}^{+}$ to get
\begin{equation}
  \label{eta-I-II}
  \eta_{n}\spii{P_{n}K_{n}f_{n}^{+},P_{n}f_{n}^{+}} = \gm_{n}I - II,
\end{equation}
where
\[
  I = \spii{P_{n}K_{n}\ph_{n}, P_{n}f_{n}^{+}},\qquad
  II = \spii{B_{n}P_{n}\ph_{n},P_{n}f_{n}^{+}}.
\]

We begin by estimating $\spii{P_{n}K_{n}f_{n}^{+},P_{n}f_{n}^{+}}$. Using that $K_{n} = \Id + T_{n}K_{n}$ one gets
\begin{align*}
  \spii{P_{n}K_{n}f_{n}^{+},P_{n}f_{n}^{+}}
   = \n{P_{n}f_{n}^{+}}_{L^2}^{2}
    + \spii{T_{n}K_{n}f_{n}^{+},P_{n}f_{n}^{+}},
\end{align*}
and by Cauchy-Schwarz
\[
  \abs{\spii{T_{n}K_{n}f_{n}^{+},P_{n}f_{n}^{+}}} \le
  \p[\bigg]{\sum_{m\in \setd{\pm n}} \abs{\spii{T_{n}K_{n}f_{n}^{+},e_{m}}}^{2}}^{1/2}\n{P_{n}f_{n}^{+}}_{L^{2}}.
\]
Note that $\n{P_{n}f_{n}^{+}}_{L^{2}} \le \n{f_{n}^{+}}_{L^{2}} = 1$ and further by Lemma~\ref{cn-est} (i)
\[
  \abs{\spii{T_{n}K_{n}f_{n}^{+},e_{\pm n}}}
   \le \w{n}^{\abs{s}}C_{s,p}\ep_{s,p}(n)\n{q}_{s,p}\n{K_{n}f_{n}^{+}}_{s,p;\pm n}.
\]
By Corollary~\ref{Q-soln}, $\n{K_{n}}_{s,p;n} \le 2$ and as $L_{\C}^{2}[0,2] \opento \Ls_{\star,\C}^{s,p}$, $\n{f_{n}^{+}}_{s,p;\pm n}\le \n{f_{n}^{+}}_{0,2;\pm n} =1$, whereas by the definition of $\ep_{s,p}(n)$ in Lemma~\ref{cn-est} (i), $n^{\abs{s}}\ep_{s,p}(n) \le C_{s,p}/n^{1/p}$.  Hence there exists $n_{s,p}'\ge n_{s,p}$ so that
\[
  \abs{\spii{T_{n}K_{n}f_{n}^{+},P_{n}f_{n}^{+}}} \le \frac{1}{8}.
\]
By increasing $n_{s,p}'$ if necessary, Lemma~\ref{Pn-est} below assures that $\n{P_{n}f_{n}^{+}}_{L^2} \ge 1/2$. Thus the left hand side of \eqref{eta-I-II} can be estimated as follows
\begin{equation}
  \label{eta-est-1}
  \abs*{\eta_{n}\spii{P_{n}K_{n}f_{n}^{+},P_{n}f_{n}^{+}}} \ge \abs{\eta_{n}}
  \p*{ \frac{1}{4}-\frac{1}{8} } = \frac{1}{8}\abs{\eta_{n}},
  \qquad
  \forall n\ge n_{s,p}'.
\end{equation}

Next let us estimate the term $I = \spii{P_{n}K_{n}\ph_{n}, P_{n}f_{n}^{+}}$ in \eqref{eta-I-II}. Using again $K_{n} = \Id + T_{n}K_{n}$ one sees that
\[
  I = \spii{P_{n}\ph_{n}, P_{n}f_{n}^{+}} + \spii{T_{n}K_{n}\ph_{n}, P_{n}f_{n}^{+}}.
\]
Clearly, $\abs{\spii{P_{n}\ph_{n}, P_{n}f_{n}^{+}}} \le \n{\ph_{n}}_{L^2}\n{f_{n}^{+}}_{L^2} \le 1$ and arguing as above for the second term, one then concludes that
\begin{equation}
  \label{eta-est-2}
  \abs{I} \le 1 + 1/8,\qquad \forall n\ge n_{s,p}'.
\end{equation}

Finally it remains to estimate $II = \spii{B_{n}P_{n}\ph_{n},P_{n}f_{n}^{+}}$. Using again $\n{\ph_{n}}_{L^{2}}  = \n{f_{n}^{+}}_{L^{2}} = 1$, we conclude from the matrix representation \eqref{Bn-mat} of $B_{n}$ that
\[
  \abs{\spii{B_{n}P_{n}\ph_{n},P_{n}f_{n}^{+}}} \le \n{B_{n}}\n{\ph_{n}}_{L^{2}}\n{f_{n}^{+}}_{L^{2}}
  \le \abs{\lm_{n}^{+}-n^{2}\pi^{2}-a_{n}} + \abs{b_{n}}+ \abs{b_{-n}}.
\]
Since $\det B_{n}(\lm_{n}^{+}) = 0$, one has
\[
  \abs{\lm_{n}^{+}-n^{2}\pi^{2}-a_{n}} = \abs{b_{n}b_{-n}}^{1/2} \le \frac{1}{2}(\abs{b_{n}} + \abs{b_{-n}}),
\]
and hence it follows that for all $n\ge n_{s,p}'$ that
\begin{equation}
  \label{eta-est-3}
  \abs{II} \le 2(\abs{b_{n}}+\abs{b_{-n}}).
\end{equation}
Combining \eqref{eta-est-1}-\eqref{eta-est-3} leads to the claimed estimate for $\eta_{n}$.\qed
\end{proof}

It remains to prove the estimate of $P_{n}$ used in the proof of Lemma~\ref{En-basis}. To this end, we introduce the Riesz projector $P_{n,q}\colon L^{2}\to E_{n}$ given by (see also Appendix~\ref{a:hill-op})
\[
  P_{n,q} = \frac{1}{2\pi \ii}\int_{\abs{\lm -n^{2}\pi^{2}} = n} (\lm - L(q))^{-1}\,\dlm.
\]

\begin{lem}
\label{Pn-est}
Let $q\in\Ls_{0,\C}^{s,p}$ with $-1/2\le s \le 0$ and $2\le p < \infty$. Then there exists $\tilde n_{s,p} \ge n_{s,p}$ so that for any eigenfunction $f\in \Ls_{\star,\C}^{s+2,p}$ of $L(q)$ corresponding to an eigenvalue $\lm\in S_{n}$ with $n\ge \tilde n_{s,p}$,
\[
  \n{P_{n}f}_{L^2}\ge \frac{1}{2}\n{f}_{L^2}.\fish
\]
\end{lem}

\begin{proof}
In Lemma~\ref{proj-bound} we show that
\[
  \n{P_{n,q}-P_{n}}_{L^{2}\to L^{\infty}} = o(1),
\]
locally uniformly in $q\in\Ls_{0,\C}^{s,p}$ and uniformly as $n\to \infty$.
Clearly, $P_{n,q}f = f$, hence
\[
  \n{P_{n}f}_{L^{2}} \ge \n{P_{n,q}f}_{L^{2}} - \n{(P_{n,q}-P_{n})f}_{L^{2}}
  \ge \p[\big]{1+o(1)}\n{f}_{L^{2}}.\qed
\]
\end{proof}


\subsection{Proof of Theorem~\ref{fw:dir}}

We begin with a brief outline of the proof of Theorem~\ref{fw:dir}. Let $q\in\Ls_{0,\C}^{s,p}$ with $-1/2\le s\le 0$ and $2\le p < \infty$. Since for any $q\in H_{0,\C}^{-1}$ the Dirichlet eigenvalues, when listed in lexicographical ordering and with their algebraic multiplicities, $\mu_{1}\lex \mu_{2}\lex \dotsb$, satisfy the asymptotics $\mu_{n} = n^{2}\pi^{2} + n\ell_{n}^{2}$, 
they are simple for $n\ge n_{\dir}$, where $n_{\dir}\ge 1$ can be chosen locally uniformly for $q\in H_{0,\C}^{-1}$. For any $n\ge n_{\dir}$ let $g_{n}$ be an $L^{2}$-normalized eigenfunction corresponding to $\mu_{n}$. Then
\[
  g_{n}\in H_{\dir,\C}^{1}\defl \setdef{g\in H^{1}([0,1],\C)}{g(0) = g(1) = 0}.
\]
Now let $q\in\Ls_{0,\C}^{s,p}$ with $-1/2\le s\le 0$ and $2\le p < \infty$. Increase $n_{s,p}'$ of Lemma~\ref{En-basis}, if necessary, so that $n_{s,p}'\ge n_{\dir}$ and denote by $E_{n}$ the two dimensional subspace introduced in Section~\ref{ss:jordan}.
We will choose an $L^{2}$-normalized function $\tilde G_{n}$ in $E_{n}$ so that its restriction $G_{n}$ to $\Ic = [0,1]$ is in $H_{\dir,\C}^{1}$ and close to $g_{n}$. We then show that $\mu_{n}-\lm_{n}^{+}$ can be estimated in terms of $\spi{(L_{\dir}-\lm_{n}^{+})G_{n},G_{n}}$. As by Lemma~\ref{En-basis}
\[
  (L-\lm_{n}^{+})\tilde G_{n} = O(\abs{\gm_{n}} + \abs{b_{n}(\lm_{n}^{+})} + \abs{b_{-n}(\lm_{n}^{+})}),
\]
the claimed estimates for $\mu_{n}-\tau_{n} = \mu_{n} - \lm_{n}^{+}+\gm_{n}/2$ then follow from the estimates of $\gm_{n}$ of Theorem~\ref{fw:per} and the ones of $b_{n}-q_{2n}$, $b_{-n}-q_{-2n}$ of Lemma~\ref{cn-est} (ii).

The function $\tilde G_{n}$ is defined as follows. Let $f_{n}^{+}$, $\ph_{n}$ be the $L^{2}$-orthonormal basis of $E_{n}$ chosen in Section~\ref{ss:jordan}. As $E_{n}\subset H_{\C}^{1}(\R/2\Z)$, its elements are continuous functions by the Sobolev embedding theorem. If $f_{n}^{+}(0) = 0$, then $f_{n}^{+}(1) = 0$ as $f_{n}^{+}$ is an eigenfunction of the $1$-periodic/antiperiodic eigenvalue $\lm_{n}^{+}$ of $L(q)$ and we set $\tilde G_{n} = f_{n}^{+}$.
If $f_{n}^{+}(0)\neq 0$, then we define $\tilde G_{n}(x) = r_{n}\bigl(\ph_{n}(0)f_{n}^{+}(x) - f_{n}^{+}(0)\ph_{n}(x)\bigr)$, where $r_{n} > 0$ is chosen in such a way that $\int_{0}^{1} \abs{\tilde G_{n}(x)}^{2}\,\dx = 1$. Then $\tilde G_{n}(0) = \tilde G_{n}(1) = 0$ and since $\tilde G_{n}$ is an element of $E_{n}$ its restriction $G_{n} \defl \tilde G_{n}\big|_{\Ic}$ is in $H_{\dir,\C}^{1}$.

Denote by $\Pi_{n,q}$ the Riesz projection, introduced in Appendix~\ref{a:hill-op},
\[
  \Pi_{n,q} \defl \frac{1}{2\pi \ii}\int_{\abs{\lm-n^{2}\pi^{2}} = n} (\lm-L_{\dir}(q))^{-1}\,\dlm.
\]
It has $\spanx(g_{n})$ as its range and hence $\Pi_{n,q}G_{n} = \nu_{n}g_{n}$ for some $\nu_{n}\in \C$.

\begin{lem}
\label{gn-nu-est}
Let $q\in \Ls_{0,\C}^{s,p}$ with $-1/2\le s \le 0$ and $2\le p < \infty$. Then there exists $n_{s,p}''\ge n_{s,p}'$ with $n_{s,p}'$ as in Lemma~\ref{En-basis} so that for any $n\ge n_{s,p}''$
\begin{renum}
\item the following identity holds
\begin{equation}
  \label{nu-eqn}
  \nu_{n}(\mu_{n}-\lm_{n}^{+})g_{n}
   = \bt_{n}\p*{ \eta_{n}\Pi_{n,q}(f_{n}^{+}\big|_{\Ic}) - \gm_{n}\Pi_{n,q}(\ph_{n}\big|_{\Ic}) },
\end{equation}
where $\bt_{n} \in \C$ with $\abs{\bt_{n}} \le 1$ and $\eta_{n}$ is the off-diagonal coefficient in the matrix representation of $(L-\lm_{n}^{+})\big|_{E_{n}}$ with respect to the basis $\setd{f_{n}^{+},\ph_{n}}$, introduced in Lemma~\ref{En-basis}, and
\item $1/2 \le \abs{\nu_{n}} \le 3/2$.
\end{renum}
$n_{s,p}''$ can be chosen locally uniformly for $q\in\Ls_{0,\C}^{s,p}$.\fish
\end{lem}

\begin{proof}
(i) Write $G_{n} = \nu_{n}g_{n} + h_{n}$, where $h_{n} = (\Id - \Pi_{n,q})G_{n}$. Then
\[
  (L_{\dir}-\lm_{n}^{+})G_{n} = \nu_{n}(\mu_{n}-\lm_{n}^{+})g_{n} + (L_{\dir}-\lm_{n}^{+})h_{n}.
\]
On the other hand, $G_{n} = \tilde G_{n}\big|_{\Ic}$, where $\tilde G_{n}\in E_{n}$ is given by $\tilde G_{n} = \al_{n}f_{n}^{+} + \bt_{n}\ph_{n}$ with $\al_{n}$, $\bt_{n}\in\C$ satisfying $\abs{\al_{n}}^{2} + \abs{\bt_{n}}^{2} = 1$ and $G_{n}\in H_{\dir,\C}^{1}$. Hence by Lemma~\ref{V-dir-V-per} and Lemma~\ref{En-basis}, for $n\ge n_{s,p}'$,
\[
  (L_{\dir}-\lm_{n}^{+})G_{n} = (L-\lm_{n}^{+})\tilde G_{n}\big|_{\Ic}
   = \bt_{n}(\eta_{n}f_{n}^{+} - \gm_{n}\ph_{n})\big|_{\Ic}.
\]
Combining the two identities and using that $\Pi_{n,q}h_{n} = 0$ and that $\Pi_{n,q}$ commutes with $(L_{\dir}-\lm_{n}^{+})$, one obtains, after projecting onto $\spanx(g_{n})$, identity~\eqref{nu-eqn}.

(ii) Taking the inner product of $\Pi_{n,q}G_{n} = \nu_{n} g_{n}$ with $g_{n}$ one gets
\[
  \nu_{n} = \nu_{n}\spi{g_{n},g_{n}} = \spi{\Pi_{n,q}G_{n},g_{n}}.
\]
Let $s_{n}(x) = \sqrt{2}\sin(n\pi x)$ and denote by $\Pi_{n} = \Pi_{n,0}$ the orthogonal projection onto $\spanx\setd{s_{n}}$. Recall that $P_{n,q}\colon L^{2}\to E_{n}$ is the Riesz projection onto $E_{n}$. In Lemma~\ref{proj-bound} we show that
\begin{equation}
  \label{proj-est-1}
  \n{\Pi_{n,q}-\Pi_{n}}_{L^{2}(\Ic)\to L^{\infty}(\Ic)},\;
  \n{P_{n,q}-P_{n}}_{L^{2}\to L^{\infty}} = o(1),
\end{equation}
locally uniformly in $q\in\Ls_{0,\C}^{s,p}$ and uniformly as $n\to \infty$.
Thus using $\Pi_{n}G_{n} = \Pi_{n}(P_{n}\tilde G_{n})\big|_{\Ic}$ and recalling that $\n{G_{n}}_{L^{2}(\Ic)}^{2} = \n{g_{n}}_{L^{2}(\Ic)}^{2} = 1$ we obtain
\[
  v_{n}
   = \spi{\Pi_{n}G_{n},g_{n}} + \spi{(\Pi_{n,q}-\Pi_{n})G_{n},g_{n}}
   = \spi{P_{n}\tilde G_{n},\Pi_{n}g_{n}} + o(1).
\]
Moreover, it follows from \eqref{proj-est-1} that uniformly in $0\le x \le 1$
\[
  \Pi_{n}g_{n}(x) = \e^{\ii \phi_{n}}s_{n}(x) + o(1),\qquad n\to \infty,
\]
with some real $\phi_{n}$. Similarly, again by \eqref{proj-est-1}, uniformly in $0\le x \le 2$
\[
  P_{n}\tilde G_{n}(x) = a_{n}e_{n}(x) + b_{n}e_{-n}(x) + o(1),\qquad n\to \infty,
\]
where, since $\n{G_{n}}_{L^{2}} = 1$ and $G_{n}(0) = 0$, the coefficients $a_{n}$ and $b_{n}$ can be chosen so that
\[
  \abs{a_{n}}^{2} + \abs{b_{n}}^{2} = 1,\quad a_{n} + b_{n} = 0.
\]
That is $P_{n}\tilde G_{n}(x) = \e^{\ii \psi_{n}}s_{n}(x) + o(1)$ with some real $\psi_{n}$ and hence
\[
  \spi{P_{n}\tilde G_{n},\Pi_{n}g_{n}}
   = \e^{\ii\psi_{n}-\ii\phi_{n}}\spi{s_{n},s_{n}} + o(1)
   = \e^{\ii\psi_{n}-\ii\phi_{n}} + o(1),\qquad n\to \infty.
\]
From this we conclude
\[
  \abs{\nu_{n}} = 1 + o(1),\qquad n\to\infty.
\]
Therefore, $1/2 \le \abs{\nu_{n}} \le 3/2$ for all $n\ge n_{s,p}''$ provided $n_{s,p}'' \ge n_{s,p}'$ is sufficiently large.

Going through the arguments of the proof one verifies that $n_{s,p}''$ can be chosen locally uniformly in $q$.\qed
\end{proof}

Lemma~\ref{gn-nu-est} allows to complete the proof of Theorem~\ref{fw:dir}.

\begin{proof}[Proof of Theorem~\ref{fw:dir}.]
Take the inner product of~\eqref{nu-eqn} with $g_{n}$ and use that $\abs{\nu_{n}} \ge 1/2$ by Lemma~\ref{gn-nu-est} to conclude that
\begin{equation}
  \label{mu-est-1}
  \frac{1}{2}\abs{\mu_{n}-\lm_{n}^{+}} \le \abs{\bt_{n}}
  \p*{ \abs{\eta_{n}} \spi{\Pi_{n,q}(f_{n}^{+}\big|_{\Ic}),g_{n}}
  + \abs{\gm_{n}}\abs{\spi{\Pi_{n,q}(\ph_{n}\big|_{\Ic})},g_{n}} }.
\end{equation}
Recall that $\abs{\bt_{n}} \le 1$ and note that for any $f,g\in L_{\C}^{2}(\Ic)$
\[
  \abs{\spi{\Pi_{n,q}f,g}} \le 
  \abs{\spi{\Pi_{n}f,g}} +
  \abs{\spi{(\Pi_{n,q}-\Pi_{n})f,g}} \le 
  (1+o(1))\n{f}_{L^{2}(\Ic)}\n{g}_{L^{2}(\Ic)}.
\]
Since $\n{f_{n}^{+}}_{L^{2}(\Ic)} = \n{\ph_{n}}_{L^{2}(\Ic)} = 1$ and $\n{g_{n}}_{L^{2}(\Ic)} = 1$, \eqref{mu-est-1} implies that
\[
  \abs{\mu_{n}-\lm_{n}^{+}} \le (2+o(1))(\abs{\eta_{n}}+\abs{\gm_{n}})
\]
yielding with Lemma~\ref{En-basis} the estimate
\[
  \abs{\mu_{n}-\tau_{n}} \le  (3+o(1))\abs{\gm_{n}}
  + (32+o(1))(\abs{\gm_{n}} + \abs{b_{n}(\lm_{n}^{+})}+
  \abs{b_{-n}(\lm_{n}^{+})}).
\]
By Theorem~\ref{fw:per} and Lemma~\ref{cn-est} (ii) it then follows that $(\tau_{n}-\mu_{n})_{n\ge 1}\in \ell_{\C}^{s,p}(\N)$. Going through the arguments of the proof one verifies that the map 
$\Ls_{0,\C}^{s,p}\to \ell_{\C}^{s,p}(\N)$, 
$q\mapsto (\tau_{n}-\mu_{n})_{n\ge 1}$ is locally bounded.\qed
\end{proof}


\section{Proof of Theorem~\ref{main}}
\label{section4}

The proof of Theorem~\ref{main} relies on a formula for the nonlinear part of the KdV Hamiltonian which we derive in Subsection~\ref{ss:4-F-exp}. The formula is based on the integral
\[
  F(\lm) = \int_{\lm_{0}^{+}}^{\lm} \frac{\dDl}{\sqrt[c]{\Dl^{2}-4}}\,\dlm,
\]
whose properties are studied in Subsection~\ref{ss:4-om}. The proof of Theorem~\ref{main} is then completed in Subsection~\ref{ss:4-pf}.
In the sequel we use without further reference the notation introduced in Section~\ref{section1} and Section~\ref{section2}.

\subsection{Properties of $F$}
\label{ss:4-om}

Let $q$ be in $\Ws$. It is convenient to define the \emph{standard root}
\[
  \vs_{n}(\lm) = \sqrt[\mathrm{s}]{(\lm_{n}^{+}-\lm)(\lm_{n}^{-}-\lm)},
                 \qquad \lm\in\C\setminus G_{n},\qquad n\ge 1,
\]
by the condition
\begin{align}
  \label{s-root}
  \vs_{n}(\lm) = (\tau_{n}-\lm)\sqrt[+]{1 - \gm_{n}^{2}/4(\tau_{n}-\lm)^{2}},
  						 \qquad \tau_{n} = (\lm_{n}^{-}+\lm_{n}^{+})/2.
\end{align}
Here $\sqrt[+]{\phantom{a}}$ denotes the principal branch of the square root on the complex plane minus the ray $(-\infty,0]$. The standard root is analytic in $\lm$ on $\C\setminus G_{n}$ and in $(\lm,p)$ on $(\C\setminus \ob{U_{n}})\times \Wp_{q}$. The \emph{canonical root} $\sqrt[c]{\Dl^{2}(\lm)-4}$, introduced in Section~\ref{section2}, can then be written in terms of standard roots as follows
\begin{equation}
  \label{c-root}
  \sqrt[c]{\Dl^{2}(\lm)-4} \defl 
   -2\ii\sqrt[+]{\lm-\lm_{0}^{+}}\prod_{m\ge 1} \frac{\vs_{m}(\lm)}{m^{2}\pi^{2}}
\end{equation}
and is analytic in $\lm$ on $\C\setminus\bigcup_{\gm_{n}\neq 0} G_{n}$ and in $(\lm,p)$ on $(\C\setminus \bigcup_{n\ge 0} \ob{U_{n}})\times \Wp_{q}$.

A path in the complex plane is said to be \emph{admissible} for $q$ if, except possibly at its endpoints, it does not intersect any non collapsed gap $G_{n}(q)$.

\begin{lem}
\label{w-closed}
For each $q\in \Wp$ the following holds:
\begin{renum}
\item
$\frac{\dDl(\lm)}{\sqrt[c]{\Dl^{2}(\lm)-4}}$ is analytic in $(\lm,p)$ on $(\C\setminus \bigcup_{n\ge 0} \ob{U_{n}}) \times \Wp_{q}$ and analytic in $\lm$ on $\C\setminus \bigcup_{\atop{\gm_{n}\neq 0}{n\ge 0}} G_{n}$. (We recall that $\gm_{0} = \infty$.)
\item
For any $n\ge 1$ and any admissible path from $\lm_{n}^{-}$ to $\lm_{n}^{+}$ in $U_{n}$,
$
  \int_{\lm_{n}^{-}}^{\lm_{n}^{+}} \frac{\dDl(\lm)}{\sqrt[c]{\Dl^{2}(\lm)-4}}\,\dlm =  0.
$
As a consequence, for any closed circuit $\Gm_{n}$ in $U_{n}$ around $G_{n}$,
$
  \int_{\Gm_{n}} \frac{\dDl(\lm)}{\sqrt[c]{\Dl^{2}(\lm)-4}}\,\dlm = 0.
$
\item
For $q = 0$, $\frac{\dDl(\lm)}{\sqrt[c]{\Dl^{2}(\lm)-4}} = \frac{1}{2\ii \sqrt[+]{\lm}}$.\fish
\end{renum}
\end{lem}

\begin{proof}
Standard -- see e.g. \cite[Lemma 1]{Molnar:_ROURXz4}.\qed
\end{proof}

Next we define for any $q\in\Wp$ and $\lm\in \C\setminus \bigcup_{\atop{\gm_{n}\neq 0}{n\ge 0}} G_{n}$ the improper integral
\begin{equation}
  \label{F-prod}
    F(\lm) \defl \int_{\lm_{0}^{+}}^{\lm} \frac{\dDl(z)}{\sqrt[c]{\Dl^{2}(z)-4}}\,\dz,
  \qquad
  \frac{\dDl(\lm)}{\sqrt[c]{\Dl^{2}(\lm)-4}}
   = \frac{1}{2\ii}\frac{1}{\sqrt[+]{\lm-\lm_{0}^{+}}}\prod_{m\ge 1}
     \frac{\lm_{m}^{\ld}-\lm}{\vs_{m}(\lm)},
\end{equation}
computed along an arbitrary admissible path. The improper integral $F(\lm)$ exists as in the product representation~\eqref{F-prod} the factor ${1}/{\sqrt[+]{\lm-\lm_{0}^{+}}}$ is integrable on $\C\setminus G_{0}$. Furthermore, by Lemma~\ref{w-closed} it is independent of the chosen admissible path and hence well defined. Moreover, $F(\lm)$ continuously extends to $G_{n}^{\pm}$ for any $n\ge 0$, where $G_{n}^{+}$ [$G_{n}^{-}$] is the left [right] hand side of $G_{n}$.

\begin{lem}
\label{F-prop}
For any $q\in \Wp$ the following holds:

\begin{renum}
\item $F$ is analytic in $(\lm,p)$ on $(\C\setminus\bigcup_{n\ge 0} \ob{U_{n}}) \times \Wp_{q}$ and $F(\lm)\equiv F(\lm,q)$ is analytic in $\lm$ on $\C\setminus\bigcup_{\atop{\gm_{n}\neq 0}{n\ge 0}} G_{n}$. (We recall that $\gm_{0} = \infty$.)

\item $F(\lm_{0}^{+}) = 0$ and $F(\lm_{n}^{+}) = F(\lm_{n}^{-}) = -\ii n\pi$ for any $n\ge 1$.

\item 
Locally uniformly on $\Wp_{q}$
\[
  \sup_{\lm\in G_{n}^{+}\cup G_{n}^{-}}\abs{F(\lm)+\ii n\pi} = O(\gm_{n}/n),\qquad n\to\infty.
\]

\item For $q = 0$, $F(\lm)$ is analytic on $\C\setminus(-\infty,0]$ and given by $F(\lm) = -\ii \sqrt[+]{\lm}$.

\end{renum}

\end{lem}

\begin{proof}
Items (i), (ii), and (iv) are well known. See e.g. \cite[Lemma 2]{Molnar:_ROURXz4} for a proof.

(iii) 
To obtain the claimed estimate for $F(\lm)+\ii n\pi = \int_{\lm_{n}^{-}}^{\lm} \frac{\dDl(\lm)}{\sqrt[c]{\Dl^{2}(\lm)-4}}\,\dlm $ on $G_{n}^{+}\cup G_{n}^{-}$ it suffices to consider $p\in \Wp_{q}$ with $\gm_{n}(p)\neq0$. Parameterizing $G_{n}^{\pm}$ by $\lm_{s}^{\pm} = \tau_{n} + (s\pm \ii 0)\gm_{n}/2$, $-1\le s\le 1$, one gets for $-1\le t\le 1$
\[
  F(\lm_{t}^{\pm}) + \ii n\pi 
  = \int_{\lm_{n}^{-}}^{\lm_{t}^{\pm}} \frac{\dDl}{\sqrt[c]{\Dl^{2}-4}}\,\dz
  = \mp \int_{-1}^{t} \frac{\lm_{n}^{\ld}-\tau_{n} - s\gm_{n}/2}{\ii\sqrt[+]{1-s^{2}}}\chi_{n}(\lm_{s})\,\ds,
\]
where
\begin{equation}
\label{chi.n.def} 
  \chi_{n}(\lm) \defl 
  \frac{1}{2\ii}
  \frac{1}{\sqrt[+]{\lm-\lm_{0}^{+}}}
  \prod_{1\le m\neq n}\frac{\lm_{m}^{\ld}-\lm}{\vs_{m}(\lm)}
\end{equation}
The map $\chi_{n}$ is analytic on $U_{n}\times \Wp_{q}$ and $n\chi_{n}|_{G_{n}} = O(1)$ as $n\to\infty$ locally uniformly on $\Wp_{q}$ -- see \cite[Lemma 9.7]{Kappeler:2005fb}.
Moreover, by \cite[Proposition 2.18]{Kappeler:2005fb} one can choose $n_{0}\ge 1$ locally uniformly on $\Wp_{q}$ such that $\abs{\lm_{n}^{\ld}-\tau_{n}}\le\abs{\gm_{n}}$ for $n\ge n_{0}$. Altogether this shows
\[
  \sup_{\lm\in G_{n}^{\pm}} \abs{F(\lm) + \ii n\pi} = O(\abs{\gm_{n}}/n)
\]
locally uniformly on $\Wp_{q}$.\qed
\end{proof}

Occasionally we write for $n\ge 0$
\begin{equation}
  \label{F.n.int.formula}
  F_n(\lm) \defl F(\lm) + \ii n\pi = \int_{\lm_n^+}^{\lm} \frac{\dDl}{\sqrt[c]{\Dl^{2}-4}}\,\dz
\end{equation}
to denote the primitive of $\frac{\dDl}{\sqrt[c]{\Dl^{2}-4}}$ normalized by the condition $F_n(\lm_n^+) = 0$.

\begin{lem}
\label{F2.analytic}
For any $q\in \Wp$ and any $n\ge 0$, $F_{n}^{2}$ is analytic on $U_{n}$ and hence on $\C\setminus\bigcup_{\atop{n\neq m\ge 0}{\gm_{m}\neq 0}} G_{m}$.
\end{lem}

\begin{proof}
Let $q\in \Wp$ and $n\ge 0$ be arbitrary. It remains to show that $F_{n}^{2}$ admits an analytic extension from $U_{n}\setminus G_{n}$ to all of $U_{n}$. In view of \eqref{F-prod} we can write
\[
  \frac{\dDl(\lm)}{\sqrt[c]{\Dl^{2}(\lm)-4}}
   = \frac{1}{\sqrt[+]{\lm-\lm_{0}^{+}}}\chi_{0}(\lm),\qquad
   \chi_{0}(\lm) = \frac{1}{2\ii}
  \prod_{m\ge 1}\frac{\lm_{m}^{\ld}-\lm}{\vs_{m}(\lm)},
\]
and for any $k\ge 1$,
\[
   \frac{\dDl(\lm)}{\sqrt[c]{\Dl^{2}(\lm)-4}}
   = \frac{\lm_{k}^{\ld}-\lm}{\vs_{k}(\lm)}\chi_{k}(\lm),
   \qquad\quad
   \chi_{k}(\lm) = 
   \frac{1}{2\ii}
   \frac{1}{\sqrt[+]{\lm-\lm_{0}^{+}}}
   \prod_{1\le m\neq k}\frac{\lm_{m}^{\ld}-\lm}{\vs_{m}(\lm)}.
\]
For any $k\ge 0$, $\chi_{k}$ is analytic on $U_{k}$. Moreover, the root $\sqrt[+]{\lm-\lm_{0}^{+}}$ [$\vs_{k}$] admits opposite signs on opposite sides of $G_{0}$ [$G_{k}$]. For any given $\lm\in G_{n}$ denote by $\lm^{\pm}$ the corresponding elements on $G_{n}^{\pm}$. Thus for  all $\lm \in G_{n}$
\[
  F_{n}(\lm^{+}) = - F_{n}(\lm^{-}).
\]
Consequently, $F_{n}^{2}$ is continuous on all of $U_{n}$. Since it is already analytic on $U_{n}\setminus G_{n}$, it is analytic on $U_{n}$.\qed
\end{proof}


\subsection{Asymptotics of F}
\label{ss:4-F-exp}

Assume that $q$ is a real valued finite gap potential with $\int_0^1 q(x) \,\dx = 0$. Then $q$ is real analytic and there exists $A \subset \N$ finite so that $\gm_n >0$ if and only if $n \in A$. Hence there exists $n_0 \ge 1$ so that $\lm_n^+ = \lm_n^-$ and $\abs{\lm_n^+ - \pi^2n^2} \le 1$ for any $n > n_0$ and $|\lm_{0}^{+}|$, $\abs{\lm_n^\pm} \le n_0^2 \pi^2 + 1$ for any $n \le n_0$. By Lemma \ref{F2.analytic} it then follows that $F^2(\lm)$ and hence $F^4(\lm)$ are analytic functions in the domain $\abs{\lm} > n_0^2 \pi^2 + 1$. In particular, $F^2(\lm)$ admits a Laurent expansion,
\begin{equation}
\label{F2.laurent.expansion} 
F^2(\lm) = \sum_{k \in \Z} a_k \lm^k \ , \qquad \abs{\lm} > n_0^2 \pi^2 + 1.
\end{equation}
From the product representation~\eqref{F-prod} one sees that $F^2(\lm) = O(\lm)$ and hence $a_k = 0$ for any $k \ge 2$. It is well known that the coefficients $a_{k}$ can be explicitly computed -- see e.g. \cite{Molnar:_ROURXz4} for a self-contained proof.

\begin{lem}
\label{F-exp}
For any real valued finite gap potential $q$ with $\int_0^1 q(x)\, \dx = 0$, 
\begin{equation}
  \label{F4.exp} 
  F^4(\lm) = \lm^2 - \Hc_{0} -  \frac{\Hkc}{4}\frac{1}{\lm} + O\left(\frac{1}{\lm^2}\right),
  \qquad \Hc_{0} = \frac{1}{2}\int_{0}^{1} q(x)^{2}\,\dx.\fish
\end{equation}
\end{lem}

As an application of the asymptotics \eqref{F4.exp} of $F^4(\lm)$ we derive a formula for the KdV Hamiltonian $\Hkc$ for finite gap  potentials. To this end we introduce for any $q \in \Wp$ and $n \ge 1$
\begin{equation}
\label{Rn.def}
  \Rc_n \equiv \Rc_n(q) = \frac{1}{\pi} \int_{\Gm_n} F_{n}^{3} \, \dlm,
  \qquad
  F_{n}(\lm) = F(\lm) + \ii n\pi.
\end{equation}
According to Lemma \ref{F-prop} (i), $F_{n}(\lm)$ is analytic on $\C \setminus \bigcup_{\atop{\gm_{m} \neq 0}{m\ge 0}} G_m$, hence $\Rc_n$ vanishes if $\gm_n = 0$.

\begin{prop}
For any real finite gap potential $q$, 
\begin{equation}
  \label{ham.form.fg}
  \Hkc(q) = \sum_{ n \in A } 8n^{3}\pi^{3}I_{n} + \sum_{n \in A} 8n\pi \Rc_{n} \ ,
\end{equation}
where $A \equiv A_q = \setdef{ n \ge 1 }{ \gm_{n} \neq 0 }$  is finite.\fish
\end{prop}

\begin{proof}
By Lemma \ref{F2.analytic}, $F^4(\lm)$ is analytic on $\C \setminus \bigcup_{n \in A} G_n$, hence in view of Lemma~\ref{F-exp} and Cauchy's integral formula
\begin{equation}
  \label{H1.formula}
  \frac{1}{2\pi \ii} \int_{C_r} F^4(\lm) \, \dlm  = -\frac{\Hkc}{4},
\end{equation}
where $C_r$ denotes a counter clockwise oriented circle around the origin enclosing all open gaps. By contour deformation we thus obtain
\begin{equation}
  \label{H1.formula1}
  \Hkc = -4 \frac{1}{2\pi \ii} \int_{C_r} F^4(\lm) \, \dlm  =
  -4\sum_{n\in A}  \frac{1}{2\pi \ii} \int_{\Gm_{n}} F^4(\lm) \, \dlm.
\end{equation}
Moreover, since $F_{n}^{2}$ is analytic on $U_{n}$, one concludes that in the expression $F^{4}(\lm) = (F_{n}(\lm)-\ii n\pi)^{4}$, the even powers of $F_n$ will not contribute to the integral  $\int_{\Gm_n} F^4(\lm) \, \dlm$. As a consequence 
\begin{align*}
  \frac{1}{2\pi \ii} \int_{\Gm_n} F^4(\lm) \, \dlm
   & = -\frac{1}{2\pi \ii} \int_{\Gm_n} \left(4(\ii n\pi)F_{n}^3(\lm) + 4(\ii n\pi)^{3} F_{n}(\lm)\right)  \, \dlm.
\end{align*}
After integration by parts, formula~\eqref{action} of the action $I_{n}$ becomes
\[
  I_n = -\frac{1}{\pi} \int_{\Gm_{n}} F(\lm)\,\dlm = -\frac{1}{\pi} \int_{\Gm_{n}} F_{n}(\lm)\,\dlm.
\]
One then concludes with \eqref{Rn.def} that
\begin{equation}
  \label{formula.1}
  -4 \frac{1}{2 \pi \ii} \int_{\Gm_n} F^4(\lm) \, \dlm
   =  (8n^{3}\pi^{3})I_n + (8n\pi)\Rc_{n}
\end{equation}
yielding in view of \eqref{H1.formula1} the claimed formula.
\qed
\end{proof}

\subsection{Nonlinear part of $\Hkc$}
\label{ss:4-pf}

Recall that $\Rc_{n}$, $n\ge 1$, has been defined in \eqref{Rn.def} as follows
\[
  \Rc_{n} = \frac{1}{\pi}\int_{\Gm_{n}} F_{n}^{3}\,\dlm,\qquad F_{n}(\lm) = F(\lm) + \ii n\pi.
\]
When expressed in Birkhoff coordinates, we denote $\Rc_{n}$ by $R_{n}$.
The following lemma describes the properties of $\Rc_{n}$ and $R_{n}$ used in the sequel.

\begin{lem}
\label{Rn-prop}

\begin{equivenum}
\item 
For any $n\ge 1$,
$\Rc_{n}$ is analytic on $\Wp$ and satisfies $\Rc_{n} = O(\gm_{n}^{4}/n^{3})$ locally uniformly on $\Wp$. Furthermore, if $\gm_{n} = 0$, then $\Rc_{n} = 0$, and for $q\in\Wp$ real valued, $\Rc_{n}\le 0$. In particular, for $q = 0$, $\Rc_{n} = 0$ for all $n\ge 1$.

\item
For any $n\ge 1$, $R_{n}$ is an analytic function of the actions alone on a neighborhood $\Vs$ (independent of $n$) of the positive cone $\ell_{+}^{-1,1}(\N)$ in $\ell_{\C}^{-1,1}(\N)$.
If $I_{n} = 0$, then $R_{n} = 0$, and on $\ell_{+}^{-1,1}(\N)$ we have $R_{n}\le 0$.\fish

\end{equivenum}

\end{lem}

\begin{proof}
(i)
By Lemma~\ref{F-prop} (i), $\Rc_{n}$ is analytic on $\Ws$, and by Lemma~\ref{F-prop} (iii),
\[
  F_{n}\Big|_{G_{n}^{\pm}} = O(\gm_{n}/n),
\]
locally uniformly in $q$, hence $\Rc_{n} = O(\gm_{n}^{4}/n^{3})$ locally uniformly on $\Ws$. If $\gm_{n} = 0$, then $F(\lm)$ is analytic on $U_{n}$ by Lemma~\ref{F-prop} (i) hence $\Rc_{n} = 0$. If $q$ is real valued, then one has for $\lm_{n}^{-} < \lm < \lm_{n}^{+}$
\[
  F_{n}(\lm)\bigg|_{G_{n}^{\pm}}
   = \int_{\lm_{n}^{+}}^{\lm} \frac{\dDl}{\sqrt[c]{\Dl^{2}-4}}\bigg|_{G_{n}^{\pm}}\,\dlm
   = \pm\int_{\lm_{n}^{+}}^{\lm} \frac{(-1)^{n}\dDl}{\sqrt[+]{\Dl^{2}-4}}\,\dlm
   = \pm\arccosh\frac{(-1)^{n}\Dl(\lm)}{2},
\]
and consequently by shrinking the contour $\Gm_{n}$ in the definition of $\Rc_{n}$ to $G_{n}$
\[
  \Rc_{n}(q) = -\frac{2}{\pi}\int_{\lm_{n}^{-}}^{\lm_{n}^{+}}
   \left(\arccosh\left(\frac{(-1)^{n}\Dl(\lm)}{2} \right)\right)^{3}\,\dlm \le 0.
\]

(ii)
Since $F_{n}$ and hence $\Rc_{n}$ is defined in terms of the discriminant, $\Rc_{n}$ depends only on the periodic spectrum of $q$. Therefore, when expressed as a function of the Birkhoff coordinates, it is a function of the actions alone.
Arguing as in the proof of \cite[Theorem 20.3]{Grebert:2014iq} it follows that each $R_{n}$, $n\ge 1$, is an analytic function on a neighborhood $\Vs$ (independent of $n$) of the positive cone $\ell_{+}^{-1,1}(\N)$ in $\ell_{\C}^{-1,1}(\N)$. Since every sequence $(I_{n})_{n\ge 1}\in \ell_{+}^{-1,1}(\N)$ is the sequence of actions for some real valued potential $q\in H_{0}^{-1}$, we conclude that $R_{n}\le 0$ on $\ell_{+}^{-1,1}$.\qed
\end{proof}

Lemma~\ref{Rn-prop} can be applied to prove the following result which is the key ingredient in the proof of Theorem~\ref{main}.
\begin{prop}
\label{R-prop}
\begin{equivenum}
\item
The series $\sum_{n\ge 1} 8n\pi\Rc_{n}$ converges absolutely locally uniformly on $\Ws\cap \Ls_{0,\C}^{-1/2,4}$ and hence is analytic.

\item
The series
\[
  \Hn \defl \sum_{n\ge 1} 8n\pi R_{n}
\]
is an analytic function of the action variables defined on a complex neighborhood of the positive quadrant $\ell_{+}^{2}(\N)$ in $\ell_{\C}^{2}(\N)$. On $\ell_{+}^{2}(\N)$ it is non-positive.~\fish

\end{equivenum}

\end{prop}

\begin{proof}
(i)
By Lemma~\ref{Rn-prop} (i), every $\Rc_{n}$ with $n\ge 1$ is analytic on $\Ws$. As $\Ls_{0,\C}^{-1/2,4} \opento H_{0,\C}^{-1}$,  its restriction to $\Ws\cap \Ls_{0,\C}^{-1/2,4}$ is analytic as well. In addition, also by Lemma~\ref{Rn-prop}, $8n\pi \Rc_{n} = O(\gm_{n}^{4}/n^{2})$ locally uniformly on $\Ws$, and by Theorem~\ref{fw:per}, $\gm_{n} = n^{1/2}\ell_{n}^{4}$ locally uniformly on $\Ws\cap \Ls_{0,\C}^{-1/2,4}$. It then follows that $\sum_{n\ge 1} 8n\pi \Rc_{n}$ is analytic on $\Ws\cap \Ls_{0,\C}^{-1/2,4}$ -- cf. \cite[Theorem~A.4]{Grebert:2014iq}.

(ii) By Lemma~\ref{Rn-prop} (ii), each $R_{n}$, $n\ge 1$, is analytic on $\Vs$. According to \eqref{In-gmn} the action variable $I_{n}$ is related to the gap length $\gm_{n}$ as
\[
  \frac{\gm_{n}^{2}}{n} = O(I_{n})
\]
locally uniformly on $\Vs$. Thus, $8n\pi R_{n} = O(I_{n}^{2})$ locally uniformly on $\Vs$.
Clearly, the restriction of each $R_{n}$, $n\ge 1$, to $\Vs_{2} = \Vs \cap \ell^{2}_{\C}$ is analytic, too. Therefore, $\sum_{n\ge 1} 8n\pi R_{n}$ converges absolutely and locally uniformly to an analytic function on $\Vs_{2}$, which we denote by $\Hn$. On $\ell^{2}_{+}(\N)$ the function $\Hn$ is real valued with values in $(-\infty,0]$.\qed

\end{proof}

From the property that for $q\in \Wp\cap H_{0,\C}^{1}$, $I_{n} = n^{3}\ell_{n}^{1}$, Proposition~\ref{R-prop} leads to the following

\begin{cor}
\label{H-kdv-series}
On $\Wp\cap H_{0,\C}^{1}$
\begin{equation}
  \label{H-R}
  \Hkc = \sum_{n\ge 1} 8n^{3}\pi^{3} I_{n} + \sum_{n\ge 1} 8n\pi  \Rc_{n}
\end{equation}
where both sums converge absolutely locally uniformly on $\Wp\cap H_{0,\C}^{1}$.
In addition, expressing~\eqref{H-R} in Birkhoff coordinates gives $\omk_{n} = \partial_{I_{n}}\Hn + 8n^{3}\pi^{3}$ on $\ell_{+}^{3,1}(\N)$.
\fish
\end{cor}

\begin{proof}
Let us first consider the case where the potential is real valued. Recall that real valued finite gap potentials are dense in $H_{0}^{1}$ and that for any $q\in H_{0}^{1}$ and $n\ge 1$ with $\lm_{n}^{+} = \lm_{n}^{-}$ one has $I_{n} = 0$ and $\Rc_{n} = 0$. The claimed formula~\eqref{H-R} for $q\in H_{0}^{1}$ then follows from \eqref{ham.form.fg} by approximating $q$ by a sequence of finite gap potentials. As both sums in \eqref{H-R} are real analytic on $\Wp\cap H_{0,\C}^{1}$, the identity holds on $\Wp\cap H_{0,\C}^{1}$.\qed
\end{proof}

The proof of Theorem~\ref{main} now easily follows from Proposition~\ref{R-prop}.

\begin{proof}[Proof of Theorem~\ref{main}.]
By Proposition~\ref{R-prop}, $\Hn$ is an analytic function on a complex neighborhood $\Vs\subset\ell_{\C}^{2}(\N)$ of the positive quadrant $\ell_{+}^{2}(\N)$. Its differential $\omr \defl \ddd_{I} \Hn$ is an analytic map with values in $\ell_{\C}^{2}(\N)$,
\[
  \omr\colon \Vs \to \ell_{\C}^{2}(\N),\quad 
  I = (I_{n})_{n\ge 1} \mapsto (\omr_{n} = \partial_{I_{n}} \Hn)_{n\ge 1}.
\]
In view of Corollary~\ref{H-kdv-series} and \eqref{omn-asymptotics} one has $\omr_{n} = \omk_{n} - 8n^{3}\pi^{3}$ on $V_{0}\cap \Vs$ for any $n\ge 1$.
The expansion \eqref{exp-H-kdv} of $\Hk\colon \ell_{+}^{3,1}(\N)\to \R$ at $I=0$ implies that $\Hn\colon \ell_{+}^{2}(\N)\to \R$ has an expansion at $I=0$ of the form $-3\sum_{n\ge 1} I_{n}^{2} + \dotsb$ where the dots stand for higher order terms in $I$. In particular, $\omr(I) = (-6 I_{n} + \dotsb)_{n\ge 1}$ vanishes at $I=0$ and its differential at $I=0$ is given by
\[
  \ddd_{0}\omr = -6\Id_{\ell^{2}_{\C}(\N)}.
\]
It implies that $\omr$ defines a local diffeomorphism near $I=0$ and that the Hessian $\ddd_{I}^{2}\Hn$ satisfies the estimate~\eqref{convexity}.\qed
\end{proof}


\section{Proof of Theorem \ref{main2}}
\label{section5}

In \cite{Kappeler:2005fb} and \cite{Kappeler:2008fl}, the restrictions of the Birkhoff map
\[
  \Phi\colon H_{0}^{-1}\to \ell_{0}^{-1/2,2},\qquad q\mapsto (z_{n}(q))_{n\in\Z},\quad z_{0}(q) = 0,
\]
to the Sobolev spaces $H_{0}^{s}$, $-1\le s\le 0$, are studied. It turns out that the arguments developed in these two papers can also be adapted to prove Theorem~\ref{main2}. As a first step we prove the following version of \cite[Theorem 5.3]{Kappeler:2005fb}.

\begin{lem}
\label{Phi-ana}
For any $-1/2\le s \le 0$ and $2\le p < \infty$
\[
  \Phi_{s,p} \equiv \Phi\bigg|_{\Ls_{0}^{s,p}}\colon \Ls_{0}^{s,p}\to \ell_{0}^{s+1/2,p},
  \qquad q \mapsto (z_{n}(q))_{n\in\Z},
\]
is real analytic and extends analytically to an open neighborhood $\Wp_{s,p}$ of $\Ls_{0}^{s,p}$ in $\Ls_{0,\C}^{s,p}$. Its Jacobian $\ddd_{0}\Phi_{s,p}$ at $q = 0$ is the weighted Fourier transform 
\[
  \ddd_{0}\Phi_{s,p} \colon \Ls_{0}^{s,p}\to \ell_{0}^{s+1/2,p},\qquad
  f\mapsto \left(\frac{1}{\sqrt{2\pi\max(\abs{n},1)}} \spi{f,e_{2n}} \right)_{n \in \Z}
\]
with inverse given by
\[
  (\ddd_{0}\Phi_{s,p})^{-1} \colon \ell_{0}^{s+1/2,p} \to \Ls_{0}^{s,p},\quad
  (w_{n})_{n\in\Z}\mapsto \sum_{n\in\Z} \sqrt{2\pi \abs{n}}w_{n}e_{2n}.
\]
In particular, $\Phi_{s,p}$ is a local diffeomorphism at $q = 0$.\fish
\end{lem}


\begin{proof}
Since $\Ls_{0,\C}^{s,p} \opento H_{0,\C}^{-1}$, the coordinate functions $z_{n}$ defined in \eqref{e:z_k} are analytic functions on the complex neighborhood $\Wp\subset H_{0,\C}^{-1}$ of $H_{0}^{-1}$ of Theorem~\ref{bm.H-1}. Furthermore,
\[
  z_{\pm n}(q) = O\p*{\frac{\abs{\gm_{n}(q)} + \abs{\mu_{n}(q)-\tau_{n}(q)}}{\sqrt{n}}}
\]
locally uniformly on $\Ws$ and uniformly in $n\ge 1$. By the asymptotics of the periodic and Dirichlet eigenvalues of Theorems~\ref{fw:per}-~\ref{fw:dir}, $\Phi_{s,p}$ maps the complex neighborhood $\Ws_{s,p} \defl \Ws\cap \Ls_{0,\C}^{s,p}$ of $\Ls_{0}^{s,p}$ into the space $\ell_{0,\C}^{s+1/2,p}$ and is locally bounded. Arguing as in the proof of \cite[Theorem 8.5]{Kappeler:2003up} one sees that $\Phi_{s,p}$ is analytic. The formulas for $\ddd_{0}\Phi_{s,p}$ and its inverse follow from \cite[Theorem 9.7]{Kappeler:2003up} by continuity.~\qed
\end{proof}

As a next step we show that $\Phi_{s,p}$ is a local diffeomorphism. 

\begin{lem}
\label{dPhi-iso}
For any $q\in \Ls_{0}^{s,p}$ with $-1/2 \le s \le 0$ and $2\le p < \infty$,
\[
  \ddd_{q}\Phi_{s,p}\colon \Ls_{0}^{s,p}\to \ell_{0}^{s+1/2,p}
\]
is a linear isomorphism.
\end{lem}

\begin{proof}
We follow the line of arguments used in \cite[Proposition 7.1]{Kappeler:2005fb}.
By Theorem~\ref{bm.H-1}, $\ddd_{q}\Phi\colon H_{0}^{-1}\to \ell_{0}^{-1/2,2}$ is a linear isomorphism for any $q\in H_{0}^{-1}$. Furthermore, by Lemma~\ref{Phi-ana}, $\ddd_{0}\Phi_{s,p}\colon \Ls_{0}^{s,p}\to \ell_{0}^{s+1/2,p}$ is a linear isomorphism. It is convenient to introduce the maps $\Om\defl (\ddd_{0}\Phi)^{-1}\circ\Phi\colon H_{0}^{-1}\to H_{0}^{-1}$ and $\Om_{s,p}\defl (\ddd_{0}\Phi_{s,p})^{-1}\circ\Phi_{s,p}\colon \Ls_{0}^{s,p}\to \Ls_{0}^{s,p}$. It is to prove that, for any $q\in \Ls_{0}^{s,p}$, $\ddd_{q} \Om_{s,p}$ is a linear isomorphism. By Theorem~\ref{Th:main*}, $\ddd_{q}\Om\colon H_{0}^{-1}\to H_{0}^{-1}$ is a linear isomorphism and as $\ddd_{q}\Om_{s,p} = \ddd_{q}\Om\big|_{\Ls_{0}^{s,p}}$ it then follows that $\ddd_{q}\Om_{s,p}$ is one-to-one. To show that $\ddd_{q}\Om_{s,p}$ is onto, we prove that $\ddd_{q}\Om_{s,p}$ is a Fredholm operator of index zero. 
To this end, denote by $\partial_{q}z_{n}$ the $L^{2}$-gradient of $z_{n}\colon \Ws\to \C$. Then $\partial_{q}z_{n}\in H_{0,\C}^{1}$ and hence $\partial_{q}z_{n}\in \Ls_{0,\C}^{\abs{s},p'} \cong (\Ls_{0,\C}^{s,p})'$ where $p'$ is conjugate to $p$, $\frac{1}{p}+\frac{1}{p'} = 1$. By Lemma~\ref{Phi-ana}, for any $q\in \Ls_{0}^{s,p}$, $\ddd_{q}\Om_{s,p}$ has the form
\[
  \ddd_{q}\Om_{s,p}(f) = \sum_{n\neq 0} \spi{f,d_{n}(q)}e_{2n},\qquad \forall f\in \Ls_{0}^{s,p},
\]
where $d_{n}(q) = \sqrt{2\abs{n}\pi}\partial_{q}z_{-n}$. We claim that for any $q\in \Ls_{0}^{s,p}$, $T(q) = (\ddd_{q}\Om_{s,p})-\uno$ is a compact operator on $\Ls_{0}^{s,p}$. Since $q\mapsto T(q)$ is analytic and by \cite[Theorem~B.16]{Kappeler:2003up} the finite gap potentials are dense in $H_{0}^{0}$ and hence in $\Ls_{0}^{s,p}$, it suffices to show that $T(q)$ is compact for any finite gap potential. Recall that $q\in H_{0}^{-1}$ is a finite gap potential if $A_{q} = \setdef{n\ge 1}{\gm_{n}(q)\neq 0}$ is finite. Using that for any $f\in \Ls_{0}^{s,p}$, $f = \sum_{n\neq 0} \spi{f,e_{2n}}e_{2n}$, one gets
\[
  Tf = \sum_{n\neq 0} \spi{f,d_{n}-e_{2n}}e_{2n}.
\]
For any $N\ge 1$ define the operator $T_{N}\colon \Ls_{0}^{s,p}\to \Ls_{0}^{s,p}$ given by
\[
  T_{N}f = \sum_{0\neq \abs{n}\le N} \spi{f,d_{n}-e_{2n}}e_{2n}.
\]
Then $T_{N}$ is a linear operator of finite rank and thus compact. We now show that $\n{T-T_{N}}_{s,p} \to 0$ in the operator norm as $N\to \infty$. For any $f\in \Ls_{0}^{s,p}$, $\abs{\spi{f,d_{n}-e_{2n}}} \le \n{f}_{s,p}\n{d_{n}-e_{2n}}_{\abs{s},p'}$ and hence
\begin{align*}
  \n{(T-T_{N})f}_{s,p}
   &= \p*{\sum_{\abs{n} > N} \frac{\abs{\spi{f,d_{n}-e_{2n}}}^{p}}{\abs{n}^{\abs{s}p}}}^{1/p}\\
   &\le \n{f}_{s,p}\p*{\sum_{\abs{n} > N} \frac{\n{d_{n}-e_{2n}}_{\abs{s},p'}^{p}}{\abs{n}^{\abs{s}p}}}^{1/p}.
\end{align*}
Since by Lemma~\ref{dn-en} below $\sum_{\abs{n}> N} \p*{\frac{\n{d_{n}-e_{2n}}_{\abs{s},p'}^{p}}{\abs{n}^{\abs{s}p}}}^{1/p}\to 0$ as $N\to\infty$, it then follows that $\n{T-T_{N}}_{s,p}\to 0$ in the operator norm as $N\to \infty$. Hence $T(q)$ is compact for any finite gap potential which proves the claim.\qed
\end{proof}

The following lemma, used to prove Lemma~\ref{dPhi-iso}, extends Lemma~7.3 in \cite{Kappeler:2005fb}.

\begin{lem}
\label{dn-en}
For any real valued finite gap potential $q$ with $q_{0} = 0$,
\[
  \sum_{\abs{n}\neq 0} \frac{\n{d_{n}(q)-e_{2n}}_{\abs{s},p'}^{p}}{\abs{n}^{\abs{s}p}} < \infty
\]
for any $-1/2\le s \le 0$ and any $2\le p <\infty$.
\end{lem}

\begin{proof}
The case $(s,p) = (0,2)$ is contained in \cite[Lemma 11.12]{Kappeler:2003up} so it remains to consider the case $(s,p)\neq (0,2)$. We adapt the proof of \cite[Lemma 7.3]{Kappeler:2005fb} to the case at hand. Let $q$ be a finite gap potential with $q_{0} = 0$ and introduce $\Nc = \setdef{n\ge 1}{\gm_{n} = 0;\quad \lm_{n}^{+}\neq 0}$. Then $\N\setminus\Nc$ is finite and it suffices to show that
\[
  \sum_{\abs{n}\in \Nc} \frac{\n{d_{n}-e_{2n}}_{\abs{s},p'}^{p}}{\abs{n}^{\abs{s}p}} < \infty
\]
where $d_{n} = d_{n}(q)$. By \cite[Theorem 9.7]{Kappeler:2003up} one has for $n\in\Nc$,
\[
  d_{n} = \sqrt{2\abs{n}\pi}\partial_{q}z_{-n} = c_{n}\partial_{q}z_{n}^{+},\qquad
  d_{-n} = \sqrt{2\abs{n}\pi}\partial_{q}z_{n} = \ob{c_{n}}\partial_{q}z_{n}^{-}.
\]
Here $c_{n}$ is a complex number independent of $x$ 
%
%
and
\[
  \partial_{q}z_{n}^{\pm} = h_{n}^{2} - g_{n}^{2} \pm 2\ii g_{n}h_{n},
\]
where $g_{n}$ denotes the real valued eigenfunction of the Dirichlet eigenvalue $\mu_{n} = \lm_{n}^{+}$ of $-\partial_{x}^{2} + q$, normalized by $\int_{0}^{1} g_{n}^{2}\,\dx = 1$, $g_{n}'(0) > 0$, and $h_{n}$ denotes the real valued eigenfunction of the periodic eigenvalue $\lm_{n}^{+} = \lm_{n}^{-}$ which is $L^{2}$-orthogonal to $g_{n}$ and satisfies $\int_{0}^{1} h_{n}^{2}\,\dx = 1$ and $h_{n}(0) > 0$. Following the arguments of the proof of Lemma~7.3 in \cite{Kappeler:2005fb} consider the operator $Q = -\frac{1}{2}\partial_{x}^{3} + 2q\partial_{x} + q'$. One verifies that the product $Y = y_{i}y_{j}$ of any two solutions of $-y'' + qy = \lm y$ is smooth and satisfies $QY = 2\lm \partial_{x} Y$. Since $\int_{0}^{1}\partial_{q}z_{n}\,\dx = 0$ for any $n\in\Nc$ and $\partial_{q}z_{n}$ is a linear combination of  products of solutions of $-y'' + qy  = \lm_{n}^{+}y$ it follows that for any $n\in\Nc$
\[
  d_{\pm n} = \frac{1}{2\lm_{n}^{+}} D_{0}^{-1}Q d_{\pm n}.
\]
Recall that $\lm_{n}^{+}\neq 0$ for $n\in\Nc$. Here $D_{0}^{-1}$ denotes the inverse of the restriction $D_{0}\colon \Ls_{0}^{s,p}\to \Ls_{0}^{s-1,p}$ of $\partial_{x}$ to $\Ls_{0}^{s,p}$. For $f\in\Ls_{0}^{2+s,p}$ one has for any $n\in\Nc$
\begin{align*}
  \spi{d_{\pm n}, f}
   = \frac{1}{2\lm_{n}^{+}}\spi{D_{0}^{-1}Q d_{\pm n}, f}
   = \frac{1}{2\lm_{n}^{+}}\spi{d_{\pm n}, QD_{0}^{-1}f}
   = -\frac{1}{4\lm_{n}^{+}}\spi{d_{\pm n}, f^{*}}
\end{align*}
where $f^{*}\defl f'' - 4qf -2q'D_{0}^{-1}f\in\Ls_{0}^{s,p}$. This yields
\begin{align*}
  \spi{d_{\pm n},f''}
  &= \spi{d_{\pm n},f^{*}} - \spi{d_{\pm n},f^{*}-f''}\\
  &= -4\lm_{n}^{+}\spi{d_{\pm n},f} - \spi{d_{\pm n},-4qf - 2q'D_{0}^{-1}f}.
\end{align*}
Combining this with $\spi{e_{\pm 2n},f''} = -4n^{2}\pi^{2}\spi{e_{\pm 2n},f}$ one obtains for $n\in\Nc$
\begin{align}
\label{dn-en-eqn}
\begin{split}
  \spi{d_{\pm n}-e_{\pm 2n},f''} &= 4(n^{2}\pi^{2}-\lm_{n}^{+})\spi{e_{\pm 2n},f}
  -4\lm_{n}^{+}\spi{d_{\pm n} - e_{\pm 2n},f}\\
  &\qquad + \spi{e_{\pm 2n}, 4qf + 2q' D_{0}^{-1}f}\\
  &\qquad + \spi{d_{\pm n}-e_{\pm 2n},4qf + 2q'D_{0}^{-1}f}.
\end{split}
\end{align}
We now estimate the terms on the right hand side of the latter identity separately. By standard estimates one has for $f\in \Ls_{0}^{2+s,p}\opento H_{0}^{1}$
\begin{align*}
  \spi{e_{\pm 2n},f} &= O\p*{\frac{1}{n}\n{f}_{H_{0}^{1}}},\\ 
  \lm_{n}^{+} &= n^{2}\pi^{2} + O\p*{\frac{1}{n}},\\
  \spi{e_{\pm 2n},4qf + 2q'D_{0}^{-1}f} &= O\p*{\frac{1}{n}\n{f}_{H_{0}^{1}}}.
\end{align*}
For the latter one we used that $qf$ and $q'D_{0}^{-1}f$ are in $H_{0}^{1}$ since $H_{0}^{1}$ is an algebra.
Finally, by \cite[Lemma 11.12]{Kappeler:2003up} one has for any $n\in\Nc$
\[
  \spi{d_{\pm n}-e_{\pm 2n},f} = O\p*{\frac{\log n}{n^{5/2}}\n{f}_{H_{0}^{1}}},
\]
and, using again that $qf$ and $q'D_{0}^{-1}f$ are in $H_{0}^{1}$,
\[
  \spi{d_{\pm n}-e_{\pm 2n},4qf + 2q'D_{0}^{-1}f} = O\p*{\frac{\log n}{n^{5/2}}\n{f}_{H_{0}^{1}}}.
\]
Note that any element in $\Ls_{0}^{s,p}$ is of the form $f'' = \sum_{n\neq 0} 4n^{2}\pi^{2}f_{2n}e_{2n}$ where $f = \sum_{n\neq 0} f_{2n}e_{2n}$ is in $\Ls_{0}^{2+s,p}$ and
\[
  \n{f''}_{\Ls_{0}^{s,p}}^{p}
   \le \sum_{n\neq 0} \frac{(4\pi^{2}n^{2})^{p}\abs{f_{2n}}^{p}}{\w{n}^{\abs{s}p}} 
   \le   (2\pi)^{2p}\sum_{n\neq 0} \w{n}^{(2-\abs{s})p}\abs{f_{2n}}^{p}
   \le (2\pi)^{2p}\n{f}_{\Ls_{0}^{2-\abs{s}},p}^{p}
\]
Hence, combining the above estimates we conclude from \eqref{dn-en-eqn} that
\[
  \n{d_{\pm n}-e_{\pm 2n}}_{\abs{s},p'}
   \le \sup_{\n{f}_{2+s,p}\le 1} \abs{\spi{d_{\pm n}-e_{\pm 2n},f''}}
   \le C\frac{\log \w{n}}{n^{1/2}}
\]
and, in turn,
\[
  \sum_{\abs{n}\in \Nc} \frac{\n{d_{n}-e_{2n}}_{\abs{s},p'}^{p}}{\abs{n}^{\abs{s}p}} \le
  C\sum_{\abs{n}\in\Nc} \frac{(\log \w{n})^{p}}{\abs{n}^{(\abs{s}+1/2)p}} < \infty
\]
as $(\abs{s}+1/2)p > 1$ for any $-1/2\le s \le 0$ and $2\le p < \infty$ with $(s,p)\neq (0,2)$.\qed
\end{proof}

In a third step, following arguments used in \cite{Kappeler:2008fl}, we prove that $\Phi_{s,p}$ is onto.

\begin{lem}
\label{Phi-onto}
For any $-1/2\le s \le 0$ and $2\le p < \infty$, $\Phi_{s,p}\colon \Ls_{0}^{s,p}\to \ell_{0}^{s+1/2,p}$ is onto.
\end{lem}

\begin{proof}
Fix $-1/2\le s \le 0$, $2\le p < \infty$, and $z^{(0)} = (z_{n}^{(0)})_{n\neq 0} \in \ell_{0}^{s+1/2,p}$. Since $\Phi\colon H_{0}^{-1}\to \ell_{0}^{-1/2,2}$ is onto, there exits $q^{(0)}\in H_{0}^{-1}$ such that $z^{(0)} = \Phi(q^{(0)})$. Assume that
\begin{equation}
  \label{q-0}
  q^{(0)} \in H_{0}^{-1}\setminus \Ls_{0}^{s,p}.
\end{equation}
We show that \eqref{q-0} leads to a contradiction. As $\Phi_{s,p}(0) = 0$ and by Lemma~\ref{dPhi-iso} the differential of $\Phi_{s,p}$ at $0$ is a linear isomorphism, by the inverse function theorem, there exists an open neighborhood $U$ of $0$ in $\Ls_{0}^{s,p}$ and an open neighborhood $V$ of $0$ in $\ell_{0}^{s+1/2,p}$ so that
\[
  \Phi_{s,p}\big|_{U}\colon U\to V
\]
is a diffeomorphism. Without loss of generality we can assume that $V$ is a ball of radius $2\ep$, centered at $0$,
\begin{equation}
  \label{V-ball}
  V = B_{2\ep}(0)\subset \ell_{0}^{s+1/2,p}.
\end{equation}
To obtain the desired contradiction we construct a finite sequence $q^{(1)}, \dotsc, q^{(N)}$ with the property that $q^{(N)}\in U$ and $q^{(k)}-q^{(k-1)}\in H_{0}^{0}$ for any $1\le k\le N$. To this end consider the action angle variables $(I_{k},\th_{k})$, $k\ge 1$, constructed in \cite{Kappeler:2005fb}. For any $k\ge 1$, $I_{k}$ is defined on $H_{0}^{-1}$ and satisfies $I_{k} = \abs{z_{k}}^{2}$ whereas the angle variable $\th_{k}$ is defined on $H_{0}^{-1}\setminus Z_{k}$ with values in $\R/2\pi\Z$ and is real analytic when considered $\mod \pi$. Here $Z_{k}$ denotes the set $Z_{k} = \setdef{q\in \Ws}{\gm_{k}(q) = 0}$, introduced in \eqref{Zk}. The $L^{2}$-gradient $\partial_{q}\th_{k}$ is a real analytic map from $H_{0}^{-1}\setminus Z_{k}$ with values in $H_{0}^{1}$ whereas the corresponding Hamiltonian vector field,
\[
  Y_{k} = \partial_{x}\partial_{q}\th_{k}\colon H_{0}^{-1}\setminus Z_{k}\to H_{0}^{0}
\]
is a real analytic map with values in $H_{0}^{0}$. It defines a dynamical system
$
  \dot q = Y_{k}(q)
$
on $H_{0}^{-1}\setminus Z_{k}$. We now use the flows of these vector fields to construct the sequence $q^{(1)}, \dotsc, q^{(N)}$  recursively as follows. The potential $q^{(0)}$ is given by \eqref{q-0}. For any $n\ge 1$ assume that $q^{(n-1)}$ has already been constructed. If $\lin{n}^{2s+1}I_{n}(q^{(n-1)}) < \ep^{2}/2^{2n/p}$, then set $q^{(n)} \equiv q^{(n-1)}$. If $\lin{n}^{2s+1}I_{n}(q^{(n-1)}) \ge \ep^{2}/2^{2n/p}$, then $q^{(n-1)}\in H_{0}^{-1}\setminus Z_{n}$ and hence the vector field $Y_{n}$ is well defined in a neighborhood of $q^{(n-1)}$. The element $q^{(n)}$ is then chosen to be an element on the solution curve of the vector field $Y_{n}$ passing through $q^{(n-1)}$ so that $\lin{n}^{2s+1}I_{n}(q^{(n)}) < \ep^{2}/2^{2n/p}$. The existence of such a potential $q^{(n)}$ follows from Lemma~\ref{Yn-flow} (i) below. Moreover, by the commutator relations (cf. section~\ref{section2})
\[
  Y_{n}(I_{k}) = \pbr{I_{k},\th_{n}} = \dl_{nk},
\]
the vector field $Y_{n}$ preserves the value of the action variable $I_{k}$ with $k\neq n$. In particular, we have
\[
  \lin{k}^{2s+1}I_{k}(q^{(n)}) < \ep^{2}\frac{1}{2^{2k/p}},\qquad 1\le k\le n,
\]
and
\[
  I_{k}(q^{(n)}) = I_{k}(q^{(0)}),\qquad k > n.
\]
One then obtains
\begin{align*}
  \n{\Phi(q^{(n)})}_{s+1/2,p}^{p}
   &= \sum_{k\neq 0} \lin{k}^{(s+1/2)p}\abs{z_{k}(q^{(n)})}^{p}
    = 2\sum_{k\ge 1} \p*{ \lin{k}^{2s+1}I_{k}(q^{(n)}) }^{p/2}\\
  &\le 2\ep^{p} \sum_{1\le k\le n} \frac{1}{2^{k}}
   + 2\sum_{k > n} \p*{ \lin{k}^{2s+1} I_{k}(q^{(0)}) }^{p/2}.
\end{align*}
Since $\Phi(q^{(0)})\in \ell_{0}^{s+1/2,p}$, one has $\sum_{k\ge 1} \p{ \lin{k}^{2k+1} I_{k}(q^{(0)}) }^{p/2} < \infty$. Choose $N\ge 1$ so that
\[
  \sum_{k > N} \p*{ \lin{k}^{2s+1} I_{k}(q^{(0)}) }^{p/2} < \ep^{p}.
\]
As a consequence, since $p\ge 2$,
$
  \n{\Phi(q^{(N)})}_{s+1/2,p}^{p} < 4\ep^{p} \le (2\ep)^{p}.
$
Since $\Phi\big|_{U}\colon U\to V = B_{2\ep}(0)$ is a diffeomorphism and $\Phi(q^{(N)})\in V$ it then follows that
\begin{equation}
  \label{q-N-s,p}
  q^{(N)}\in U\subset \Ls_{0}^{s,p}.
\end{equation}
On the other hand, it follows from the assumption \eqref{q-0} that $q^{(0)}\in H_{0}^{-1}\setminus \Ls_{0}^{s,p}$ and from Lemma~\ref{Yn-flow} (ii) that $q^{(N)}\in H_{0}^{-1}\setminus \Ls_{0}^{s,p}$ contradicting \eqref{q-N-s,p}.\qed
\end{proof}

In the proof of Lemma~\ref{Phi-onto} we have used \cite[Lemma 1]{Kappeler:2008fl}. We state it for the convenience of the reader.

\begin{lem}
\label{Yn-flow}
For any $q\in H_{0}^{-1}\setminus Z_{k}$ with $k\ge 1$ the initial value problem $\dot q = Y_{k}(q)$, $q(0) = q_{0}$ has a unique solution $t\mapsto q(t)$ in $C^{1}((-I_{k}(q_{0}),\infty), H_{0}^{-1})$. The solution has the following additional properties
\begin{renum}
\item $\lim\limits_{t\searrow -I_{k}(q_{0})} I_{k}(q(t)) = 0$,

\item $q(t)-q_{0} \in H_{0}^{0}$ for any $-I_{k}(q_{0}) < t < \infty$.\fish

\end{renum}
\end{lem}

\begin{proof}[Proof of Theorem~\ref{main2}.]
The claimed results follow from Lemma~\ref{Phi-ana}, Lemma~\ref{dPhi-iso}, and Lemma~\ref{Phi-onto}.\qed
\end{proof}

A first application of Theorem~\ref{main2} is in inverse spectral theory.

\begin{cor}
\label{inv-gap}
Let $q$ be in $H_{0}^{-1}$ and $-1/2\le s \le 0$, $2\le p <\infty$. If its sequence of gap lengths $(\gm_{n}(q))_{n\ge 1}$ is in $\ell^{s,p}(\N)$, then $q\in \Ls_{0}^{s,p}$.\fish
\end{cor}

\begin{proof}
Let $z = (z_{n})_{n\in\Z} = \Phi(q)$. By Theorem~\ref{bm.H-1} there exists $n_{0}\ge 1$ depending on $q$ such that for any $n\ge n_{0}$,
$
  \gm_{n}^{2}/2 \le 8n\pi I_{n}(q) \le 2\gm_{n}^{2},
$
where $I_{n}(q)$ denotes the $n$th action variable of $q$. Since $I_{n} = \abs{z_{n}}^{2}$ it then follows that
\[
  (z_{n})_{n\in\Z}\in \ell_{0}^{s+1/2,p} \iff (\gm_{n})_{n\ge 1}\in \ell^{s,p}(\N).
\]
The claim then follows from Theorem~\ref{main2} saying that
\[
  (z_{n})_{n\in\Z}\in \ell_{0}^{s+1/2,p} \iff q\in \Ls_{0}^{s,p}.\qed
\]
\end{proof}

A second application of Theorem~\ref{main2} concerns the isospectral sets defined as follows: For any $q\in H_{0}^{-1}$ let
\[
  \Iso(q) \defl \setdef{\tilde q\in H_{0}^{-1}}{\spec(L(\tilde q)) = \spec(L(q))}.
\]
For $z\in \ell_{0}^{-1/2,2}$ let
\[
  \Tc_{z} = \setdef{w\in \ell_{0}^{-1/2,2}}{\abs{w_{k}}^{2} = \abs{z_{k}}^{2}\quad \forall k\in \Z}.
\]
Note that for any $z \in \ell_{0}^{s+1/2,p}$, $\Tc_{z}\subset\ell_{0}^{s+1/2,p}$.

\begin{cor}
\label{iso-q}
For any $q\in\Ls_{0}^{s,p}$ with $-1/2\le s \le 0$ and $2\le p <\infty$, $\Iso(q) \subset\Ls_{0}^{s,p}$ and $\Phi(\Iso(q)) = \Tc_{\Phi(q)}$. In particular, $\Iso(q)$ is a compact subset of $\Ls_{0}^{s,p}$.
\end{cor}

\begin{proof}
By \cite[Theorem 1.2]{Kappeler:2005fb} one has $\Phi(\Iso(q)) = \Tc_{\Phi(q)}$ for any $q\in H_{0}^{-1}$. By definition, for any $q\in \Ls_{0}^{s,p}$, $\Tc_{\Phi(q)}$ is a compact subset of $\ell_{0}^{s+1/2,p}$. As a consequence,
$
  \Iso(q) = \Phi^{-1}(\Tc_{\Phi(q)})\subset\Ls_{0}^{s,p}
$
and $\Iso(q)$ is compact in $\Ls_{0}^{s,p}$.~\qed
\end{proof}

%
%


\appendix

\section{Auxiliary estimates}
\label{a:aux}

In this appendix we collect auxiliary estimates needed in Section~\ref{section3}.

\begin{lem}
\label{hilbert-sum}
For any $1/2< \sg < \infty$ there exists a constant $C_{\sg} > 0$ so that for any $n\ge 1$,
$\sum_{\abs{m}\neq n} \frac{1}{\abs{m^{2}-n^{2}}^{\sg}}$ is bounded by $C_{\sg}/n^{2\sg-1}$ if $1/2 < \sg < 1$, $C_{\sg} \frac{\log \w{n}}{n}$ if $\sg = 1$, and $C_{\sg}/n^{\sg}$ if $\sg > 1$.\fish
\end{lem}

\begin{proof}
Note that for $n\ge 1$,
\[
  \sum_{\abs{m}\neq n} \frac{1}{\abs{m^{2}-n^{2}}^{\sg}} = \frac{1}{n^{2\sg}} + 
  2\sum_{m=1}^{n-1} \frac{1}{(n^{2}-m^{2})^{\sg}} + 
  2\sum_{m=n+1}^{\infty} \frac{1}{(m^{2}-n^{2})^{\sg}}.
\]
As $n^{2}-(n-k)^{2} \ge \frac{1}{3}((n+k)^{2} - n^{2})$ for any $0\le k\le n$ it follows that
\[
  \sum_{\abs{m}\neq n} \frac{1}{\abs{m^{2}-n^{2}}^{\sg}} \le 5\cdot 3^{\sg}\sum_{m = n+1}^{\infty}
  \frac{1}{(m^{2}-n^{2})^{\sg}}.
\]
Furthermore,
\begin{align*}
  \sum_{m=n+1}^{\infty} \frac{1}{(m^{2}-n^{2})^{\sg}}
   &\le \frac{1}{n^{\sg}} + \sum_{m=n+2}^{2n} \frac{1}{n^{\sg}}\frac{1}{(m-n)^{\sg}}
   + \sum_{m=2n+1}^{\infty} \frac{1}{(m-n)^{2\sg}}\\
   &\le \frac{1}{n^{\sg}} + \frac{1}{n^{\sg}} \int_{n+1}^{2n} \frac{1}{(x-n)^{\sg}}\,\dx
   + \int_{2n}^{\infty} \frac{1}{(x-n)^{2\sg}}\,\dx.
\end{align*}
One computes,
\[
  \int_{2n}^{\infty} \frac{1}{(x-n)^{2\sg}}\,\dx = \frac{1}{2\sg-1}\frac{1}{n^{2\sg-1}}.
\]
whereas for $\sg > 1/2$ with $\sg \neq 1$
\[
  \int_{n+1}^{2n} \frac{1}{(x-n)^{\sg}}\,\dx \le \frac{1}{\abs{\sg-1}}
  \p*{ \frac{1}{n^{\sg-1}} + 1}
\]
and  $\int_{n+1}^{2n} \frac{1}{x-n}\,\dx = \log n$. Combining these estimates yields the stated ones.~\qed
\end{proof}

The next result concerns estimates for the convolution of sequences.

\begin{lem}
\label{convolution}
\begin{renum}
\item Let $-1 \le t < -1/2$. For $a = (a_{m})_{m\in \Z}\in \ell_{\C}^{t,2}$ and $b=(b_{m})_{m\in\Z}\in \ell_{\C}^{1,2}$, the convolution $a*b = (\sum_{m\in\Z} a_{n-m}b_{m})_{n\in\Z}$ is well defined and
\[
  \n{a*b}_{t,2} \le C_{t}\n{a}_{t,2}\n{b}_{1,2}.
\]
\item Let $-1/2 \le s \le 0$, $2 \le p < \infty$, and $-s-3/2 < t < 0$. For any $a = (a_{m})_{m\in \Z}\in \ell_{\C}^{s,p}$ and $b=(b_{m})_{m\in\Z}\in \ell_{\C}^{t+2,p}$,
\[
  \n{a*b}_{s,p} \le C_{s,t}\n{a}_{s,p}\n{b}_{t+2,p}.\fish
\]
\end{renum}
\end{lem}

\begin{proof}
(i) First note that for any $a\in \ell_{\C}^{t,2}$, $b\in \ell_{\C}^{1,2}$
\[
  \sum_{n\in\Z}\frac{1}{\w{n}^{2\abs{t}}}
  \p*{ \sum_{m\in\Z} \abs{a_{n-m}}\abs{b_{m}} }^{2}
  \le 2^{2}I + 2^{2}II,
\]
where
\[
  I = \sum_{n\in\Z}\frac{1}{\w{n}^{2\abs{t}}}\p*{ \sum_{m\in\Z} \mathbf{1}_{\setd{\abs{n-m} \le 2\abs{m}}} \abs{a_{n-m}}\abs{b_{m}} }^{2}
\]
and
\[
  II = \sum_{n\in\Z}\frac{1}{\w{n}^{2\abs{t}}}\p*{ \sum_{m\in\Z} \mathbf{1}_{\setd{\abs{n-m} > 2\abs{m}}} \abs{a_{n-m}}\abs{b_{m}} }^{2}.
\]
To estimate $I$, use that $(\w{n-m}^{\abs{t}}/\w{m})\mathbf{1}_{\setd{\abs{n-m} \le 2\abs{m}}} \le 2$ to conclude by the Cauchy-Schwarz inequality that
\[
  I \le \sum_{n\in\Z} \frac{4}{\w{n}^{2\abs{t}}}\n{a}_{t,2}^{2}\n{b}_{1,2}^{2}.
\]
Towards $II$, note that for $\abs{n-m} > 2\abs{m}$, $\abs{n} \ge \abs{n-m} - \abs{m} \ge \abs{m}$, and hence $\abs{n-m} \le \abs{n}+\abs{m} \le 2\abs{n}$ yielding the estimate
\[
  \frac{\w{n-m}^{\abs{t}}}{\w{n}^{\abs{t}}}\mathbf{1}_{\setd{\abs{n-m} > 2\abs{m}}} \le 2.
\]
Hence by the Cauchy-Schwarz inequality
\[
  II \le \sum_{n\in\Z} \p*{\sum_{m\in\Z} \frac{4}{\w{m}^{2}}}
  \sum_{m\in \Z} \p*{\frac{\abs{a_{n-m}}}{\w{n-m}^{\abs{t}}}\w{m}\abs{b_{m}}}^{2}
  \le \sum_{m\in\Z} \frac{4}{\w{m}^{2}} \n{a}_{t,2}^{2}\n{b}_{1,2}^{2}.
\]
This proves item (i). Towards (ii), write
\[
  \n{a*b}_{s,p}^{p} \le \sum_{n\in\Z} \p*{ \sum_{m\in\Z} \frac{\w{n-m}^{\abs{s}}}{\w{n}^{\abs{s}}\w{m}^{\abs{s}}}\frac{\abs{a_{n-m}}}{\w{n-m}^{\abs{s}}}\w{m}^{\abs{s}}\abs{b_{m}} }^{p}.
\]
Use that $\w{n-m} \le \w{n}\w{m}$ and Young's inequality to conclude
\[
  \n{a*b}_{s,p} \le \p*{\sum_{m\in\Z} \w{m}^{\abs{s}} \abs{b_{m}}}\n{a}_{s,p}.
\]
By the Cauchy-Schwarz inequality we get
\[
  \sum_{m\in\Z} \w{m}^{\abs{s}} \abs{b_{m}} \le 
  \p*{ \sum_{m} \w{m}^{2(2-\abs{t})}\abs{b_{m}}^{2} }^{1/2}
  \p*{ \sum_{m} \frac{1}{\w{m}^{(2-\abs{t}-\abs{s})2}} }^{1/2}.
\]
Since $-3/2-s < t \le 0$, it follows that $2-\abs{t}-\abs{s} > 1/2$ and therefore
\[
  C_{s,t} \defl \p*{\sum_{m\in\Z} \frac{1}{\w{m}^{(2-\abs{t}-\abs{s})2}}}^{1/2} < \infty
\]
implying the claimed result.\qed
\end{proof}

Finally, we prove the following result on embeddings of sequence spaces.

\begin{lem}
\label{ell-embedding}
Suppose $-1/2 \le s \le 0$ and $2\le p < \infty$, then $\ell_{\C}^{s,p}\opento \ell_{\C}^{t,2}$ for any $t < s-1/2+1/p$. In particular, $\ell_{\C}^{s,p} \opento \ell_{\C}^{-1+1/2p,2}$ for any $-1/2\le s\le 0$.\fish
\end{lem}

\begin{proof}
For $p=2$ the claim clearly holds. So it remains to consider the case $2 < p < \infty$.
Note that for $s$ and $p$ in the given ranges, $s-1/2+1/p \le 0$. Thus $t < 0$ and $\abs{t}-\abs{s} > 1/2 - 1/p > 0$. Let $(a_{m})_{m\in\Z}\in \ell_{\C}^{s,p}$. By Hölder's inequality with $p/2$, $p''/2$ where $1/p+1/p'' = 1/2$
\begin{align*}
  \sum_{n\in\Z} \frac{\abs{a_{n}}^{2}}{\w{n}^{2\abs{t}}}
  \le \sum_{n\in\Z} \frac{\abs{a_{n}}^{2}}{\w{n}^{2\abs{s}}}\frac{1}{\w{n}^{2(\abs{t}-\abs{s})}}
  \le \p*{ \sum_{n\in\Z} \frac{\abs{a_{n}}^{p}}{\w{n}^{\abs{s}p}} }^{2/p}
  \p*{ \sum_{n\in\Z} \frac{1}{\w{n}^{(\abs{t}-\abs{s})p''}} }^{2/p''}.
\end{align*}
Since $(\abs{t}-\abs{s})p'' > (1/2-1/p)(1/2-1/p)^{-1} = 1$, one has $\n{a}_{t,2}\le C\n{a}_{s,p}$ with $C = C_{t,s,p} < \infty$.\qed
\end{proof}

\section{Schrödinger Operators}
\label{a:hill-op}

In this appendix we review definitions and properties of Schrödinger operators $-\partial_{x}^{2}+q$ with a singular potential $q$ used in Section~\ref{section3} -- see e.g. \cite{Kappeler:2001bi}.

\emph{Boundary conditions.} Denote by $H^{1}_{\C}[0,1] = H^{1}([0,1],\C)$ the Sobolev space of functions $f\colon [0,1]\to \C$ which together with their distributional derivative $\partial_{x}f$ are in $L^{2}_{\C}[0,1]$. On $H^{1}_{\C}[0,1]$ we define the following three boundary conditions (bc),
\[
  (\per+)\quad f(1) = f(0);\quad
  (\per-)\quad  f(1) = -f(0);\quad
  (\dir)\quad   f(1) = f(0) = 0.
\]
The corresponding subspaces of $H^{1}_{\C}[0,1]$ are defined by
\[
  H_{bc}^{1} = \setdef{f\in H^{1}_{\C}[0,1]}{f\text{ satisfies (bc)}},
\]
and their duals are denoted by $H_{bc}^{-1} \defl (H_{bc}^{1})'$. Note that $H_{\per+}^{1}$ can be canonically identified with the Sobolev space $H^{1}(\R/\Z,\C)$ of 1-periodic functions $f\colon \R\to \C$ which together with their distributional derivative are in $L^{2}_{\loc}(\R,\C)$. Analogously, $H_{\per-}^{1}$ can be identified with the subspace of $H^{1}(\R/2\Z,\C)$ consisting of functions $f\colon \R\to \C$ with $f,\partial_{x}f\in L_{\loc}^{2}(\R,\C)$ satisfying $f(x+1) = -f(x)$ for all $x\in \R$. In the sequel we will not distinguish these pairs of spaces.
Furthermore, note that $H_{\dir}^{1}$ is a subspace of $H^{1}_{\per+}$ as well as of $H^{1}_{\per-}$. Denote by $\lin{\cdot,\cdot}_{bc}$ the extension of the $L^{2}$-inner product $\spi{f,g} = \int_{0}^{1} f\ob{g}\,\dx$ to a sesquilinear pairing of $H_{bc}^{-1}$ and $H_{bc}^{1}$. Finally, we record that the multiplication
\begin{equation}
  \label{H1-mul}
  H^{1}_{bc}\times H^{1}_{bc}\to H^{1}_{\per+},\qquad (f,g) \mapsto fg,
\end{equation}
and the complex conjugation $H_{bc}^{1}\to H_{bc}^{1}$, $f\mapsto \ob{f}$ are bounded operators.

\emph{Multiplication operators.}
For $q\in H_{\per+}^{-1}$ define the operator $V_{bc}$ of multiplication by $q$, $V_{bc}\colon H^{1}_{bc}\to H^{-1}_{bc}$ as follows: for any $f\in H_{bc}^{1}$, $V_{bc}f$ is the element in $H_{bc}^{-1}$ given by
\[
  \lin{V_{bc}f,g}_{bc} \defl \lin{q,\ob{f}g}_{\per+},\qquad g\in H_{bc}^{1}.
\]
In view of~\eqref{H1-mul}, $V_{bc}$ is a well defined bounded linear operator.

\begin{lem}
\label{V-dir-V-per}
Assume that $q\in H^{-1}_{\per+}$. Then for any $g\in H_{\dir}^{1}$
\[
  (V_{\per\pm }g)\big|_{H_{\dir}^{1}} = V_{\dir}g.\fish
\]
\end{lem}
\begin{proof}
Since any $h\in H_{\dir}^{1}$ is also in $H^{1}_{\per+}$, the definitions of $V_{\per+}$ and $V_{\dir}$ imply
\[
  \lin{V_{\per+}g,h}_{\per+} = \lin{q,\ob{g}h}_{\per+} = \lin{V_{\dir}g,h}_{\dir},
\]
which gives $(V_{\per+}g)\big|_{H_{\dir}^{1}} = V_{\dir}g$. Similarly,  one sees that $V_{\per-}g\big|_{H_{\dir}^{1}} = V_{\dir}g$.~\qed
\end{proof}

It is convenient to introduce also the space $H_{\per+}^{1}\oplus H_{\per-}^{1}$ and define the multiplication operator $V$ of multiplication by $q$
\[
  V\colon H_{\per+}^{1}\oplus H_{\per-}^{1}\to H_{\per+}^{-1}\oplus H_{\per-}^{-1},\qquad
  (f,g) \mapsto (V_{\per+}f,V_{\per-}g).
\]
We note that $H_{\per+}^{1}\oplus H_{\per-}^{1}$ can be canonically identified with $H^{1}(\R/2\Z,\C)$,
\[
  H^{1}(\R/2\Z,\C) \to H_{\per+}^{1}\oplus H_{\per-}^{1},\qquad f\mapsto (f^{+},f^{-}),
\]
where $f^{+}(x) = \frac{1}{2}(f(x) + f(x+1))$ and $f^{-}(x) = \frac{1}{2}(f(x) - f(x+1))$.
Its dual is denoted by $H^{-1}(\R/2\Z,\C)$.

\emph{Fourier basis.} The spaces $H_{\per\pm}^{1}$, $H^{1}(\R/2\Z,\C)$ and $H_{\dir}^{1}$ and their duals admit the following standard Fourier basis.

\underline{Basis for $H_{\per+}^{1}$, $H_{\per+}^{-1}$.} Any element $f\in H^{1}_{\per+}$ [$H^{-1}_{\per+}$] can be represented as $f = \sum_{m\in\Z} f_{m}e_{m}$ where $(f_{m})_{m\in\Z}\in \ell_{\C}^{1,2}$ [$\ell_{\C}^{-1,2}$] and
\[
  f_{2m} = \lin{f,e_{2m}}_{\per+},\qquad f_{2m+1} = 0,\qquad \forall m\in \Z.
\]
Furthermore, for any $q = \sum_{m\in\Z} q_{m}e_{m}\in H^{-1}_{\per+}$,
\[
  V_{\per+}f = \sum_{n\in\Z}\p*{ \sum_{m\in \Z} q_{n-m}f_{m} }e_{n} \in H_{\per+}^{-1}.
\]
Note that by Lemma~\ref{convolution}, $(\sum_{m\in \Z} q_{n-m}f_{m} )_{n\in\Z}$ is in $\ell_{\C}^{-1,2}$.

\underline{Basis for $H_{\per-}^{1}$, $H_{\per-}^{-1}$.} Any element $f\in H^{1}_{\per-}$ [$H^{-1}_{\per-}$] can be represented as $f = \sum_{m\in\Z} f_{m}e_{m}$ where $(f_{m})_{m\in\Z}\in \ell_{\C}^{1,2}$ [$\ell_{\C}^{-1,2}$] and
\[
  f_{2m+1} = \lin{f,e_{2m+1}}_{\per-},\qquad f_{2m} = 0,\qquad \forall m\in \Z.
\]
Similarly, for any $q = \sum_{m\in\Z} q_{m}e_{m}\in H^{-1}_{\per+}$,
\[
  V_{\per-}f = \sum_{n\in\Z}\p*{ \sum_{m\in \Z} q_{n-m}f_{m} }e_{n} \in H_{\per-}^{-1}.
\]

\underline{Basis for $H^{1}(\R/2\Z,\C)$, $H^{-1}(\R/2\Z,\C)$.} Any element $f\in H^{1}(\R/2\Z,\C)$ [$H^{-1}(\R/2\Z,\C)$] can be represented as $f = \sum_{m\in\Z} f_{m}e_{m}$ where $f_{m} = \spii{f,e_{m}}$. Here $\spii{\cdot,\cdot}$ denotes the normalized $L^{2}$-inner product on $[0,2]$ extended to a sesquilinear pairing between $H^{1}(\R/2Z,\R)$ and its dual. In particular, for $f\in H^{1}(\R/2\Z,\C)$, and $m\in\Z$,
\[
  \spii{f,e_{m}} = \frac{1}{2}\int_{0}^{2} f(x)\e^{-\ii 2m\pi x}\,\dx.
\]
For any $q = \sum_{m\in\Z} q_{m}e_{m}\in H_{\per+}^{-1}\opento H^{1}(\R/2\Z,\C)$
\[
  Vf = \sum_{n\in\Z} \p*{ \sum_{m\in\Z} q_{n-m} f_{m} }e_{n}\in H^{-1}(\R/2\Z,\C).
\]

\underline{Basis for $H_{\dir}^{1}$, $H_{\dir}^{-1}$:} Any element $f\in H_{\dir}^{1}$ can be represented as
\[
  f(x) = \sum_{m\ge 1} \spi{f,\sqrt{2}s_{m}}\sqrt{2}s_{m}(x) = 
  \sum_{m\in \Z} f_{m}^{\sin}s_{m}(x),\qquad s_{m}(x) = \sin(m\pi x),
\]
where $f_{m}^{\sin} = \int_{0}^{1} f(x)s_{m}(x)\,\dx$. For any element $g\in H_{\dir}^{-1}$ one gets by duality
\[
  g = \sum_{m\in\Z} g_{m}^{\sin} s_{m},\qquad g_{m}^{\sin} = \lin{g,s_{m}}_{\dir}.
\]
One verifies that $g_{-m}^{\sin} = -g_{m}^{\sin}$ for all $m\in\Z$ and $\sum_{m\in \Z} \w{m}^{-2}\abs{g_{m}^{\sin}}^{2} < \infty$.
For any $q\in H_{0,\C}^{-1}$ with $\n{q}_{t,2} < \infty$ and $-1 < t < -1/2$, we need to expand for a given $f\in H_{\dir}^{1}$, $V_{\dir}f\in H_{\dir}^{-1}$ in its sine series $\sum_{m\in \Z} (V_{\dir}f)_{m}^{\sin} s_{m}$ where by the definition of $V_{\dir}$
\[
  (V_{\dir}f)_{m}^{\sin} = \lin{V_{\dir}f,s_{m}}_{\dir} = 
  \lin{q,\ob{f}s_{m}}_{\per+} = \sum_{n\in\Z} f_{n}^{\sin}\lin{q,s_{n}s_{m}}_{\per+}.
\]
Using that $f_{-n}^{\sin} = -f_{n}^{\sin}$ for any $n\in\Z$ and that
\[
  s_{m}(x)s_{n}(x) = \frac{1}{2}(\cos((m-n)\pi x) - \cos((m+n)\pi x))
\]
it follows that
\[
  \sum_{m-n\text{ even}} f_{n}^{\sin} \lin{q,s_{n}s_{m}}_{\per+} = 
  \sum_{m-n\text{ even}} f_{n}^{\sin} \lin{q,\cos((m-n)\pi x)}_{\per+}.
\]
Note that $\lin{q,\cos((m-n)\pi x)}_{\per+}$ is well defined as $\cos((m-n)\pi x)\in H_{\per+}^{1}$ if $m-n$ is even. If $m-n$ is odd, we decompose the difference of the cosines in $H_{\per+}^{1}$ as follows
\begin{align*}
  &\cos((m-n)\pi x) - \cos((m+n)\pi x) \\
  &\qquad = 
  \p*{\cos((m-n)\pi x) - \cos(\pi x)} - \p*{\cos((m+n)\pi x) - \cos(\pi x)}
\end{align*}
and then obtain, using again that $f_{-n}^{\sin} = -f_{n}^{\sin}$ for all $n\in\Z$,
\[
  \sum_{m-n\text{ odd}} f_{n}^{\sin} \lin{q,s_{n}s_{m}}_{\per+} = 
  \sum_{m-n\text{ odd}} f_{n}^{\sin} \lin{q,\cos((m-n)\pi x) - \cos(\pi x)}_{\per+}.
\]
Altogether we have shown that
\[
  V_{\dir}f = \sum_{m\in \Z} \p*{\sum_{n\in\Z} q_{m-n}^{\cos}f_{n}^{\sin}} s_{m},
\]
where
\begin{equation}
  \label{q-cos-coeff}
  q_{k}^{\cos} = 
  \begin{cases}
  \lin{q,\cos(k\pi x)}_{\per+}, & \text{if }k\in\Z\text{ even},\\
  \lin{q,\cos(k\pi x)-\cos(\pi x)}_{\per+}, & \text{if }k\in\Z\text{ odd}.
  \end{cases}
\end{equation}
Taking into account that $\n{q}_{t,2} < \infty$ with $-1 < t < -1/2$, one argues as in \cite[Proposition 3.4]{Kappeler:2001bi}, using duality and interpolation, that
\begin{equation}
  \label{H-dir-per-est}
  \p*{ \sum_{m\in\Z} \w{m}^{2t}\abs{q_{m}^{\cos}}^{2} }^{1/2} \le C_{t}\n{q}_{t,2}.
\end{equation}

\emph{Schrödinger operators with singular potentials.}
For any $q\in H_{0,\C}^{-1}$ denote by $L(q)$ the unbounded operator $-\partial_{x}^{2} + V$ acting on $H^{-1}(\R/2\Z,\C)$ with domain $H^{1}(\R/2\Z,\C)$. As $H^{1}(\R/2\Z,\C) = H^{1}_{\per+}\oplus H^{1}_{\per-}$ and $V = V_{\per+}\oplus V_{\per-}$ the operator $L(q)$ leaves the spaces $H_{\per\pm}^{1}$ invariant and $L(q) = L_{\per+}(q)\oplus L_{\per-}(q)$ with $L_{\per\pm}(q) = -\partial_{x}^{2} + V_{\per\pm}$. Hence the spectrum $\spec(L(q))$ of $L(q)$, also referred to as spectrum of $q$, is the union $\spec(L_{\per+}(q))\cup \spec(L_{\per-}(q))$ of the spectra $\spec(L_{\per\pm}(q))$ of $L_{\per\pm}(q)$. 
It is known to be discrete and to consist of complex eigenvalues which, when counted with multiplicities and ordered lexicographically, satisfy
\[
  \lm_{0}^{+} \lex \lm_{1}^{-} \lex \lm_{1}^{+} \lex \dotsb,\qquad
  \lm_{n}^{\pm} = n^{2}\pi^{2} + n\ell_{n}^{2},
\]
-- see e.g. \cite{Kappeler:2003vh}. For any $q\in H_{0,\C}^{-1}$ there exists $N\ge 1$ so that
\[
  \abs{\lm-n^{2}\pi^{2}} \le n/2,\quad n\ge N,\qquad
  \abs{\lm_{n}^{+}} \le (N-1)^{2}\pi^{2}+N/2,\quad n < N,
\]
where $N$ can be chosen locally uniformly in $q$. Since for $q=0$ and $n\ge 0$, $\Dl(\lm_{2n}^{+}(0),0) = 2$ and $\Dl(\lm_{2n+1}^{+}(0),0) = -2$, all $\lm_{2n}^{+}(0)$ are 1-periodic and all $\lm_{2n+1}^{+}(0)$ are 1-antiperiodic eigenvalues of $q = 0$. By considering the compact interval $[0,q] = \setdef{tq}{0\le t\le 1}\subset H_{0,\C}^{-1}$ it then follows after increasing $N$, if necessary, that for any $n\ge N$
\begin{equation}
  \label{ev-per-antiper}
  \lm_{n}^{+}(q),\lm_{n}^{-}(q) \in \spec(L_{\per+}(q)),\;[\spec(L_{\per-}(q))]
  \quad \text{if }n\text{ even [odd]}.
\end{equation}

For any $q\in H_{0,\C}^{-1}$ and $n\ge N$ the following Riesz projectors are thus well defined on $H^{-1}(\R/2\Z,\C)$
\[
  P_{n,q} \defl \frac{1}{2\pi\ii}\int_{\abs{\lm - n^{2}\pi^{2 }} = n} (\lm-L(q))^{-1}\,\dlm.
\]
For $n\ge N$ even [odd], the range of $P_{n,q}$ is contained in $H_{\per+}^{1}$ [$H_{\per-}^{1}$]. If $q = 0$, we write $P_{n}$ for $P_{n,0}$.

Similarly, for any $q\in H_{0,\C}^{-1}$ denote by $L_{\dir}(q)$ the unbounded operator $-\partial_{x}^{2}+ V_{\dir}$ acting on $H_{\dir}^{-1}$ with domain $H_{\dir}^{1}$. Its spectrum $\spec(L_{\dir}(q))$ is known to be discrete and to consist of complex eigenvalues which, when counted with multiplicities and ordered lexicographically, satisfy
\[
  \mu_{1}\lex \mu_{2}\lex \dotsb,\qquad \mu_{n} = n^{2}\pi^{2} + n\ell_{n}^{2},
\]
-- see e.g. \cite{Kappeler:2003vh}. By increasing $N$ as chosen above, if necessary, we can thus assume that
\[
  \abs{\mu_{n}-n^{2}\pi^{2}} < n/2,\quad n\ge N,\qquad
  \abs{\mu_{n}} \le (N-1)^{2}\pi^{2} + N/2,\quad n\le N.
\]
In particular, for any $n\ge N$, $\mu_{n}$ is simple and the corresponding Riesz projector
\[
  \Pi_{n,q} \defl \frac{1}{2\pi\ii} \int_{\abs{\lm-n^{2}\pi^{2}}=n} (\lm-L_{\dir}(q))^{-1}\,\dlm
\]
is well defined on $H_{\dir}^{-1}$. If $q=0$, we write $\Pi_{n}$ for $\Pi_{n,0}$.

\smallskip

\emph{Regularity of solutions.}
In Section~\ref{section3} we consider solutions $f$ of the equation $(L(q)-\lm)f = g$ in $\Ls_{\star,\C}^{s,p}$ and need to know their regularity.

\begin{lem}
\label{regularity-inhomogeneous}
For any $q\in \Ls_{0,\C}^{s,p}$ with $-1/2\le s \le 0$ and $2\le p < \infty$ the following holds: For any $g\in \Ls_{\star,\C}^{s,p}$ and any $\lm\in\C$, a solution $f\in H^{1}(\R/2\Z,\C)$ of the inhomogeneous equation $(L(q)-\lm)f = g$ is an element in $\Ls_{\star,\C}^{s+2,p}$.\fish
\end{lem}

\begin{proof}
Write $(L-\lm)f = g$ as $A_{\lm}f = Vf - g$ where $A_{\lm} = \partial_{x}^{2} + \lm$. Since $q\in \Ls_{0,\C}^{s,p}$, Lemma~\ref{ell-embedding} implies that $q$ and $g$ are in $\Ls_{\star,\C}^{r,2}$ with $r = s-1/2+1/2p$. By Lemma~\ref{convolution} (i), $Vf\in \Ls_{\star,\C}^{r,2}$ and hence $A_{\lm}f = Vf-g\in\Ls_{\star,\C}^{r,2}$ implying that $f\in \Ls_{\star,\C}^{r+2,2}$. Since $-s-3/2 \le -1 < r \le 0$, Lemma~\ref{convolution} (ii) applies. Therefore, $Vf\in \Ls_{\star,\C}^{s,p}$ and using the equation $A_{\lm}f = Vf - g$ once more one gets $f\in \Ls_{\star,\C}^{s+2,p}$ as claimed.\qed
\end{proof}

For any $q\in \Ls_{0,\C}^{s,p}$ with $-1/2\le s \le 0$, $2\le p < \infty$, and $n\ge n_{s,p}$ as in Lemma~\ref{Sn-roots}, introduce
\[
  E_{n} \equiv E_{n}(q) \defl 
  \begin{cases}
  \Null(L(q)-\lm_{n}^{+})\oplus \Null(L(q)-\lm_{n}^{-}), & \lm_{n}^{+}\neq \lm_{n}^{-},\\
  \Null(L(q)-\lm_{n}^{+})^{2}, & \lm_{n}^{+} = \lm_{n}^{-}.
  \end{cases}
\]
Then $E_{n}$ is a two-dimensional subspace of $H_{\per+}^{1}$ [$H_{\per-}^{1}$] if $n$ is even [odd]. The following result shows that elements in $E_{n}$ are more regular.

\begin{lem}
\label{regularity-En}
For any $q\in \Ls_{0,\C}^{s,p}$ with $-1/2\le s \le 0$, $2\le p <\infty$ and for any $n\ge n_{s,p}$ $E_{n}(q)\subset\Ls_{\star,\C}^{s+2,p}\cap H_{\per+}^{1}$ [$\Ls_{\star,\C}^{s+2,p}\cap H_{\per-}^{1}$] if $n$ is even [odd].\fish
\end{lem}

\begin{proof}
By Lemma~\ref{regularity-inhomogeneous} with $g=0$, any eigenfunction $f$ of an eigenvalue $\lm$ of $L(q)$ is in $\Ls_{\star,\C}^{s+2,p}$. Hence if $\lm_{n}^{+} \neq \lm_{n}^{-}$ or if $\lm_{n}^{+} = \lm_{n}^{-}$ and has geometric multiplicity two, then $E_{n}\subset\Ls_{\star,\C}^{s+2,p}$. Finally, if $\lm_{n}^{+} = \lm_{n}^{-}$ is a double eigenvalue of geometric multiplicity $1$ and $g$ is an eigenfunction corresponding to $\lm_{n}^{+}$, there exists an element $f\in H^{1}(\R/2\Z,\C)$ so that $(L-\lm_{n}^{+})f = g$. As $g$ is an eigenfunction it is in $\Ls_{\star,\C}^{s+2,p}$ and by applying Lemma~\ref{regularity-inhomogeneous} once more, it follows that $f\in \Ls_{\star,\C}^{s+2,p}$. Clearly, $E_{n} = \mathrm{span}(g,f)$ and hence $E_{n}\subset \Ls_{\star,\C}^{s+2,p}$ also in this case. By the consideration above one knows that $\lm_{n}^{\pm}$ are 1-periodic [1-antiperiodic] eigenvalues of $q$ if $n$ is even [odd]. Hence $E_{n}\subset \Ls_{\star,\C}^{s+2,p}\cap H_{\per+}^{1}$ [$\Ls_{\star,\C}^{s+2,p}\cap H_{\per-}^{1}$] if $n$ is even [odd] as claimed.\qed
\end{proof}

\emph{Estimates for projectors.}
The projectors $P_{n,q}$ [$\Pi_{n,q}$] with $n\ge N$ are defined on $H^{-1}(\R/2\Z,\C)$ [$H_{\dir}^{-1}$] and have range in $H^{1}(\R/2\Z,\C)$ [$H^{1}_{\dir}$]. The following result concerns estimates for the restriction of $P_{n,q}$ [$\Pi_{n,q}$] to $L^{2} = L^{2}([0,2],\C)$ [$L^{2}(\Ic) = L^{2}([0,1],\C)$] needed in Section~\ref{section3} for getting the asymptotics of $\mu_{n}-\tau_{n}$.

\begin{lem}
\label{proj-bound}
Assume that $q\in H_{0,\C}^{-1}$ with $\n{q}_{t,2} < \infty$ and $-1 < t < -1/2$. Then there exist constants $C_{t} > 0$ (only depending on $t$) and $N'\ge N$ (with $N$ as above) so that
\begin{renum}
\item
\[
  \n{P_{n,q}-P_{n}}_{L^{2}\to L^{\infty}}
  \le C_{t} \frac{(\log n)^{2}}{n^{1-\abs{t}}}\n{q}_{t,2}
\]
\item
\[
  \n{\Pi_{n,q}-\Pi_{n}}_{L^{2}(\Ic)\to L^{\infty}(\Ic)}
  \le C_{t} \frac{(\log n)^{2}}{n^{1-\abs{t}}}\n{q}_{t,2}
\]
\end{renum}
The constant $N'$ can be chosen locally uniformly in $q$.\fish
\end{lem}

\begin{rem}
We will apply Lemma~\ref{proj-bound} for potentials $q\in\Ls_{0,\C}^{s,p}$ with $-1/2\le s\le 0$ and $2\le p < \infty$ using the fact that by the Sobolev embedding theorem there exists $-1 < t < -1/2$ so that $\Ls_{0,\C}^{s,p} \opento \Ls_{0,\C}^{t,2}$.
\end{rem}

\begin{proof}
The proof follows the line of arguments of~\cite{Kappeler:2001bi}. Let us first prove (i). Suppose $\abs{\lm-m^{2}\pi^{2}} \ge \w{m}$ for a given $m\ge 0$. Then we may write with $L\equiv L(q)$
\[
  \lm-L = D_{\lm}^{1/2}(I_{\lm} - S_{\lm})D_{\lm}^{1/2},
\]
where the operators $D_{\lm}$, $I_{\lm}$ and $S_{\lm}$ are defined with respect to the $L^{2}$-orthonormal basis $(e_{m})_{m\in\Z}$ of $L^{2}([0,2],\C)$ by ($l,m\in\Z$)
\begin{align*}
  &D_{\lm}^{1/2}(l,m) = \abs{\lm-l^{2}\pi^{2}}^{1/2}\dl_{lm},
  \qquad
  I_{\lm}(l,m) = \frac{\lm - l^{2}\pi^{2}}{\abs{\lm-l^{2}\pi^{2}}}\dl_{lm},\\
  &S_{\lm}(l,m) = \frac{q_{l-m}}{\abs{\lm-l^{2}\pi^{2}}^{1/2}\abs{\lm-m^{2}\pi^{2}}^{1/2}}.
\end{align*}
Then there exists $N'$ with $N'\ge N$ so that any $\lm\in\C$ with $\abs{\lm-n^{2}\pi^{2}} = n$ and $n\ge N'$ is in the resolvent set of $L$. Hence $I_{\lm}-S_{\lm} = D_{\lm}^{-1/2}(\lm-L)D_{\lm}^{-1/2}\colon L^{2}([0,2],\C)\to L^{2}([0,2],\C)$ is a boundedly invertible operator and
\[
  P_{n,q}-P_{n}
   = \frac{1}{2\pi\ii}\int_{\abs{\lm-n^{2}\pi^{2}} = n} 
   D_{\lm}^{-1/2}(I_{\lm} - S_{\lm})^{-1}S_{\lm}I_{\lm}^{-1}D_{\lm}^{-1/2}\,\dlm.
\]
For any $f\in L^{2}([0,2],\C)$ and $\lm\in \C$ with $\abs{\lm-n^{2}\pi^{2}} = n$, $n\ge N'$, one gets by Cauchy-Schwarz and Lemma~\ref{hilbert-sum},
\begin{align*}
    \n{D_{\lm}^{-1/2}f}_{L^{\infty}}
    \le \sum_{m\in\Z} \frac{\abs{f_{m}}}{\abs{\lm-m^{2}\pi^{2}}^{1/2}}
   &\le \left(\frac{2}{n} + \sum_{\abs{m}\neq n} \frac{1}{\abs{n^{2}-m^{2}}}\right)^{1/2}\n{f}_{L^{2}}\\
   &\le 4 \frac{(\log n)^{1/2}}{n^{1/2}}\n{f}_{L^{2}}.
\end{align*}
Similarly, one has
\[
  \n{D_{\lm}^{-1/2}f}_{L^{2}}
    \le \p*{\sum_{m\in\Z} \frac{\abs{f_{m}}^{2}}{\abs{\lm-m^{2}\pi^{2}}}}^{1/2}
    \le \frac{1}{n^{1/2}}\n{f}_{L^{2}},
\]
and by~\cite[Lemma 1.2]{Kappeler:2001bi} using that $\n{q}_{t,2} < \infty$ with $-1 < t < -1/2$, $\abs{\lm-n^{2}\pi^{2}} = n$ and $n\ge N'$,
\begin{equation}
  \label{Slm-est}
  \n{S_{\lm}}_{L^{2}\to L^{2}} \le C \frac{\log n}{n^{1-\abs{t}}}\n{q}_{t,2}.
\end{equation}
By increasing $N'$, if necessary, it follows that $\n{S_{\lm}}_{L^{2}\to L^{2}}\le 1/2$ and hence
\[
  \n{P_{n,q}-P_{n}}_{L^{2}\to L^{\infty}}
   \le C_{t}\sup_{\abs{\lm-n^{2}\pi^{2}}=n} (\log n)^{1/2}\n{S_{\lm}}_{L^{2}\to L^{2}},\qquad
   \forall n\ge N'.
\]
Combined with~\eqref{Slm-est}, this yields the estimate of (i). 

(ii) The estimate for the Dirichlet projectors are proved in a similar way: Suppose that $\abs{\lm-m^{2}\pi^{2}} \ge \w{m}$ for a given $m\ge 0$. Then with $L_{\dir}\equiv L_{\dir}(q)$
\[
  \lm - L_{\dir} = \tilde D_{\lm}^{1/2}(\tilde I_{\lm} - \tilde S_{\lm})\tilde D_{\lm}^{1/2}
\]
where the operators $\tilde D_{\lm}$, $\tilde I_{\lm}$, $\tilde S_{\lm}$ are defined in terms of their coefficients with respect to the sine basis $s_{n}$, $n\ge 1$. For any $l,m\in\Z$,
\begin{align*}
  &\tilde D_{\lm}^{1/2}(l,m) = \abs{\lm-l^{2}\pi^{2}}^{1/2}\dl_{lm},
  &&\tilde I_{\lm}(l,m) = \frac{\lm-l^{2}\pi^{2}}{\abs{\lm-l^{2}\pi^{2}}}\dl_{lm},\\
  &\tilde S_{\lm}(l,m) = \frac{q_{l-m}^{\cos}}{\abs{\lm-l^{2}\pi^{2}}^{1/2}\abs{\lm-m^{2}\pi^{2}}^{1/2}},
\end{align*}
with $q_{l-m}^{\cos}$ given by \eqref{q-cos-coeff}. By increasing $N'$, if necessary, we can assure that any $\lm\in\C$ with $\abs{\lm-n^{2}\pi^{2}} = n$ and $n\ge N'$ is in the resolvent set of $L_{\dir}$. Hence
\[
  (\tilde I_{\lm} - \tilde S_{\lm}) = \tilde D_{\lm}^{-1/2} (\lm - L_{\dir}) \tilde D_{\lm}^{-1/2}
  \colon L^{2}([0,1],\C)\to L^{2}([0,1],\C)
\]
is a boundedly invertible operator and
\[
  \Pi_{n,q} - \Pi_{n} = \frac{1}{2\pi \ii} \int_{\abs{\lm-n^{2}\pi^{2}} = n} 
  \tilde D_{\lm}^{-1/2}(\tilde I_{\lm} - \tilde S_{\lm})^{-1}\tilde S_{\lm}\tilde I_{\lm}^{-1}\tilde D_{\lm}^{-1/2}\,\dlm.
\]
Using the same arguments as in the proof of (i) together with \eqref{H-dir-per-est}, the estimate (ii) then follows.

Going through the arguments of the proofs of (i) and (ii) one sees that $N'$ can be chosen locally uniformly in $q$ and $C_{t}$ independently of $q$.\qed
\end{proof}


\small

\bibliography{library}

\begin{thebibliography}{21}
\providecommand{\natexlab}[1]{#1}
\providecommand{\url}[1]{#1}
\providecommand{\urlprefix}{}
\expandafter\ifx\csname urlstyle\endcsname\relax
  \providecommand{\doi}[1]{DOI~\discretionary{}{}{}#1}\else
  \providecommand{\doi}{DOI~\discretionary{}{}{}\begingroup
  \urlstyle{rm}\Url}\fi

\bibitem[{Dis(????)}]{DispWiki}
\emph{Dispersive Wiki}.
\newblock \urlprefix\url{http://wiki.math.toronto.edu/DispersiveWiki/}

\bibitem[{Bikbaev \& Kuksin(1993)}]{Bikbaev:1993jl}
Bikbaev, R.F., Kuksin, S.B.: \emph{{On the parametrization of finite-gap
  solutions by frequency and wavenumber vectors and a theorem of I.
  Krichever}}.
\newblock Lett. Math. Phys. \textbf{28}(2), 115--122, 1993.

\bibitem[{Bobenko \& Kuksin(1991)}]{Bobenko:1991fv}
Bobenko, A.I., Kuksin, S.B.: \emph{{Finite-gap periodic solutions of the KdV
  equation are nondegenerate}}.
\newblock Phys. Lett. A \textbf{161}(3), 274--276, 1991.

\bibitem[{Djakov \& Mityagin(2006)}]{Djakov:2006ba}
Djakov, P., Mityagin, B.: \emph{{Instability zones of periodic 1-dimensional
  Schr{\"o}dinger and Dirac operators}}.
\newblock Russian Math. Surveys \textbf{61}, 663--766, 2006.

\bibitem[{Djakov \& Mityagin(2009)}]{Djakov:2009fx}
Djakov, P., Mityagin, B.: \emph{{Spectral gaps of Schr\"odinger operators with
  periodic singular potentials}}.
\newblock Dyn. Partial Differ. Equ. \textbf{6}(2), 95--165, 2009.

\bibitem[{Gr{\'e}bert \& Kappeler(2014)}]{Grebert:2014iq}
Gr{\'e}bert, B., Kappeler, T.: \emph{{The defocusing NLS equation and its
  normal form}}.
\newblock European Mathematical Society (EMS), Z{\"u}rich, 2014.

\bibitem[{Henrici \& Kappeler(2009)}]{Henrici:2009ek}
Henrici, A., Kappeler, T.: \emph{{Nekhoroshev theorem for the periodic Toda
  lattice}}.
\newblock Chaos \textbf{19}(3), 033,120--033,113, 2009.

\bibitem[{Huang \& Kuksin(2014)}]{Guan:2014jc}
Huang, G., Kuksin, S.B.: \emph{{The KdV equation under periodic boundary
  conditions and its perturbations}}.
\newblock Nonlinearity \textbf{27}(9), R61--R88, 2014.

\bibitem[{Kappeler \& Mityagin(2001)}]{Kappeler:2001hsa}
Kappeler, T., Mityagin, B.: \emph{{Estimates for periodic and Dirichlet
  eigenvalues of the Schr\"odinger operator}}.
\newblock SIAM J. Math. Anal. \textbf{33}(1), 113--152, 2001.

\bibitem[{Kappeler \& M{\"o}hr(2001)}]{Kappeler:2001bi}
Kappeler, T., M{\"o}hr, C.: \emph{{Estimates for periodic and Dirichlet
  eigenvalues of the Schr\"odinger operator with singular potentials}}.
\newblock J. Funct. Anal. \textbf{186}(1), 62--91, 2001.

\bibitem[{Kappeler et~al.(2005)Kappeler, M{\"o}hr, \&
  Topalov}]{Kappeler:2005fb}
Kappeler, T., M{\"o}hr, C., Topalov, P.: \emph{{Birkhoff coordinates for KdV on
  phase spaces of distributions}}.
\newblock Selecta Math. (N.S.) \textbf{11}(1), 37--98, 2005.

\bibitem[{Kappeler \& P{\"o}schel(2003)}]{Kappeler:2003up}
Kappeler, T., P{\"o}schel, J.: \emph{{KdV {\&} KAM}}.
\newblock Springer, Berlin, 2003.

\bibitem[{Kappeler et~al.(2008)Kappeler, Serier, \& Topalov}]{Kappeler:2008fl}
Kappeler, T., Serier, F., Topalov, P.: \emph{{On the symplectic phase space of
  KdV}}.
\newblock Proc. Amer. Math. Soc. \textbf{136}(5), 1691--1698, 2008.

\bibitem[{Kappeler \& Topalov(2003)}]{Kappeler:2003vh}
Kappeler, T., Topalov, P.: \emph{{Riccati representation for elements in
  $H^{-1}(\mathbb{T})$ and its applications}}.
\newblock Pliska Stud. Math. Bulgar. \textbf{15}, 171--188, 2003.

\bibitem[{Kappeler \& Topalov(2006)}]{Kappeler:2006fr}
Kappeler, T., Topalov, P.: \emph{{Global wellposedness of KdV in
  $H^-1(\mathbb{T},\mathbb{R})$}}.
\newblock Duke Math. J. \textbf{135}(2), 327--360, 2006.

\bibitem[{Korotyaev(2003)}]{Korotyaev:2003gp}
Korotyaev, E.: \emph{{Characterization of the spectrum of Schr\"odinger
  operators with periodic distributions}}.
\newblock Int. Math. Res. Not. \textbf{2003}(37), 2019--2031, 2003.

\bibitem[{Korotyaev \& Kuksin(2011)}]{Korotyaev:2011tw}
Korotyaev, E., Kuksin, S.B.: \emph{{KdV Hamiltonian as function of actions}}.
\newblock arXiv , 2011.

\bibitem[{Krichever(1992)}]{Krichever:1992uq}
Krichever, I.M.: \emph{{Perturbation Theory in Periodic Problems for
  Two-Dimensional Integrable Systems}}.
\newblock Sov. Sci. Rev. C. Math. Phys. \textbf{9}, 1--103, 1992.

\bibitem[{Molnar(2015)}]{Molnar:_ROURXz4}
Molnar, J.C.: \emph{{On Two-Sided Estimates for the Nonlinear Fourier Transform
  of KdV}}.
\newblock arXiv:1502.04550 , 2015.

\bibitem[{P{\"o}schel(2011)}]{Poschel:2011iua}
P{\"o}schel, J.: \emph{{Hill's potentials in weighted Sobolev spaces and their
  spectral gaps}}.
\newblock Math. Ann. \textbf{349}(2), 433--458, 2011.

\bibitem[{Savchuk \& Shkalikov(2003)}]{Savchuk:2003vl}
Savchuk, A.M., Shkalikov, A.A.: \emph{{Sturm-Liouville operators with
  distribution potentials}}.
\newblock Tr. Mosk. Math. Obs. \textbf{64}, 159--212, 2003.

\end{thebibliography}

\end{document}